\documentclass[twoside,11pt]{article}

\usepackage{jmlr2e}
\usepackage{amsmath,amssymb,amsfonts,amscd,graphics} 

\usepackage{algorithmic}
\usepackage[ruled,vlined]{algorithm2e}
\usepackage{etoolbox,verbatim,color}
\usepackage{url}
\usepackage{enumitem}
\usepackage{booktabs}
\usepackage[utf8x]{inputenc}

\newtheorem{assumption}{Assumption}
\usepackage{multicol}
\usepackage{multirow}

\newcommand{\iprod}[2]{\left\langle {#1}, {#2} \right\rangle}

\newcommand{\lrpar}[1]{\left(#1\right)}

\newcommand{\lrbrace}[1]{\left\{#1\right\}}

\newcommand{\calA}{\mathcal{A}}

\newcommand{\bbE}{\mathbb{E}}

\newcommand{\bbN}{\mathbb{N}}

\newcommand{\bbR}{\mathbb{R}}

\DeclareMathAlphabet{\mathbsf}{OT1}{cmss}{bx}{n}
\DeclareMathAlphabet{\mathssf}{OT1}{cmss}{m}{sl}

\DeclareSymbolFont{bsfletters}{OT1}{cmss}{bx}{n}  
\DeclareSymbolFont{ssfletters}{OT1}{cmss}{m}{n}
\DeclareMathSymbol{\bsfGamma}{0}{bsfletters}{'000}
\DeclareMathSymbol{\ssfGamma}{0}{ssfletters}{'000}
\DeclareMathSymbol{\bsfDelta}{0}{bsfletters}{'001}
\DeclareMathSymbol{\ssfDelta}{0}{ssfletters}{'001}
\DeclareMathSymbol{\bsfTheta}{0}{bsfletters}{'002}
\DeclareMathSymbol{\ssfTheta}{0}{ssfletters}{'002}
\DeclareMathSymbol{\bsfLambda}{0}{bsfletters}{'003}
\DeclareMathSymbol{\ssfLambda}{0}{ssfletters}{'003}
\DeclareMathSymbol{\bsfXi}{0}{bsfletters}{'004}
\DeclareMathSymbol{\ssfXi}{0}{ssfletters}{'004}
\DeclareMathSymbol{\bsfPi}{0}{bsfletters}{'005}
\DeclareMathSymbol{\ssfPi}{0}{ssfletters}{'005}
\DeclareMathSymbol{\bsfSigma}{0}{bsfletters}{'006}
\DeclareMathSymbol{\ssfSigma}{0}{ssfletters}{'006}
\DeclareMathSymbol{\bsfUpsilon}{0}{bsfletters}{'007}
\DeclareMathSymbol{\ssfUpsilon}{0}{ssfletters}{'007}
\DeclareMathSymbol{\bsfPhi}{0}{bsfletters}{'010}
\DeclareMathSymbol{\ssfPhi}{0}{ssfletters}{'010}
\DeclareMathSymbol{\bsfPsi}{0}{bsfletters}{'011}
\DeclareMathSymbol{\ssfPsi}{0}{ssfletters}{'011}
\DeclareMathSymbol{\bsfOmega}{0}{bsfletters}{'012}
\DeclareMathSymbol{\ssfOmega}{0}{ssfletters}{'012}

\DeclareMathOperator*{\argmin}{argmin} 
\DeclareMathOperator*{\dist}{dist}

\usepackage{footnote}
\usepackage{mathtools}

\allowdisplaybreaks
\newtheorem{requirement}{Condition}

\definecolor{brightpink}{rgb}{1.0, 0.0, 0.5}

\newcommand{\ngi}[1]{{{\color{brightpink} #1}}}

\newcommand{\dom}{\mathbf{dom}\,}

\newcommand{\revise}[1]{{{\color{black} #1}}} 

\newcommand{\mbfx}{\mathbf{x}}
\newcommand{\mbfr}{\mathbf{r}}
\newcommand{\mbfm}{\mathbf{m}}
\newcommand{\mbfn}{\mathbf{n}}
\newcommand{\barmx}{\bar{\mathbf{x}}}

\jmlrheading{}{}{}{25 May 2021}{0}{}{L. T. K. Hien, D. N. Phan and N. Gillis}

\ShortHeadings{Inertial Block Majorization Minimization}{Hien, Phan and Gillis}
\firstpageno{1}

\begin{document}
\title{An Inertial Block Majorization Minimization Framework for \\ 
Nonsmooth Nonconvex Optimization
\thanks{L.~T.~K.~Hien finished this work when she was at the University of Mons, Belgium. L.~T.~K.~Hien and N.~Gillis are supported by the Fonds de la Recherche Scientifique - FNRS and the Fonds Wetenschappelijk Onderzoek - Vlaanderen (FWO) under EOS project no 30468160 (SeLMA), and by the European Research Council
(ERC starting grant 679515).}}
\author{\name Le Thi Khanh Hien \email let.hien@huawei.com \\
       \addr Huawei Belgium Research Center, 3001, Leuven, Belgium
       \AND
       \name  Duy Nhat Phan \email nhatpd@hcmue.edu.vn \\
       \addr Department of Mathematics and Informatics, HCMC University of Education, Vietnam
       \AND
       \name  Nicolas Gillis \email nicolas.gillis@umons.ac.be \\
       \addr Department of Mathematics and Operational Research, University of Mons, Belgium}

\editor{}

\maketitle

\begin{abstract}
In this paper, we introduce TITAN, a novel iner\textbf{TI}al block 
majoriza\textbf{T}ion minimiz\textbf{A}tio\textbf{N} framework for nonsmooth nonconvex optimization problems. 
To the best of our knowledge, TITAN is the first framework of block-coordinate update  method  that relies on the majorization-minimization framework while embedding inertial force to each  step of the block updates. 
The inertial force is obtained via an extrapolation operator that subsumes heavy-ball and Nesterov-type accelerations for block proximal gradient methods as special cases. By choosing various surrogate functions, such as proximal, Lipschitz gradient, Bregman, quadratic, and composite surrogate functions, and by varying the extrapolation operator, TITAN produces a rich set of inertial block-coordinate update methods. 
We study sub-sequential convergence as well as global convergence for the generated sequence of TITAN. 
We illustrate the effectiveness of TITAN on two important machine learning problems, namely 
sparse non-negative matrix factorization and matrix completion. 

\end{abstract}

\section{Introduction}
\label{sec:intro}

In this paper, we consider the following nonsmooth nonconvex optimization problem
\begin{equation}
\label{model}
\begin{aligned}
    \min_x \quad   
    & F(x) := f(x_1,\ldots,x_m)+ \sum_{i=1}^m g_i(x_i) \\
     \text{ such that } \quad 
     &  x_i \in \mathcal X_i \text{ for } i \in [m] = \{1,\ldots,m\},
\end{aligned}
\end{equation}
where $\mathcal X_i \subseteq \mathbb E_i$ is a closed convex set of a finite dimensional real linear space $\mathbb E_i$, $x$ can be decomposed into $m$ blocks $ x=(x_1,\ldots,x_m)$ with $x_i\in \mathcal X_i$, \revise{ $f:\mathbb E_1\times\ldots\times\mathbb E_m\to\mathbb R$ is a lower semi-continuous function that can possibly be nonsmooth nonconvex}, and $g_i(\cdot)$ is a proper and lower semi-continuous function (possibly with extended values). We assume  $\dom g_i \cap\mathcal X_i$ is a non-empty closed set and $F$ is bounded from below. We denote $\mathcal X:=\prod_{i=1}^m \mathcal X_i$. \revise{Problem~\eqref{model} is equivalent to the following optimization problem
\begin{equation}
\label{unconstrained}
\begin{array}{ll}
\min\limits_{x\in\mathbb E} \Phi(x):= F(x) + \sum\limits_{i=1}^m \mathcal I_{\mathcal X_i}(x_i),
\end{array}
\end{equation}
where $\mathcal I_{\mathcal X_i}(\cdot)$, for $i\in [m]$, is the indicator function of $\mathcal X_i$. Hence, it makes sense to consider the optimality condition $0\in \partial \Phi(x^*)$ for Problem~\eqref{model}, that is, $x^*$ is a critical point of $\Phi$. Note that $\Phi(x)=F(x)$ when $\mathcal X_i=\mathbb E_i$. 
Throughout the paper we assume the following. 
\begin{assumption}
\label{partialF}
We have 
$$\partial \Phi(x)=\lrbrace{\partial_{x_1} (F(x) + \mathcal I_{\mathcal X_1}(x_1))}\times\ldots\times\lrbrace{\partial_{x_m} (F(x) + \mathcal I_{\mathcal X_m}(x_m))}.$$
\end{assumption}
This assumption is satisfied when $f$ is a sum of a continuously differentiable function and a block separable function, see~\citealt[Proposition 2.1]{Attouch2010}.}

\subsection{Applications} \label{intro:appl}

Some remarkable applications of Problem~\eqref{model} include nonnegative matrix factorization (see \citealt{Gillis2020}), sparse dictionary learning (see \citealt{aharon2006k,XuYin2016}), and 
``$l_p$-norm" regularized sparse regression problems with $0\leq p<1$ (see~\citealp{BLUMENSATH2009,Natarajan1995}). In this paper, we will illustrate our new proposed algorithmic framework (TITAN, Algorithm~\ref{algo:iMM_multiblock} in Section~\ref{sec:multi_iMM}) on the following two machine learning problems. 

\paragraph{Sparse Non-negative Matrix Factorization (Sparse NMF).}
We consider the following sparse NMF problem, see~\cite{Peharz2012}, 
\begin{equation}
\label{sparseNMF}
\min_{U,V} \Big \{\frac12\|M-UV\|^2: U\in \mathbb R^{\mbfm\times \mbfr}_+, V \in  \mathbb R^{\mbfr\times \mbfn}_+, \|U_{:,i}\|_0\leq s \text{ for } i \in [\revise{\mbfr}] \Big\},
\end{equation}
where $M \in \mathbb{R}^{\mbfm\times \mbfn}_+$ is a data matrix, $\mbfr$ is a given positive integer,  $U_{:,i}$ denotes the $i$-th column of $U$ and $\|U_{:,i}\|_0$ denotes the number of non-zero entries of $U_{:,i}$. Problem~\eqref{sparseNMF} is an instance of Problem~\eqref{model} with $U \in  \mathcal X_1 =\mathbb R^{\mbfm\times \mbfr}$, 
$V \in \mathcal X_2 =\mathbb R^{\mbfr\times \mbfn}$, 
$f(U,V)=\frac12\|M-UV\|^2$, 
$g_1 (\cdot)$ is the indicator function of the closed nonconvex
set $\{U: U\in \mathbb R^{\mbfm\times \mbfr}_+,  \|U_{:,i}\|_0\leq s 
\text{ for } i \in [\revise{\mbfr}]\}$,  
and $g_2 (\cdot)$ is the indicator function of the closed convex set 
$\{ V: V\in \mathbb R^{\mbfr\times \mbfn}_+\}$. 

We note that $g_1$ is nonconvex while $g_2$ is convex.

\paragraph{Matrix Completion Problem (MCP).} We consider the following MCP 
\begin{equation}
\label{MF}
    \min_{U\in\mathbb{R}^{\mbfm\times \mbfr},V\in\mathbb{R}^{\mbfr\times \mbfn}} \biggl\{ \frac{1}{2}\|\mathcal P(A - UV)\|_F^2 + \mathcal R(U,V)\biggr\},
\end{equation}
where $A\in\mathbb R^{\mathbf m \times \mathbf n}$ is a given data matrix, $\mathcal R$ is a regularization term, and $\mathcal P(Z)_{ij} = Z_{ij}$ if $A_{ij}$ is observed and is equal to $0$ otherwise. The MCP \eqref{MF} is one of the workhorse approaches in recommendation system; 
see \cite{kormat, dacrema2019we, rendle2019difficulty}. 
Other applications of the MCP include sensor network localization (\citealt{Biswas2006}), social network analysis  (\citealt{Kim2011}), and image processing (\citealt{Lui2013}). 
For $\mathcal R(U,V)$, we will use the exponential regularization (see,  e.g., \citealt{brafea}), namely $\mathcal R = \phi\circ r$, where $\phi$ and $r$ are given by
\begin{equation}
    \begin{array}{ll}
        \phi(U,V) & =  \lambda\Big(\sum_{ij}\big(1-\exp(-\theta u_{ij})\big) + \sum_{ij}\big(1-\exp(-\theta v_{ij} )\big) \Big), \\
        r(U,V) & = (r_1(U),r_2(V)) = (|U|,|V|),
    \end{array}
\end{equation}
where $u_{ij}$ is the entry of $U$ at position $(i,j)$,  
$|U|$ is component-wise absolute value of $U$,  
and 
$\lambda$ and $\theta$ are tuning parameters. 
 Problem~\eqref{MF} is an instance of Problem~\eqref{model} with $U \in \mathcal X_1 =\mathbb R^{\mbfm\times \mbfr}$, $V \in \mathcal X_2 =\mathbb R^{\mbfr\times \mbfn}$, $g_i=0$ for $i=1,2$, 
and $f(U,V)=\psi(U,V) + \phi(r(U,V))$, where $\psi(U,V): = \frac{1}{2}\|\mathcal P(A - UV)\|_F^2$ is the data-fitting term. 

We note that $\mathcal R$ is nonsmooth and the proximal mappings of the functions $U\mapsto \mathcal R(U,V)$ and $V\mapsto \mathcal R(U,V)$ do not have closed forms (see more details in Section~\ref{sec:MCP}). Hence, the subproblems of proximal alternating linearized minimization method (see~\citealt{Bolte2014}) and its inertial versions (see~\citealt{Ochs2014,Xu2013,Xu2017,Pock2016,Hien_ICML2020}) do not have closed forms when solving the MCP.

\subsection{Related works} \label{relwork}

Our new proposed algorithmic framework (TITAN, Algorithm~\ref{algo:iMM_multiblock} in Section~\ref{sec:multi_iMM}) relies on block-coordinate update methods based on  
majorization minimization, and the addition of inertial force. 
In the next two paragraphs, we briefly summarize previous works on these topics.

\paragraph{Block-coordinate update methods} 

Block coordinate descent (BCD) methods are standard approaches to solve the nonsmooth nonconvex problem~\eqref{model}. Starting with a given initial point, BCD updates one block of variables at a time while fixing the values of the other blocks. Typically, there are three main types of BCD methods: classical BCD (see \citealt{GRIPPO20001,Hildreth,Powell1973,Tseng2001}), proximal BCD (see \citealt{GRIPPO20001,Razaviyayn2013,Xu2013}), 
and proximal gradient BCD (see \citealt{Beck2013,Bolte2014,Razaviyayn2013,Tseng2009}).  
Let us briefly describe these three types of BCD methods. 
Fixing $x_j$ for $j\in \{1,\ldots,m\} \setminus\{i\}$, let us call the function $x_i \mapsto f(x)$ a block $i$ function of $f$.   
The classical BCD methods alternatively minimize the block $i$ functions of the objective.  These methods fail to converge for some nonconvex problems, see for example~\cite{Powell1973}. The proximal BCD methods improve the classical BCD methods by coupling the block $i$ objective functions with a proximal term. Considering Problem~\eqref{model} with $m=2$, the authors in \cite{Attouch2010} proved the global convergence of the generated sequence of the proximal BCD methods to a critical point of $F$, which is assumed to satisfy the Kurdyka-{\L}ojasiewicz (KL) property, see \cite{Kurdyka1998, Bolte2007}. 
The proximal gradient BCD methods minimize a standard proximal linearization of the objective function, that is, they linearize $f$, which is assumed to be smooth, and take a proximal step (which can involve Bregman divergences) on the nonsmooth part $g$. Using the KL property of $F$, \cite{Bolte2014} proved the global convergence of the proximal gradient BCD for solving Problem~\eqref{model} when each block function of $f$ is assumed to be Lipschitz smooth. When the block functions are relative smooth (\citealt{Bauschke2017,lu2018relatively}), 
\cite{ahookhosh2019multi,HienNicolas_KLNMF,Teboulle2020} 
prove the global convergence.

The BCD methods presented in the previous paragraph belong to a more general framework that was proposed in \cite{Razaviyayn2013}, and named the block successive upper-bound minimization algorithm (BSUM). BSUM for one block problem is closely related to the majorization-minimization algorithm. BSUM updates one block $i$ of $x$ by minimizing an upper-bound approximation function (also known as a majorizer, or a surrogate function) of the corresponding block $i$ objective function. BSUM recovers proximal BCD when the proximal surrogate functions are chosen, and it recovers proximal gradient BCD when the Lipschitz gradient surrogate or Bregman surrogate functions are chosen, see Section~\ref{sec:TITAN_example} and~\cite{Mairal_ICML13} for examples of surrogate functions. Considering the nonsmooth nonconvex Problem~\eqref{model} with $g=0$, the authors in \cite{Razaviyayn2013} established sub-sequential convergence for the generated sequence of BSUM under some suitable assumptions. When $f$ and $g$ are convex functions, the iteration complexity of BSUM with respect to the optimality gap $F(x^k)-F(x^*)$, where $x^*$ is the optimal solution of \eqref{model}, was studied in \cite{Hong2017_BSUMcomplexity}. We note that global convergence for the generated sequence of BSUM for solving nonsmooth nonconvex Problem~\eqref{model} was not studied in~\cite{Razaviyayn2013}. 

\paragraph{Inertial methods} 

In the convex setting, the gradient descent (GD) method is known to have suboptimal convergence rate. To accelerate the convergence of the GD method, \cite{POLYAK1964} proposed the heavy ball method \revise{for solving the convex optimization problem  $\min_{x\in \mathbb R^n} f(x)$ by adding} an inertial force to the gradient direction using $\alpha^k(x^k-x^{k-1})$, where $x^k$ is the current iterate, $x^{k-1}$ is the previous iterate, and $\alpha^k$ is an extrapolation parameter. \revise{In fact, the heavy ball update is given by $x^{k+1}=x^{k}-\beta_{k} \nabla f(x^{k})   + \alpha_{k} (x^{k}-x^{k-1} )$, where $\beta_k$ is the step size.}  Later, in a series of works,  \cite{Nesterov1983,Nesterov1998,Nesterov2004,Nesterov2005} proposed the well-known 
accelerated fast gradient methods. 
 While extrapolation is not used to calculate the gradients as in the heavy ball method, Nesterov acceleration uses it to evaluate the gradients as well as adding the inertial force: \revise{denoting the extrapolation point $\bar x^k= x^{k} + \alpha_k (x^k-x^{k-1})$, Nesterov's acceleration has the form $ x^{k+1}=x^{k}-\beta_{k} \nabla f(\bar x^{k})   + \alpha_{k} (x^{k}-x^{k-1})$.}   
The spirit of using inertial terms to accelerate first-order methods has been brought to nonconvex problems. 
In the nonconvex setting, the  heavy ball acceleration type was used in \cite{Zavriev1993,Ochs2014,Ochs2019}, the Nesterov acceleration type was used in \cite{Xu2013,Xu2017}. Interestingly, using two different extrapolation points, one is for evaluating gradients and another one is for adding the inertial force, was also considered, by~\cite{Pock2016} and \cite{Hien_ICML2020}.  
Sub-sequential and global convergence of some specific inertial BCD methods for nonconvex problems have been established when $F$ is assumed to have the KL property, 
see, e.g., ~\cite{ahookhosh2020_inertial,Hien_ICML2020,Ochs2019,Xu2013,Xu2017}. To the best of our knowledge, applying acceleration strategies to the general BSUM framework has not been studied in the literature.

\subsection{Contribution}  First, we propose TITAN, a novel inertial block majorization minimization framework for solving the nonsmooth nonconvex problem~\eqref{model}. TITAN updates one block of $x$ at a time by choosing a surrogate function (see Definition~\ref{def:surrogate} and Section~\ref{sec:TITAN_example}) for the corresponding block objective function, embedding inertial force to this surrogate function and then minimizing the obtained inertial surrogate function. The novelty of TITAN lies in how we control the inertial force. Specifically, we use an extrapolation operator that can be wisely chosen depending on specific assumptions considered for Problem~\eqref{model} to produce various types of acceleration; see Section~\ref{sec:TITAN_example} for examples. 

Then, we study sub-sequential convergence as well as global convergence for TITAN, which  unifies the convergence analysis of many acceleration algorithms that TITAN subsumes.  
TITAN can be thought of as BSUM with extrapolation.  However, it is important noting that the objective function of Problem~\eqref{model} includes a separable nonsmooth function $g=\sum_{i=1}^m g_i$ that is very important to model the regularizers of many practical optimization problems, and we \revise{only require} $g$ to be lower semi-continuous. 
We note that Assumption~2~(B4) of \cite{Razaviyayn2013} on the continuity of the block surrogate functions of the objective $F$ over the joint variables could be violated for Problem~\eqref{model} when $g$ is not continuous but only lower semi-continuous. \revise{The sparse NMF problem~\eqref{sparseNMF} presented in Section~\ref{intro:appl} is such a case since $g_1$ will be the indicator function of a closed nonconvex set}.  And as such the analysis in~\cite{Razaviyayn2013} is not applicable to Problem~\eqref{model}. Furthermore, when no extrapolation is applied and $g=0$, TITAN becomes BSUM. Hence, 
the global convergence established for TITAN with suitable assumptions can be applied to derive the global convergence for BSUM, which was not studied in~\cite{Razaviyayn2013}. 

Finally, we illustrate the effectiveness of TITAN on the two applications presented in Section~\ref{intro:appl}, namely sparse NMF and the MCP. 
Applying TITAN to sparse NMF  illustrates the benefit of using inertial terms in BCD methods. The deployment of TITAN in solving the MCP illustrates the advantages of using suitable surrogate functions. Specifically, we will use a composite surrogate function for the MCP. Compared to the typical proximal gradient BCD method, each minimization step of  TITAN has a closed-form solution while each proximal gradient step does not. In our experiments, TITAN outperforms the proximal gradient BCD method (also known as proximal alternating linearized minimization), being at least 4 times faster on three widely used data sets.

\subsection{Organization of the paper} 

In the next section, we present TITAN with cyclic block update rule. 
In Section~\ref{sec:convergence}, we establish the subsequential and global convergence for TITAN. In Section~\ref{sec:TITAN_example}, we employ various surrogate functions and wisely choose the extrapolation operators to derive specific accelerated BCD methods. In particular, we recover \revise{the  inertial block proximal algorithm of~\cite{Hien_ICML2020} in Section~\ref{ex:proximal_surrogate}. In Section~\ref{sec:recover_Nesterov}, we recover the Nesterov type acceleration of~\cite{Xu2013,Xu2017} and the acceleration algorithm that uses two different extrapolation points of  \cite{Hien_ICML2020}. In Section~\ref{sec:Hessian_damping}, we use TITAN to derive a \emph{multiblock} version for the inertial gradient with Hessian damping proposed by~\cite{Adly2020}.  In Section~\ref{sec:Bregman_TITAN} and Section~\ref{sec:quadratic-surrogate} we use TITAN to derive heavy-ball type inertial block coordinate algorithms for  Bregman and quadratic surrogates. 
Furthermore, we employ TITAN to derive new inertial block coordinate methods for composite surrogates in Section~\ref{ex:composite_surrogate}. To the best of our knowledge, the inertial block coordinate methods in Sections~\ref{sec:Hessian_damping}  and~\ref{ex:composite_surrogate} and their convergence analysis are new.}  
We extend TITAN to allow essentially cyclic rule in choosing the block to update in Section~\ref{sec:essentially_cyclic}. In Section~\ref{sec:experiment}, we report the numerical results of TITAN applied on the sparse NMF and the MCP. We conclude the paper in Section~\ref{sec:conclusion}.

\section{Inertial Block Alternating Majorization Minimization}
\label{sec:multi_iMM}

In this section, we introduce TITAN, an inertial block alternating majorization-minimization  framework, with cyclic update rule.  The description of TITAN is given in Algorithm \ref{algo:iMM_multiblock}. 
\begin{algorithm}[ht!]
\caption{TITAN with cyclic update to solve Problem~\eqref{model}}
\begin{algorithmic}[1]
\label{algo:iMM_multiblock}
\REQUIRE Choose $x^{-1},x^0\in \mathcal X$ ($x^{-1}$ can be chosen equal to $x^0$).
\ENSURE $x^k$ that approximately solves~\eqref{model}. 
   \medskip
   \FOR{$k=0,1,\ldots$}
   \FOR{ $i = 1,...,m$}
   \STATE Choose a block $i$ surrogate function $u_i $ of $f$ and an extrapolation $\mathcal G^k_i(x^{k}_i, x^{k-1}_i)$. See Section~\ref{sec:requirement_Gu} for the conditions on  $u_i $  and $\mathcal G^k_i(x^{k}_i, x^{k-1}_i)$, and Section~\ref{sec:choice_Gu} 
   for general choices for $u_i$ and $\mathcal G_i^k(x^{k}_i, x^{k-1}_i)$.  \STATE\label{step-update} Update block $i$ by
   \begin{equation}
   \label{eq:iMM_update}
       \revise{x^{k+1}_i}
       \;\in\;\argmin_{x_i\in\mathcal X_i}u_i(x_i,x^{k,i-1}) - \langle  \mathcal G^k_i(x^{k}_i, x^{k-1}_i),x_i\rangle + g_i(x_i).
   \end{equation}
   \ENDFOR
    \ENDFOR
\end{algorithmic}
\end{algorithm}
At the $k$-th iteration, we cyclically update each block while fixing the values of the other blocks. 
 In Algorithm~\ref{algo:iMM_multiblock} and throughout the paper,  we use the notation 
$$\revise{x^{k,0} = x^k},\quad  
x^{k,i}=(x^{k+1}_1, \ldots,x^{k+1}_{i},x^{k}_{i+1},\ldots,x^{k}_m) \, \text{ for } i \in [m], 
\quad\text{ and }\, x^{k+1} = x^{k,m}.
$$
To update block $i$ at the $k$-th iteration, we first need to  choose a block $i$ surrogate function $u_i$ of $f$, which is defined below.
\begin{definition}[Block surrogate function]
\label{def:surrogate}
A function $u_i:\mathcal X_i \times \mathcal X \to \mathbb R  $ is called a block~$i$ surrogate function of $f$ 
if  $u_i(x_i,y)$ is continuous in $y$ and lower semi-continuous in $x_i$, and the following conditions are satisfied:
\begin{itemize}
    \item[(a)] $u_i(y_i,y) = f(y)$ for all $y\in \mathcal X$, 
    \item[(b)] $u_i(x_i,y) \geq f(x_i,y_{\ne i})$ for all $x_i\in\mathcal X_i$ and $y\in \mathcal X$, where $$f(x_i,y_{\ne i}):= f(y_1,\ldots,y_{i-1},x_i,y_{i+1},\ldots,y_m).$$
The block approximation error is defined as $h_i(x_i,y):=u_i(x_i,y) - f(x_i,y_{\ne i})$.
\end{itemize}
\end{definition}
 Then, we solve the sub-problem~\eqref{eq:iMM_update} in which the block surrogate function is equipped with an inertial force via the extrapolation operator $\mathcal G_i^k$. \revise{In the following, we give a simple example for the choice of $u_i$ and $\mathcal G_i^k$. More examples and a discussion in the context of TITAN are provided in Section~\ref{sec:TITAN_example}.} 
\revise{\begin{example}
Given a continuous function $f:\mathbb E_1\times \ldots\times \mathbb E_m\to \mathbb R$, we can take the block surrogate functions as $u_i(x_i,y) = f(x_i,y_{\ne i}) + \frac{\rho_i}{2}\|x_i - y_i\|^2,$  where $\rho_i$ is a positive scalar, and take the extrapolations as  $\mathcal G_i^k(x^k_i,x^{k-1}_i) =\rho_i\beta_i^k(x_i^k-x_i^{k-1})$, where $\beta_i^k$ are extrapolation parameters.  The update \eqref{eq:iMM_update} becomes
$$
\begin{array}{ll}\argmin\limits_{x_i \in \mathcal X_i} f(x_i,x^{k,i-1}_{\ne i}) + \frac{\rho_i}{2}\|x_i - (x^{k}_i+ \beta_i^k(x_i^{k}-x_i^{k-1})) \|^2 + g_i(x_i),
\end{array}
$$
which has the form of an inertial proximal method. In Section~\ref{ex:proximal_surrogate}, we will discuss a more general form of this choice ($\rho_i$ will be allowed to vary along with the updates of the blocks) and provide the details of its use in the context of TITAN.  
\end{example}
}
\revise{
\subsection{Conditions for TITAN} 
\label{sec:requirement_Gu} 
First note that TITAN is a generic scheme. The surrogate functions $u_i$ of TITAN must satisfy the following assumption (see Lemma \ref{prop:Lipschitzh} below for some sufficient conditions for Assumption~\ref{assump:surrogate_assum} to be satisfied).
\begin{assumption}\revise{[Bound of approximation error]}
\label{assump:surrogate_assum} 
\begin{itemize}
For $i\in [m]$, given $y\in \mathcal X$, there exists a function $x_i\mapsto \bar h_i(x_i,y)$ such that $ \bar h_i(\cdot,y)$ is continuously differentiable at $y_i$, $\bar h_i(y_i,y)=0$ and $\nabla_{x_i} \bar h_i(y_i,y)=0$, and the block approximation error $x_i\mapsto h_i(x_i,y)$  satisfies  
\begin{equation}
\label{lemma:h_property} 
h_i(x_i,y) \leq \bar h_i(x_i,y) \;  \text{ for all } \;  x_i \in \mathcal X_i.
\end{equation}
 \end{itemize}
\end{assumption} 
Together with Assumption~\ref{assump:surrogate_assum}, we also need the following additional condition on the generated sequence $\{x^k\}$.   
Once the formulas of surrogate functions $u_i$ as well as the extrapolation $\mathcal G_i^k$ are specified, TITAN generates a sequence, which must satisfy the following nearly sufficiently decreasing property (NSDP): 
\begin{equation}
\label{requirement}
\tag{NSDP}
F(x^{k,i-1}) +  \frac{\gamma_i^k}{2}\|x_i^{k}-x^{k-1}_i\|^2 \geq F(x^{k,i})+ \frac{\eta_i^k}{2}\|x_i^{k+1}-x^k_i\|^2, k=0,1,\ldots
\end{equation} 
where  $\gamma_i^k\geq 0$ and $\eta_i^k>0$ may depend  on the extrapolation parameters used in $\mathcal G_i^k$ and the parameters used to construct $u_i$, and the formulas of these sequences are known once $u_i$ and $\mathcal G_i^k$ are specified. 
}\revise{In Section~\ref{sec:choice_Gu}, we will provide sufficient conditions on  $u_i$ and $\mathcal G_i^k$ that make \eqref{requirement} satisfied.} 

The following lemma provides some sufficient conditions for Assumption~\ref{assump:surrogate_assum} to be satisfied.  
It will be used to verify Assumption~\ref{assump:surrogate_assum} for the block surrogate functions that will be given in Section~\ref{sec:TITAN_example}. 
\begin{lemma}
\label{prop:Lipschitzh}
Assumption~\ref{assump:surrogate_assum} is satisfied when one of the following two  conditions holds: 
\begin{itemize}
\item the block error $h_i(\cdot,y)$ is continuously differentiable at $y_i$ and 
$\nabla_{x_i} h_i(y_i,y)=0$,   
\item $h_i(x_i,y) \leq \upsilon_i\|x_i-y_i\|^{1+\epsilon_i}$ for some $\epsilon_i>0$ and $\upsilon_i>0$. 
\end{itemize} 
\end{lemma}
\begin{proof}
 In the first case, we take $\bar h_i(x_i,y)=h_i(x_i,y)$,  and in the second case, we take $\bar h_i(x_i,y) = \upsilon_i\|x_i-y_i\|^{1+\epsilon_i}$. 
\end{proof}

\subsection{General choices for $u_i$ and $\mathcal G_i^k$ such that the NSDP condition is satisfied}
\label{sec:choice_Gu}
Let us discuss the parameters $\gamma_i^{k}$ and $\eta_i^{k}$  in~\eqref{requirement}. In Section~\ref{sec:TITAN_example}, we provide their explicit formulas in  some specific examples of TITAN which correspond to specific choices of $u_i$ and $\mathcal G_i^k$. 
\revise{The following theorem is a cornerstone to characterize the general choices of $u_i$ and $\mathcal G_i^k $ that satisfy the~\eqref{requirement}. The two important parameters in Theorem~\ref{thrm:sufficient_uihi} to compute $\gamma_i^k $ and $\eta_i^k$ of~\eqref{requirement} are $\rho_i^{(y)}$ of Condition~\ref{condition-hi} (or $\rho_i^{(y)}$ of Condition~\ref{condition-ui}) and $A^k_i$ of Condition~\ref{assump:gki}. }
\revise{    
\begin{theorem}
\label{thrm:sufficient_uihi}
 Suppose $\mathcal G^k_i$ satisfies the following Condition~\ref{assump:gki}  and  $u_i$ satisfies the following Condition~\ref{condition-hi}. 
 \begin{requirement}
\label{assump:gki}
There exists a sequence $\{A^k_i\}_{i \in [m], k\geq 0}$ such that the extrapolation operator $\mathcal G^k_i$ satisfies 
$\|\mathcal G^k_i(x_i^k,x_i^{k-1})\|\leq A^k_i \|x_i^k-x_i^{k-1}\|$ for $i \in [m]$ and $k\geq 0$. \end{requirement}
 
 \begin{requirement}
\label{condition-hi}
Given $y\in \mathcal X$, there exists a positive constant $\rho_i^{(y)}$ 
(which may depend on~$y$) such that the block $i$ approximation error satisfies the inequality 
$$
\begin{array}{ll}
h_i(x_i,y) \geq \frac{\rho_i^{(y)}}{2}\|x_i-y_i\|^2 \; \text{ for all } \; x_i \in \mathcal X_i.
\end{array}
$$
\end{requirement}
Then the~\eqref{requirement} holds with 
\begin{equation}
\label{gamma-eta-rho}
\begin{array}{ll}
\gamma^{k}_i=\frac{(A^k_i)^2}{\nu \rho^{(x^{k,i-1})}_i } , \qquad\eta^{k}_i = (1-\nu)\rho^{(x^{k,i-1})}_i,
\end{array}
\end{equation}
where $0<\nu<1$ is a constant. For notation succinctness, we denote $\rho_i^k= \rho^{(x^{k,i-1})}_i$.

Equation~\eqref{requirement} also holds with $\gamma^{k}_i$ and $\eta^{k}_i$ 
given in~\eqref{gamma-eta-rho} if Condition~\ref{assump:gki} holds and the following condition~\ref{condition-ui} holds with $y=x^{k,i-1}$. 
 \begin{requirement} 
\label{condition-ui}
Given $y\in \mathcal X$, the function $x_i\mapsto u_i(x_i,y) + g_i(x_i)$ is $\rho^{(y)}_i$-strongly convex. 
\end{requirement}
\end{theorem}
}
\begin{proof} In this proof, we denote $y=x^{k,i-1}$. 
Let us consider the first case:  \revise{Condition}~\ref{assump:gki} and \revise{Condition}~\ref{condition-hi} hold. We have 
\begin{equation}
\label{eq:uihi}
\begin{array}{ll}
u_i(x^{k+1}_i,y)= f(x^{k+1}_i,y_{\ne i}) + h_i(x^{k+1}_i,y)
\geq 
f(x^{k+1}_i,y_{\ne i})+\frac{\rho^{k}_i}{2}\|x^{k+1}_i - x^{k}_i\|^2.
\end{array}
\end{equation}
On the other hand, it follows from \eqref{eq:iMM_update} that, for all $x_i \in \mathcal X_i$, we have
\begin{equation}
\label{eq:foundation}
u_i(x^{k+1}_i,y) + g_i(x^{k+1}_i)
\leq 
 u_i( x_i,y) - \langle   \mathcal G^k_i(x^k_i , x^{k-1}_i),  x_i	-x^{k+1}_i\rangle + g_i(x_i) .
\end{equation} 
Choosing $x_i=x^{k}_i$ in \eqref{eq:foundation}, we get the following inequality from~\eqref{eq:foundation} and \eqref{eq:uihi}: 
\begin{equation}
\label{eq:hstrong2}
\begin{array}{ll}
&u_i(x_i^k,y)+g_i(x_i^k)-\langle  \mathcal G^k_i(x^{k}_i , x^{k-1}_i),x^{k}_i-x^{k+1}_i\rangle
\\
&\qquad\geq f(x^{k+1}_i,y_{\ne i})+ g_i(x^{k+1}_i)+  \frac{\rho_i^{k}}{2}\|x_i^{k}-x^{k+1}_i\|^2. 
\end{array}
\end{equation}
Since
$u_i(x_i^k,y)=f(y)$, \revise{and recalling that $F(x)=f(x_1,\ldots,x_m)+\sum_{i=1}^m g_i(x_i)$, and   $f(x_i,y_{\ne i})= f(y_1,\ldots,y_{i-1},x_i,y_{i+1},\ldots,y_m)$}, we derive from \eqref{eq:hstrong2} that

\begin{equation}
\label{eq:hstrong}
\begin{array}{ll}
F(x^{k,i-1})-\iprod{ \mathcal G^k_i(x^{k}_i , x^{k-1}_i)}{x^k_i - x^{k+1}_i} \geq F(x^{k,i})+ \frac{\rho^{k}_i}{2}\|x_i^{k+1}-x^k_i\|^2. 
\end{array}
\end{equation} 
From Young's inequality, we have 
$$
\begin{array}{ll}
A^k_i \|x^{k}_i - x^{k-1}_i\| \|x^{k+1}_i - x^{k}_i\|\leq \frac{\nu\rho^k_i }{2} \|x^{k+1}_i - x^{k}_i\|^2 + \frac{(A^k_i)^2}{2\nu \rho^k_i }\|x^{k}_i - x^{k-1}_i\|^2.
\end{array}
$$
Hence, from \eqref{eq:hstrong} and  Requirement \ref{assump:gki}, we obtain
\begin{equation*}
\begin{array}{ll}
    F(x^{k,i}) +\frac{(1-\nu)\rho^k_i}{2}\|x^{k+1}_i - x^{k}_i\|^2  
   \leq
      F(x^{k,i-1}) +\frac{(A^k_i)^2}{2\nu\rho^k_i}\|x^k_i - x^{k-1}_i\|,
      \end{array}
\end{equation*}
which gives the result.

Let us now consider the second case, when  \revise{Conditions}~\ref{assump:gki} 
and~\ref{condition-ui} hold. 
Let $\tilde u_i(x_i,y)= u_i(x_i,y) + g_i(x_i)$. It follows from the optimality conditions of~\eqref{eq:iMM_update} that
\begin{equation}
\label{ieq:uistrongly} \iprod{\mathbf s_i(x^{k+1}_i)- \mathcal G^k_i(x^{k}_i , x^{k-1}_i)}{x^k_i - x^{k+1}_i}\geq 0, 
\end{equation}
where $\mathbf s_i(x^{k+1}_i)$ is a subgradient of $ \tilde u_i(\cdot,y)$ at $x^{k+1}_i$.
Since $\tilde u_i(\cdot,y)$ is strongly convex, we have 
$\tilde u_i(x_i^k,y) \geq \tilde u_i(x^{k+1}_i,y) + \iprod{\mathbf s_i(x^{k+1}_i)}{x_i^k-x^{k+1}_i} + \frac{\rho_i^k}{2}\|x_i^{k}-x^{k+1}_i\|^2. 
$
Together with~\eqref{ieq:uistrongly} and noting that $u_i(x^{k+1}_i,y)\geq f(x^{k+1}_i,y_{\ne i})$, we get~\eqref{eq:hstrong2}. The result follows using the same proof as in the first case.
\end{proof} 
Let us provide a sufficient condition for \revise{Condition}~\ref{condition-hi}. 
\begin{lemma}
\label{lemma:forh}
If $h_i(\cdot,y)$ is $\rho^{(y)}_i$-strongly convex and is differentiable at $y_i$, and \mbox{$\nabla_i h_i(y_i,y)=0$}, 
  then we have $h_i(x_i,y) \geq \frac{\rho_i^{(y)}}{2}\|x_i-y_i\|^2.$
\end{lemma}
\begin{proof}
The result follows from the definition of $\rho^{(y)}_i$-strong convexity, that is, 
$$h_i(x_i,y)\geq h_i(y_i,y)+ \iprod{\nabla_i h_i(y_i,y)}{x_i-y_i}+\frac{\rho^{(y)}_i}{2}\|x_i - y_i\|^2,
$$
the assumption $\nabla_i h_i(y_i,y)=0$, and  the property $h_i(y_i,y)=0$ from Definition~\ref{def:surrogate}.
\end{proof}
\revise{In Section~\ref{sec:TITAN_example}, we will provide the explicit formulas of $A_i^k$ in some specific examples. Note that $A_i^k$ may depend on the iterates. 
Condition~\ref{condition-hi} is always satisfied for the regularized block $i$ surrogate function that has the form $u_i(x_i,y) + \frac{\rho_i^{(y)}}{2}\|x_i-y_i\|^2$, where $u_i(x_i,y)$ is any block $i$ surrogate function of $f$. }

\section{Convergence analysis}
\label{sec:convergence} 
In this section, we will study sub-sequential convergence as well as global convergence of TITAN. \revise{Let us recall that TITAN is a generic framework, for which Assumption~\ref{assump:surrogate_assum} and the~\eqref{requirement} must be satisfied to obtain our convergence guarantees. 
To guarantee a sub-sequential convergence, we need the following additional conditions.
\begin{requirement}
\label{assump:parameter}
(i) For $k=0,1,\ldots$, we have 
 \begin{equation}
\label{parameter}
\gamma_i^{k+1}\leq  C\eta_i^{k} 
\end{equation} for some constant $0< C <1 $.

(ii) There exists a positive number $\underline{l}$ such that $\min_{i,k}\big\{\frac{\eta^{k}_i}{2}\big\}\geq \underline{l}$.
\end{requirement}
}

\begin{proposition}
\label{prop:sufficient_decrease}
 Let $\{x^k\}$ be the sequence generated by TITAN, that is, Algorithm~\ref{algo:iMM_multiblock}. 
 \revise{Suppose that the parameters of TITAN are chosen such that Condition \ref{assump:parameter} (i) holds.} Let $\eta_i^{-1}= \gamma_i^0/C$.  %
 Then the following statements hold.
 
(A) For any $K>1$, we have
    \begin{equation}
    \label{eq:to_prove_finite}
    \begin{array}{ll}
             F(x^{K})+  (1-C)\sum\limits_{k=0}^{K-1} \sum\limits_{i=1}^m \frac{\eta_i^k}{2}\|x^{k+1}_i - x^{k}_i\|^2  
      \leq 
      F(x^{0}) + C \sum\limits_{i=1}^m \frac{\eta^{-1}_i}{2} \|x^{0}_i - x^{-1}_i\|^2. 
      \end{array}
     \end{equation}

(B) \revise{If Condition \ref{assump:parameter} (ii) is also satisfied},  then we have  $$\sum_{k=0}^{+\infty}\sum_{i=1}^m\|x^{k+1}_i - x^{k}_i\|^2 < +\infty.$$ 
  
\end{proposition}
\begin{proof}
(A) 
 It follows from~\eqref{requirement} and~\eqref{parameter} that, for $k=0,1,\ldots$, we have 
\begin{equation}
\label{temp:prop5}
\begin{array}{ll}
   F(x^{k,i}) +\frac{\eta^k_i}{2}\|x^{k+1}_i - x^{k}_i\|^2  
   \leq
      F(x^{k,i-1}) + C\frac{\eta^{k-1}_i}{2}\|x^{k}_i - x^{k-1}_i\|^2.
      \end{array}
\end{equation} 
\revise{Note that $ \sum_{i=1}^m (F(x^{k,i})- F(x^{k,i-1}))=  F(x^{k+1})-F(x^{k})$.} Summing Inequality \eqref{temp:prop5} over $i=1,...,m$ gives 
\begin{equation}
\label{ieq:recursive}
\begin{array}{ll}
       F(x^{k+1})+ \sum\limits_{i=1}^m \frac{\eta^k_i}{2}\|x^{k+1}_i - x^{k}_i\|^2    
        \leq 
   F(x^{k}) +  C\sum\limits_{i=1}^m \frac{\eta^{k-1}_i}{2} \|x^{k}_i - x^{k-1}_i\|^2 .   
    \end{array}
\end{equation}
Summing up Inequality~\eqref{ieq:recursive} from $k=0$ to $K-1$, we obtain
$$
    \begin{array}{ll}
      &  F(x^{0}) +  \sum\limits_{i=1}^m  C\frac{\eta^{-1}_i}{2} \|x^{0}_i - x^{-1}_i\|^2\\
      & \qquad  \geq 
         F(x^{K})+ C\sum\limits_{i=1}^m\frac{\eta^{K-1}_i}{2}\|x^{K}_i - x^{K-1}_i\|^2 + (1-C)\sum\limits_{k=0}^{K-1} \sum\limits_{i=1}^m \frac{\eta^{k}_i}{2}\|x^{k+1}_i - x^{k}_i\|^2, 
    \end{array}
$$
which gives the result.

(B) The result is a direct consequence of the inequality \eqref{eq:to_prove_finite}.
\end{proof}

\subsection{Sub-sequential Convergence} 
\label{sec:subsequential}
Let us now prove sub-sequential convergence of TITAN. We will assume that the generated sequence $\{x^k\}$  is bounded which is a standard assumption, see \cite{Attouch2009,Attouch2010,Attouch2013,Bolte2007}. \revise{ From Inequality~\eqref{eq:to_prove_finite} in Proposition~\ref{prop:sufficient_decrease}, we have} that the boundedness of $\{x^k\}$  is satisfied for bounded-level set functions $F$. We will also assume $\|\mathcal G^k_i(x^{k}_i , x^{k-1}_i)\|$ goes to 0 when $k$ goes to $\infty$. \revise{ This assumption will be satisfied if  Condition~\ref{assump:gki} is satisfied and  $A^k_i$ is bounded for the bounded sequence $\{x^k\}$. Indeed, from Proposition~\ref{prop:sufficient_decrease}(B), $\|x^k_i - x^{k-1}_i\|$ converges to 0 when $k$ goes to $\infty$. Hence, if  $\|\mathcal G^k_i(x^{k}_i , x^{k-1}_i)\| \leq A^k_i \|x_i^k  - x^{k-1}_i\|$ and $A^k_i$ is bounded, then  $\|\mathcal G^k_i(x^{k}_i , x^{k-1}_i)\|$ goes to 0. 
}
 \begin{theorem}[Sub-sequential convergence]
\label{thm:subsequential_converge}
Suppose \revise{
Condition~\ref{assump:parameter} is satisfied for TITAN.} 
We further assume that the generated sequence $\{x^k\}$ by Algorithm~\ref{algo:iMM_multiblock} is bounded  ~~~~and  
$\|\mathcal G^k_i(x^{k}_i , x^{k-1}_i)\|$ goes to
0 when $k$ goes to $\infty$. Then every limit point $x^*$ of $\{x^k\}$ is a critical point of $\Phi$.

\end{theorem}

\begin{proof}
Suppose a subsequence $\{x^{k_n}\}$ of $\{x^k\}$ converges to $x^* \in \mathcal X$ (we remark that $x^{k+1}_i$ lies in $\dom g_i \cap \mathcal X_i$  for all $k\geq 0$, $i\in [m]$).  Proposition~\ref{prop:sufficient_decrease}(B) implies that $x^{k_n-1} \to x^*$ and $x^{k_n+1} \to x^*$.
Choosing $x_i=x^*_i$ and $k=k_n$ in \eqref{eq:foundation}, we obtain 
\begin{equation}
\label{eq:ui_converge}
\begin{array}{ll}
&u_i(x^{k_n+1}_i,x^{k_n,i-1}) +  g_i(x^{k_n+1}_i)
\\
&\qquad\leq 
 u_i( x_i^*,x^{k_n,i-1}) 
- \langle   \mathcal G^{k_n}_i(x^{k_n}_i , x^{k_n-1}_i),  x^*_i	-x^{k_n+1}_i\rangle + g_i(x^*_i).
 \end{array}
\end{equation}
Note that $x^{k_n,i-1}\to x^*$ and $u_i(x_i,y)$ is continuous in $y$. Hence, we derive from \eqref{eq:ui_converge} that 
$$\limsup_{n\to\infty} u_i(x^{k_n+1}_i,x^{k_n,i-1}) +  g_i(x^{k_n+1}_i) \leq u_i(x_i^*,x^*) + g_i(x^*).$$
Furthermore, $ u_i(x_i,y) + g_i(x_i)$ is lower semi-continuous. Hence, $u_i(x^{k_n+1}_i,x^{k_n,i-1}) +  g_i(x^{k_n+1}_i)$ converges to $u_i(x_i^*,x^*) + g_i(x_i^*)$. 	We then choose $k=k_n$ in \eqref{eq:foundation} and let $n\to\infty$ to obtain 
$$
u_i(x^*_i,x^*)+ g_i(x_i^*)\leq u_i(x_i,x^*)+ g_i(x_i)
\; \text{ for all } \; x_i \in \mathcal X_i.
$$
Note that $u_i(x^*_i,x^*)= f(x^*) $ and $u_i(x_i,x^*)= f(x_i,x^*_{\ne i}) + h_i(x_i,x^*)$. Therefore, for all  $x_i \in \mathcal X_i$, we have
\begin{equation}
\label{eq:block_min}
\begin{split}
F(x^*) &\leq  F(x_1^*,\ldots,x_{i-1}^*,x_i,x_{i+1}^*,\ldots,x_m^*)+ h_i(x_i,x^*)\\
&\leq  F(x_1^*,\ldots,x_{i-1}^*,x_i,x_{i+1}^*,\ldots,x_m^*)+ \bar h_i(x_i,x^*),
\end{split}
\end{equation}
where we have used Assumption~\ref{assump:surrogate_assum}. Inequality  \eqref{eq:block_min} shows that, for $i=1,\ldots,m$, $x_i^*$ is a minimizer of the problem 
\begin{equation}
\label{localmin}
\min_{x_i\in \mathcal X_i}  F(x_1^*,\ldots,x_{i-1}^*,x_i,x_{i+1}^*,\ldots,x_m^*)+ \bar h_i(x_i,x^*).
\end{equation}
 The result follows from the optimality condition of \eqref{localmin} and $ \nabla_i \bar h_i(x^*_i,x^*)=0$. 

\end{proof}

\begin{remark}
Considering the case $\mathcal X=\mathbb E: =\mathbb E_1 \times \ldots\times \mathbb E_m$,  the assumption that
 Inequality~\eqref{lemma:h_property} is satisfied for all $x_i \in \mathbb E_i$ can be relaxed to that 
 \revise{for any given bounded subset of  $\mathbb E_i$, $i \in [m]$,} Inequality~\eqref{lemma:h_property} is satisfied for any $x_i$ in \revise{this bounded subset.} 
 In other words, we relax the global bound for the block approximation error to the ``local" bound\footnote{\revise{Let us give an example when a property is not satisfied over the whole space but is satisfied over any given bounded subset of the space. The function $f(x)=x^3$ does not have Lipschitz continuous gradient over the whole space $\mathbb R$, but it has Lipschitz continuous gradient over any given bounded subset of $\mathbb R$.}}. \revise{Note that Inequality~\eqref{lemma:h_property} was not used before the proof of Theorem~\ref{thm:subsequential_converge}, it was not required in the proof of Proposition~\ref{prop:sufficient_decrease}. On the other hand, we assume that the generated sequence of TITAN is bounded (see the discussion at the beginning of Section~\ref{sec:convergence} for a sufficient condition on this boundedness assumption). Hence, we can consider Inequality~\eqref{lemma:h_property} in the closed bounded convex set $\bar{\mathcal X}$ that contains the generated sequence of TITAN and contains limit points $x^*$ as interior points.  
 We repeat the proof of Theorem~\ref{thm:subsequential_converge} to obtain the first inequality of~\eqref{eq:block_min}: for all $x_i\in \mathbb E_i$, 
$$
F(x^*) \leq  F(x_1^*,\ldots,x_{i-1}^*,x_i,x_{i+1}^*,\ldots,x_m^*)+ h_i(x_i,x^*). 
$$
This inequality implies that for all $x_i\in \bar{\mathcal X}_i$, we have the second inequality of \eqref{eq:block_min}. Consequently, $x_i^*$ is a minimizer of Problem \eqref{localmin} with $\mathcal X_i$ being replaced by $ \bar{\mathcal X}_i$. Note that $x_i^*$ is in the interior of $ \bar{\mathcal X}_i$. Hence, the subsequential convergence to a critical point of $F$ also holds for the relaxed condition. 
 }
\end{remark}

\subsection{Global Convergence} 

A global convergence recipe was proposed by \cite{Attouch2010,Attouch2013,Bolte2014} for  
proximal BCD (that is, when the proximal surrogate function is used; see also Section~\ref{ex:proximal_surrogate}) 
and proximal gradient BCD methods (that is, when the Lipschitz gradient surrogate function  is used; see also Sections~\ref{ex:Lipschitz_surrogate}) for solving nonsmooth nonconvex problems; see also Section~\ref{relwork}.  
The recipe was extended in \cite{Ochs2019} and \cite[Theorem 2]{Hien_ICML2020} to deal with the accelerated algorithms, which may produce non-monotone sequences of objective function values. 
For completeness, we provide \cite[Theorem 2]{Hien_ICML2020}, which will be used to prove the global convergence of TITAN, in Appendix \ref{sec:global_converge}. In order to achieve the global convergence of the generated sequence, we need the following additional assumption.
\begin{assumption}
\label{assump:Lipschitz_ui}
(i) For any $x,z\in \mathcal X$, we have  
\begin{equation}
\label{assum:sum_subgrad} 
\begin{split}
\partial_{x_i} \big( f(x) + g_i(x_i) + \mathcal I_{\mathcal X_i}(x_i)\big) = \partial_{x_i} f(x) + \partial_{x_i} \big( g_i(x_i) + \mathcal I_{\mathcal X_i}(x_i)\big), \\
\partial_{x_i} \big( u_i(x_i,z) + g_i(x_i) + \mathcal I_{\mathcal X_i}(x_i)\big) = \partial_{x_i} u_i(x_i,z) + \partial_{x_i} \big( g_i(x_i) + \mathcal I_{\mathcal X_i}(x_i)\big).
\end{split}
\end{equation} 

(ii) \revise{For any bounded subset of $\mathcal X$ and any $x,z$ in this subset}, for  $\mathbf s_i \in \partial_{x_i} u_i (x,z)$, there exists $\mathbf t_i \in \partial_{x_i} f(x)$ such that 
$$\|\mathbf s_i - \mathbf t_i \| \leq B_i \|x-z\|$$ for some constant $B_i$ \revise{that may depends on the subset}.
\end{assumption}
\revise{We make the following remarks for Assumption \ref{assump:Lipschitz_ui}.}
\begin{itemize}

\item We note that when $g_i=0$ and $\mathcal X_i=\mathbb E_i$ then Assumption~\ref{assump:Lipschitz_ui} (i) is satisfied. Let us give another simple sufficient condition that makes Assumption~\ref{assump:Lipschitz_ui} (i) hold: if the functions $x_i\mapsto f(x)$ and $x_i\mapsto u_i(x_i,z)$ are strictly differentiable then Assumption~\ref{assump:Lipschitz_ui} (i) is satisfied \cite[Exercise 10.10]{RockWets98}. We refer the readers to \cite{RockWets98} (Corollary 10.9) for a more general sufficient condition for Assumption~\ref{assump:Lipschitz_ui} (i).

\item \revise{It is important noting that the constants $B_i$ of Assumption~\ref{assump:Lipschitz_ui} (ii) do not influence  how to choose the parameters for TITAN, their existence is just for the purpose of proving the global convergence of the generated sequence. More specifically, as we assume that the generated sequence $\{x^k\}$ is bounded, in the proof of 
Theorem~\ref{thm:global_convergence}, we only work on a bounded set that contains $\{x^k\}$.} 

\item Assumption~\ref{assump:Lipschitz_ui} (ii) is naturally satisfied when the function $f(\cdot)$ and the surrogate functions $u_i(\cdot,\cdot)$ are continuously differentiable, $\nabla_{x_i} u_i(x_i,x)=\nabla_{x_i} f(x) $, and $\nabla_{x_i} u_i(\cdot,\cdot)$ is Lipschitz continuous on any bounded subsets of $\mathcal X_i \times \mathcal X$ since in this case we have $\nabla_{x_i} u_i(x,z) - \nabla_{x_i} f(x) = \nabla_{x_i} \big ( u_i(x,z) - u_i(x_i,x) \big)$. We note that all the surrogate functions given in Sections~\ref{ex:proximal_surrogate}--\ref{ex:quadratic_surrogate} satisfy Assumption~\ref{assump:Lipschitz_ui} when $f$ has Lipschitz continuous gradient on bounded subsets of $\mathcal X$. \revise{We refer the readers to 
\cite[Section~3]{HPNiADMM2021} for an example of nonsmooth $f$ that satisfies Assumption~\ref{assump:Lipschitz_ui}~(ii).} 

\end{itemize}

\begin{theorem}[Global convergence]
\label{thm:global_convergence}
Suppose \revise{the parameters of TITAN are chosen such that Condition  \ref{assump:parameter} is satisfied.} 
 Furthermore, we assume that, the block surrogate functions $u_i(x_i,y)$ is continuous on the joint variable $(x_i,y)$, Assumption~\ref{assump:Lipschitz_ui} holds,  \revise{Condition}~\ref{assump:gki} holds with bounded $A^k_i$, $\Phi$ is a KL function (see Appendix~\ref{sec:prelnnopt}), and together with the existence of $\underline{l}$, we also assume there exists $\overline{l}>0$ such that $  \max_{i,k}\big\{\frac{\eta^{k}_i}{2}\big\}\leq \overline{l}.
$
Suppose one of the following two  conditions hold. 
\begin{enumerate} 
\item Condition~\eqref{parameter} is satisfied with some $C$ satisfying $C < {\underline{l}}/{\overline{l}} $.
\item We use a restarting regime for TITAN, that is, if $ F(x^{k+1})  \geq F(x^k)$ then we re-do the $k$-iteration with $\mathcal G^k_i = 0$ (that is, no extrapolation is used). When restarting happens, we suppose that~\eqref{requirement} is satisfied with 
\footnote{If $u_i$ satisfies \revise{Condition}~\ref{condition-hi} or \revise{Condition}~\ref{condition-ui} then we repeat the proof of Theorem~\ref{thrm:sufficient_uihi} to derive Inequality~\eqref{eq:hstrong} which leads to Condition~\eqref{requirement} being satisfied with $\gamma_i^k=0$ and $\eta_i^k=\rho_i^k/2$.} $\gamma_i^k=0$, for $i\in [m]$. 
\end{enumerate}
Then the whole generated sequence $\{x^k\}$ of Algorithm~\ref{algo:iMM_multiblock}, which is assumed to be bounded, converges to a critical point $x^*$ of $\Phi$.
\end{theorem}
\begin{proof}  See Appendix~\ref{proof_global}.
\end{proof}

We make some remarks to end this section. 
\begin{remark}[Convergence rate] As long as a global convergence (see Theorem~\ref{thm:global_convergence}) is guaranteed, 
we can derive a convergence rate for the generated sequence using the same technique as in the proof of \cite{Attouch2009} (Theorem 2). We refer the reader to \cite[Theorem 3]{Hien_ICML2020} and \cite[Theorem 2.9]{Xu2013} for some examples of using the technique of \cite[Theorem 2]{Attouch2009} to derive the convergence rate and omit the details of the proof for the convergence rate for TITAN. The type of the convergence rate depends on the value of the K{\L} exponent $a$, where $\xi(s)=c s^{1-a}$ for some constant $c$ in Definition~\ref{def:KL}. 
In particular, if $a = 0$ then TITAN converges after a finite number of steps. 
If $a\in (0, 1/2]$ then TITAN has linear convergence, that is, there exists $k_0>0$, $\omega_1>0$ and $\omega_2\in [0,1)$ such that $\| x^k -x^*\| \leq \omega_1 \omega_2^k$ for all $k\geq k_0$.  
And if $a\in (1/2, 1)$ then TITAN has sublinear convergence, that is, there exists $k_0>0$ and $\omega_1>0$ such that  $\| x^k -x^*\| \leq \omega_1 k^{-(1-a)/(2a-1)}$ for all $k\geq k_0$.
Determining the value of $a$ is out of the scope of this paper. 
\end{remark}

\begin{remark}[With or without restarting steps?]
\label{remark:restarting}
If we target a global convergence guarantee \revise{and to avoid the restarting step (which could be expensive when the objective function is expensive to evaluate)}, TITAN without restarting steps is recommended when the bounds $\underline{l}$ and $\overline{l}$ are easy to estimate (\revise{then $C$ in Condition~\eqref{parameter} also needs to satisfy $C < {\underline{l}}/{\overline{l}}$}). 
If the values of the parameters $\eta_i^k$ vary along with the block updates, 
it is in general not easy to estimate the bounds $\underline{l}$ and $\overline{l}$. 
In that case, TITAN with a restarting regime is  recommended to guarantee a global convergence. It is important to note that TITAN always guarantees a sub-sequential convergence with or without restarting steps. 
\end{remark}

\section{Some TITAN Accelerated Block Coordinate Methods}
\label{sec:TITAN_example}
In order to guarantee \revise{a subsequential} convergence, TITAN must choose the parameters that satisfy the conditions in Theorem~\ref{thm:subsequential_converge}, which include Assumption~\ref{assump:surrogate_assum}, the~\eqref{requirement},  the condition $\|\mathcal G^k_i(x^{k}_i , x^{k-1}_i)\| \to 0$ and \revise{Condition \ref{assump:parameter}}. As noted in the first paragraph of Section~\ref{sec:subsequential}, the condition  $\|\mathcal G^k_i(x^{k}_i , x^{k-1}_i)\| \to 0$ is satisfied by the extrapolation satisfying  \revise{Condition}~\ref{assump:gki} with bounded $A_i^k$.  
Theorem~\ref{thrm:sufficient_uihi} characterizes some general properties of $u_i$  and $\mathcal G_i^k$ that make the~\eqref{requirement} hold, and it determines the corresponding values of $\eta_i^k$ and $\gamma_i^k$ when  \revise{Condition}~\ref{assump:gki} is satisfied, along with \revise{Condition}~\ref{condition-hi} or \ref{condition-ui}. 
Once $\eta_i^k$ and $\gamma_i^k$ are determined \revise{(such as in  \eqref{gamma-eta-rho})}, the condition in~\eqref{parameter} helps choose the appropriate extrapolation parameters to guarantee a subsequential convergence.

In the following, we consider 
 some important block surrogate functions from the literature (more examples can be found in~\cite{Mairal_ICML13}), 
 and  derive several specific instances of TITAN. We verify  Assumption~\ref{assump:surrogate_assum} using Lemma~\ref{prop:Lipschitzh}, and provide the formulas of  $\eta_i^k$ and $\gamma_i^k$ using Theorem~\ref{thrm:sufficient_uihi}. 
 \revise{TITAN recovers} many inertial methods from the literature; see  Section~\ref{ex:proximal_surrogate}--\ref{ex:quadratic_surrogate}. \revise{TITAN with Lipschitz gradient surrogates combined with an inertial gradient with Hessian damping of~\cite{Adly2020} gives us a new inertial block coordinate method; 
 see~Section~\ref{sec:Hessian_damping}. In Section~\ref{ex:composite_surrogate}, we use TITAN to derive new inertial methods when composite surrogates are used}. The method proposed in Section~\ref{ex:composite_surrogate} will be applied to solve the matrix completion problem in Section~\ref{sec:MCP}.

\subsection{TITAN with proximal surrogate function} \label{ex:proximal_surrogate}

The proximal surrogate function, which has been used for example in~\cite{Attouch2009,Attouch2013,Hien_ICML2020}, has the following form 
\begin{equation*}
\begin{array}{ll}
u_i(x_i,y) = f(x_i,y_{\ne i}) + \frac{\rho^{(y)}_i}{2}\|x_i - y_i\|^2,
\end{array}
\end{equation*}
where  $f$ is a lower semi-continuous function and $\rho^{(y)}_i>0$ is a scalar.

\paragraph{Verifying Assumption~\ref{assump:surrogate_assum}.}  We have $h_i(x_i,y)= \frac{\rho^{(y)}_i}{2} \|x_i-y_i\|^2$.  Hence, Assumptions~\ref{assump:surrogate_assum} and \revise{Condition}~\ref{condition-hi} are satisfied. 

\paragraph{Choosing $\mathcal G_i^k$ and determining $A^k_i$.} Let us choose $\mathcal G_i^k(x^k_i,x^{k-1}_i) =\rho_i^k\beta_i^k(x_i^k-x_i^{k-1})$, where $\beta_i^k$ are some extrapolation parameters and $\rho_i^k=\rho^{(x^{k,i-1})}_i $. We have $A^k_i=\rho_i^k\beta_i^k$. The minimization problem in the update \eqref{eq:iMM_update} becomes
$$
\begin{array}{ll}\min\limits_{x_i \in \mathcal X_i} f(x_i,x^{k,i-1}_{\ne i}) + \frac{\rho^{k}_i}{2}\|x_i - (x^{k}_i+ \beta_i^k(x_i^{k}-x_i^{k-1})) \|^2 + g_i(x_i).
\end{array}
$$

\paragraph{Verifying the~\eqref{requirement}.}
The formulas of  $\eta_i^k$ and $\gamma_i^k$ are determined as in Theorem~\ref{thrm:sufficient_uihi}, and the~\eqref{requirement} is thus satisfied.  \revise{Specifically, $\gamma_i^k =\frac{(A^k_i)^2}{\nu \rho_i^k}=(\beta_i^k)^2\rho_i^k/\nu$ and $\eta_i^k=(1-\nu)\rho_i^k$. Hence, when we choose the parameters $\beta_i^k$ and $\rho_i^k$ such that $(\beta_i^{k+1})^2 \rho_i^{k+1}/\nu \leq C (1-\nu) \rho_i^{k} $ and $\rho_i^k \geq\epsilon$ for some constants $\epsilon>0$, $0<\nu, C<1$, then Condition \ref{assump:parameter} is satisfied.

When we choose $\rho_i^k=\rho$ for all $i,k$, then we can take $\underline l=\overline l=(1-\nu)\rho$ so that the first condition of Theorem~\ref{thm:global_convergence} holds, and \eqref{parameter} becomes $\beta_i^{k+1} \leq \sqrt{\nu(1-\nu)C}$. The global convergence is then guaranteed without restarting steps.} 

This TITAN scheme recovers the inertial block proximal algorithm \revise{and the convergence results} of~\cite{Hien_ICML2020} for Problem~\eqref{model}.

\subsection{TITAN with Lipschitz gradient surrogates}
\label{sec:Nesterov_TITAN}
\label{sec:Lipschitz-surrogate} 
 \label{ex:Lipschitz_surrogate}

The Lipschitz gradient surrogate function, which has been used for example in~\cite{Xu2013,Xu2017,Hien_ICML2020}, 
has the form 
\begin{equation*}
\begin{array}{ll}
u_i(x_i,y) = f(y) + \langle\nabla_i f(y), x_i- y_i\rangle + \frac{\kappa_i L^{(y)}_i}{2}\|x_i - y_i\|^2,
\end{array}
\end{equation*}
where $\kappa_i\geq 1$, the block function $x_i\mapsto f(x_i,y_{\ne i})$ is differentiable and $\nabla_i f(x_i,y_{\ne i})$ is $L_i^{(y)}$-Lipschitz continuous. Note that $L^{(y)}_i$ may depend on $y$. \revise{The block approximation error $h_i$ for this case is 
$$ h_i(x_i,y)=f(y) + \langle\nabla_i f(y), x_i- y_i\rangle + \frac{\kappa_i L^{(y)}_i}{2}\|x_i - y_i\|^2 - f(x_i,y_{\ne i}).
$$}
\paragraph{Verifying Assumption~\ref{assump:surrogate_assum}.}  We have 
$$
\nabla_{x_i} h_i(x_i,y)= \kappa_i L^{(y)}_i (x_i-y_i) + \nabla_i f(y)- \nabla_i f(x_i,y_{\ne i}), 
$$
so that $\nabla_{x_i} h_i(y_i,y)=0$. Hence, Assumption~\ref{assump:surrogate_assum} is satisfied with $\bar h_i(x_i,y)= h_i(x_i,y)$.  

\paragraph{Choosing $\mathcal G_i^k$ and determining $A^k_i$.}
 We will consider two variants of $\mathcal G_i^k$: the choice in~\eqref{Gik_Lipschitz} that leads to inertial block proximal gradient methods, see Section~\ref{sec:recover_Nesterov}, and the choice in~\eqref{Gik_Hessian_damping} that leads to block proximal gradient algorithms with Hessian damping, see Section~\ref{sec:Hessian_damping}.

\paragraph{Verifying the~\eqref{requirement}.}

\revise{Consider the case} when $g_i(x_i)$ is a nonconvex function. \revise{As $\nabla_i f(x_i,y_{\ne i})$ is $L_i^{(y)}$-Lipschitz continuous, we have $x_i \mapsto \frac{L_i^{(y)}}{2} \|x_i\|^2 -  f(x_i,y_{\ne i})$ is convex, see \cite{Zhou2018misc}. Hence, we always have $x_i \mapsto h_i(x_i,y)$ is a $(\kappa_i-1) L_i^{(y)}$-strongly convex function}. 
In this case, we need to choose $\kappa_i>1$, and  \revise{Condition}~\ref{condition-hi} is satisfied with $\rho_i^{(y)}=(\kappa_i-1) L_i^{(y)}$.
If $g_i(x_i)$ is convex then we have  $x_i\mapsto u_i(x_i,y) + g_i(x_i)$ is a $\kappa_i L_i^{(y)}$-strongly convex function; as such, in this case we can choose $\kappa_i=1$ and \revise{Condition}~\ref{condition-ui} is satisfied with $ \rho_i^{(y)}= L_i^{(y)}$. 

In the following, we consider \revise{two} specific choices for $\mathcal G_i^k$, \revise{one leads to the inertial block proximal gradient method (Section~\ref{sec:recover_Nesterov}), 
the other leads to the Hessian damping algorithm (Section~\ref{sec:Hessian_damping}). We then} determine the corresponding values of $A^k_i$ \revise{and check Condition~\ref{assump:parameter}}. 
Taking $y=x^{k,i-1}$, the \revise{corresponding} formulas of  $\eta_i^k$ and $\gamma_i^k$ \revise{will be} determined as in Theorem~\ref{thrm:sufficient_uihi}, and \revise{hence} the~\eqref{requirement} is thus satisfied \revise{for both algorithms}.

\subsubsection{Deriving inertial block proximal gradient methods}
\label{sec:recover_Nesterov}
Let us consider the case  $\nabla_i f(x_i,y_{\ne i})$ is $L_i^{(y)}$-Lipschitz continuous over $\mathbb E_i$, and \revise{take}
\begin{equation}
\label{Gik_Lipschitz}
\mathcal G_i^k(x^k_i,x^{k-1}_i) = \nabla_i f(x^{k,i-1})- \nabla_i f(\bar x^{k}_i,x^{k,i-1}_{\ne i}) + \kappa_i L_i^k\beta^k_i(x^k_i-x^{k-1}_i) ,
\end{equation}
where $ \bar x^{k}_i=x^k_i+\tau_i^k(x^k_i-x^{k-1}_i)$, $ \tau_i^k$ and $\beta_i^k$ are some extrapolation parameters, and 
\mbox{$L_i^k = L_i^{(x^{k,i-1})}$}. 
The update in \eqref{eq:iMM_update} becomes
$$\begin{array}{ll}
&\argmin\limits_{x_i\in \mathcal X_i} f(x^{k,i-1})+\iprod{\nabla_i f(x^{k,i-1})}{x_i-x^k_i} + \frac{\kappa_i L_i^k}{2}\|x_i-x^k_i\|^2 \\
& \qquad -\Big\langle \nabla_i f(x^{k,i-1})- \nabla_i f(\bar x^{k}_i,x^{k,i-1}_{\ne i}) + \kappa_i L_i^k \beta^k_i(x^k_i-x^{k-1}_i),x_i \Big\rangle + g_i(x_i) \\
& = \argmin\limits_{x_i\in \mathcal X_i}   \iprod{\nabla_i f(\bar x^{k}_i,x^{k,i-1}_{\ne i})}{x_i}+ g_i(x_i) + \frac{\kappa_i L_i^k}{2}\big\|x_i-(x^k_i+\beta_i^k(x^k_i-x^{k-1}_i))\big\|^2 . 
\end{array}$$
We now determine the values of $A^k_i$ in  \revise{Condition}~\ref{assump:gki}. We consider the following situations. 
\begin{description}
\item[General case.] In general \revise{when no convexity is assumed for the block functions of $f$}, we have 
$$
\|\mathcal G_i^k(x^k_i,x^{k-1}_i)\|\leq L^k_i(\tau_i^k + \kappa_i\beta^k_i ) \|x^k_i-x^{k-1}_i\| 
$$
Hence, we can take $A_i^k=L^k_i (\tau_i^k +\kappa_i\beta^k_i)$. \revise{Let us recall  that $\rho_i^k=(\kappa_i-1)L_i^k$, \mbox{$\kappa_i>1$}, when no convexity is assumed for $g_i$ and   $\rho_i^k=L_i^k$,  $\kappa_i=1$, when $g_i$ is convex; see the above paragraph ``Verifying the (NSDP)". 
The formulas of $\gamma_i^k$ and $\eta_i^k$ are then determined as in \eqref{gamma-eta-rho} of Theorem \ref{thrm:sufficient_uihi}, and  Condition \ref{assump:parameter} (i) tells us how to choose the extrapolation parameters $\beta_i^k$ and $\tau_i^k$. Specifically, 
$(L^{k+1}_i)^2 (\tau_i^{k+1} + \kappa_i \beta_i^{k+1})^2 \leq C \nu \rho_i^{k+1} (1-\nu) \rho_i^k$. Regarding the first condition of Theorem~\ref{thm:global_convergence}, we see that estimating the values of $\overline l$ depends on estimating the bound for $L^k_i$ which highly depends on the problem at hand. 
As mentioned in Remark~\ref{remark:restarting}, a restarting step is necessary for a global convergence guarantee when the bound cannot be estimated. 
} 

\item[The block function $f(\cdot,x^{k,i-1}_{\ne i})$ is convex.] We can get a tighter value for $A_i^k$. Specifically, if we choose $ \beta_i^k \geq \tau_i^k$, then the function 
$$ 
x_i \mapsto \xi(x_i)=\frac12\kappa_i L_i^k\frac{\beta_i^k}{\tau_i^k} (x_i)^2 - f(x_i,x^{k,i-1}_{\ne i})
$$ 
is convex, and it has $\big(\kappa_i L_i^k\frac{\beta_i^k}{\tau_i^k}\big)$-Lipschitz gradient. Therefore, we get 
$$
\begin{array}{ll}
\|\nabla \xi(\bar{x}_i^k)-\nabla \xi(x_i^k)\| \leq \kappa_i L_i^k\frac{\beta_i^k}{\tau_i^k} \|\bar{x}_i^k-x_i^k\| =\kappa_i L_i^k \beta_i^k \|x_i^k - x_i^{k-1} \|. 
\end{array}
$$
On the other hand, we see that
$$
\begin{array}{ll}
\nabla \xi(\bar{x}_i^k)-\nabla \xi(x_i^k) &= \kappa_i L_i^k\frac{\beta_i^k}{\tau_i^k} \bar{x}_i^k - \nabla_i f(\bar{x}_i^k,x^{k,i-1}_{\ne i}) - \kappa_i L_i^k\frac{\beta_i^k}{\tau_i^k} x_i^k +  \nabla_i f(x^{k,i-1}) \\
& =\mathcal G_i^k(x^k_i,x^{k-1}_i).
\end{array}
$$
Hence, in this case, we can take $A_i^k=\kappa_i L^k_i \beta_i^k$.  \revise{The condition in \eqref{parameter} becomes $(\kappa_i L^{k+1}_i \beta_i^{k+1})^2 \leq C \nu \rho^{k+1}_i (1-\nu) \rho_i^k$, where $\rho_i^k=(\kappa_i-1)L_i^k$, $\kappa_i>1$, when no convexity is assumed for $g_i$ and   $\rho_i^k=L_i^k$, $\kappa_i=1$, when $g_i$ is convex. Similarly to the previous case, we see that estimating the value of $\overline l$ depends on estimating the upper bound for $L_i^k$. 
If such a bound is too difficult to estimate,  
then a restarting step is necessary 
to have a global convergence. 
}
\end{description}

This TITAN scheme recovers the accelerated methods \revise{and their convergence results} in the literature as follows.
\begin{itemize}
\item If we use $\mathcal G_i^k$ in \eqref{Gik_Lipschitz} and choose $\beta_i^k = \tau_i^k$  then we recover the Nesterov type acceleration as in~\cite{Xu2013,Xu2017}. 
\item If we use $\mathcal G_i^k$ in \eqref{Gik_Lipschitz} and let $\beta_i^k \ne \tau_i^k$ and $ \beta_i^k \geq \tau_i^k$ then the update in \eqref{eq:iMM_update} uses two different extrapolation points as in \cite{Hien_ICML2020}. 
\end{itemize}

It is important noting that we can also recover the heavy-ball type acceleration by choosing $\mathcal G_i^k(x^k_i,x^{k-1}_i) =  \kappa_i L_i^k\beta^k_i(x^k_i-x^{k-1}_i)$, and, for this case, we can  assume $\nabla_i f(x_i,y_{\ne i})$ is $L_i^{(y)}$-Lipschitz continuous over $\mathcal X_i$ (not necessary to be over $\mathbb E_i$). 

 \begin{remark}
 \label{remark:Lipschitz}
We have derived the values of $\eta_i^k$ and $\gamma_i^k$ using Theorem~\ref{thrm:sufficient_uihi}, and specific values of $A^k_i$ and 
$\rho_i^k$ of Theorem~\ref{thrm:sufficient_uihi} were given. 
 We have analyzed the following cases: (i)~the functions $f$ and $g_i$'s are nonconvex, (ii)~the block functions of $f$ are convex but the $g_i$'s are not, 
and (iii)~the function $f$ is nonconvex but the functions $g_i$'s are convex.  

When $F$ possesses the strong property that the block functions of $f$ are convex and the $g_i$'s are convex, we can obtain better values for $\gamma_i^k$ and $\eta_i^k$ that allow larger extrapolation parameters based on Condition~\eqref{parameter}. Let us choose $\mathcal G_i^k$ as in  \eqref{Gik_Lipschitz}. 
It was established in the proof in \cite[Remark 3]{Hien_ICML2020} that 
\begin{equation}
\label{eq:strongcase}
 F(x^{k,i-1}) +  \frac{L^k_i}{2}\big( (\tau_i^k)^2+ \frac{(\beta_i^k-\tau_i^k)^2}{\nu}\big)\|x_i^{k}-x^{k-1}_i\|^2 \geq F(x^{k,i})+ \frac{(1-\nu)L^k_i}{2}\|x_i^{k+1}-x^k_i\|^2, 
\end{equation}
where $0<\nu<1$ is a constant. Hence, in this case, the~\eqref{requirement} is satisfied with
\begin{equation}
\label{eq:gammastrong}
\gamma^{k}_i =L^k_i\Big( (\tau_i^k)^2+ \frac{(\beta_i^k-\tau_i^k)^2}{\nu}\Big), \quad \eta^{k}_i=(1-\nu)L^k_i.
\end{equation} 
Note that if we choose $\beta_i^k=\tau_i^k$, then the~\eqref{requirement} is satisfied with
$$\gamma^{k}_i =L^k_i (\tau_i^k)^2, \quad \eta^{k}_i=L^k_i, 
$$
see also \cite[Lemma 2.1]{Xu2013}. 
\end{remark}

\subsubsection{Inertial block proximal gradient algorithm with Hessian damping}
\label{sec:Hessian_damping}

Let us \revise{take}
\begin{equation}
\label{Gik_Hessian_damping}
\mathcal G_i^k = \alpha_i^k\big(\nabla_i f( x^{k-1}_i,x^{k,i-1}_{\ne i}\big) - \nabla_i f(x^{k,i-1}) ) + \kappa_i L_i^k \beta_i^k (x_i^k - x_i^{k-1}), 
\end{equation}
where $\alpha_i^k$ and $\beta_i^k$ are some extrapolation parameters. The problem in  \eqref{eq:iMM_update} becomes
$$\begin{array}{ll}
&\argmin\limits_{x_i} f(x^{k,i-1})+\iprod{\nabla_i f(x^{k,i-1})}{x_i-x^k_i} + \frac{\kappa_i L_i^k}{2}\|x_i-x^k_i\|^2 \\
& \qquad -\Big\langle \alpha_i^k \big(  \nabla_i f( x^{k-1}_i,x^{k,i-1}_{\ne i})-\nabla_i f(x^{k,i-1}) \big) + \kappa_i L_i^k \beta_i^k (x_i^k - x_i^{k-1}),x_i \Big\rangle + g_i(x_i) \\
& = \argmin\limits_{x_i}  \iprod{\nabla_i f( x^{k,i-1}) + \alpha_i^k\big(\nabla_i f(x^{k,i-1})-  \nabla_i f( x^{k-1}_i,x^{k,i-1}_{\ne i}) \big) }{x_i}+ g_i(x_i)\\
&\qquad\qquad \qquad\qquad + \frac{\kappa_i L_i^k}{2}\big\|x_i-(x^k_i+\beta_i^k(x^k_i-x^{k-1}_i))\big\|^2.
\end{array}$$

To determine the values of $A^k_i$ in  \revise{Condition}~\ref{assump:gki}, let us consider the following two situations: 
\begin{itemize}
\item In the general case \revise{when no convexity is assumed for  $f(\cdot,x^{k,i-1}_{\ne i})$, we have 
$$\|\mathcal G_i^k(x_i^k,x_i^{k-1})\|\leq L^k_i(\alpha_i^k + \kappa_i \beta_i^k) \|x^k_i-x^{k-1}_i\|.
$$
}
Hence, we 
take $A^k_i=L^k_i(\alpha_i^k + \kappa_i \beta_i^k)$.

\item If the block function $f(\cdot,x^{k,i-1}_{\ne i})$ is convex,  we choose $\alpha_i^k \leq \kappa_i \beta_i^k$ to guarantee the convexity of the function $x_i \mapsto \xi(x_i) = \frac12 \kappa_i L_i^k \beta_i^k (x_i)^2 - \alpha_i^k f( x_i,x^{k,i-1}_{\ne i})$. Note that $\xi(x_i)$ has $\kappa_i L_i^k \beta_i^k$-Lipschitz gradient. Hence, similarly to Section~\ref{sec:recover_Nesterov}, we can take $A_i^k=\kappa_i L_i^k \beta_i^k$.
\end{itemize}   
\revise{ The condition in \eqref{parameter} becomes 
$ (A^{k+1}_i)^2 \leq C \nu \rho_i^{k+1} (1-\nu) \rho_i^k$, where $\rho_i^k=(\kappa_i-1)L_i^k$, $\kappa_i>1$, when no convexity is assumed for $g_i$ and   $\rho_i^k=L_i^k$,  $\kappa_i=1$, when $g_i$ is convex. Furthermore, if the upper bound of $L_i^k$ is too difficult to estimate, using restarting step is recommended to have a global convergence guarantee.} 

With this TITAN scheme, we obtain an inertial block proximal gradient algorithm with the corrective term $\nabla_i f(x^{k,i-1})-  \nabla_i f( x^{k-1}_i,x^{k,i-1}_{\ne i})$ which is related to the discretization of the Hessian-driven damping term; see~\cite{Adly2020}. 
When $g_i(x_i)=0$, the update in \eqref{eq:iMM_update} becomes
$$
x^{k+1}_i=x_i^k + \beta_i^k (x_i^k - x_i^{k-1}) - \frac{1}{\kappa_i L_i^k} \Big( 
\nabla_i f( x^{k,i-1}) + \alpha_i^k \big(\nabla_i f(x^{k,i-1})-  \nabla_i f( x^{k-1}_i,x^{k,i-1}_{\ne i}) \big)
\Big),
$$
which has the form of the inertial gradient algorithm with Hessian damping of~\cite{Adly2020}. 

\subsection{TITAN with Bregman surrogates}
\label{sec:Bregman_TITAN}
\label{ex:Bregman_surrogate} 

The Bregman surrogate for relative smooth functions, which has been used for example in~\cite{ahookhosh2019multi,HienNicolas_KLNMF,Teboulle2020}, has the form 
\begin{equation*}
u_i(x_i,y) =  f(y) + \langle\nabla_i f(y), x_i- y_i\rangle + \kappa_i L^{(y)}_i D_{\varphi^{(y)}_i}(x_i,y_i),
\end{equation*}
where $\kappa_i\geq 1$,  the block function $x_i\mapsto f(x_i,y_{\ne i})$ is differentiable, $\varphi^{(y)}_i$ is a differentiable convex function such that the function $x_i \mapsto L^{(y)}_i \varphi^{(y)}_i(x_i) - f(x_i,y_{\ne i})$ is convex, and $D_{\varphi^{(y)}_i}$ is the block Bregman divergence associated with $\varphi^{(y)}_i$ defined by
\begin{equation}
    D_{\varphi^{(y)}_i}(x_i,v_i) = \varphi^{(y)}_i(x_i) - [\varphi^{(y)}_i(v_i) + \langle\nabla \varphi^{(y)}_i(v_i), x_i- v_i\rangle].
\end{equation}
 It is assumed that $\varphi^{(y)}_i$ is a $\rho_{\varphi^{(y)}_i}$-strongly convex function on $\mathbb E_i$ and its gradient is Lipschitz continuous on bounded subsets of $\mathbb E_i$.
\paragraph{Verifying Assumption~\ref{assump:surrogate_assum}.} 
\revise{The block approximation error $h_i$ for this case is
$$h_i(x_i,y)=  f(y) + \langle\nabla_i f(y), x_i- y_i\rangle + \kappa_i L^{(y)}_i D_{\varphi^{(y)}_i}(x_i,y_i)-f(x_i,y_{\ne i}). $$
We thus have }

 $$
 \nabla_{x_i} h_i(x_i,y)=\kappa_i L^{(y)}_i(\nabla \varphi^{(y)}_i(x_i)  -\nabla \varphi^{(y)}_i(y_i)) +  \nabla_i f(y)- \nabla_i f(x_i,y_{\ne i}).
 $$ 
Hence,  Assumption~\ref{assump:surrogate_assum} is satisfied with $\bar h_i(x_i,y)= h_i(x_i,y)$.
 
\paragraph{Choosing $\mathcal G_i^k$ and determining $A^k_i$.} 
 Let us consider \revise{when a weak inertial force is used:} $\mathcal G_i^k(x^k_i,x^{k-1}_i) =\beta_i^k(x_i^{k-1}-x_i^k)$, where $\beta_i^k$ are some extrapolation parameters. In this case, we have $A^k_i=\beta_i^k$. \revise{This case recovers the block inertial Bregman proximal algorithm in~\cite{ahookhosh2020_inertial}}.  

\paragraph{Verifying the~\eqref{requirement}.} We use Theorem~\ref{thrm:sufficient_uihi} to determine the values of  $\eta_i^k$ and $\gamma_i^k$ of the~\eqref{requirement}. 
Similarly to Section~\ref{ex:Lipschitz_surrogate}, if $g_i(x_i)$ is convex then $x_i\mapsto u_i(x_i,y) + g_i(x_i)$ is a $\kappa_i L^{(y)}_i \rho_{\varphi_i}$-strongly convex function. In this case we can choose $\kappa_i=1$ and Condition~\ref{condition-ui} is satisfied with $\rho_i^{(y)}=L^{(y)}_i \rho_{\varphi^{(y)}_i}$. Considering the case \revise{when no convexity is assumed for} $g_i$, as we have $h_i(\cdot,y)$ is a $(\kappa_i-1)L^{(y)}_i \rho_{\varphi^{(y)}_i}$-strongly convex function,  we need to choose $\kappa_i>1$, 
and  \revise{Condition}~\ref{condition-hi} is satisfied with $\rho_i^{(y)} = (\kappa_i-1)L^{(y)}_i \rho_{\varphi^{(y)}_i} $. Taking $y=x^{k,i-1}$, the formulas of  $\eta_i^k$ and $\gamma_i^k$ are determined as in Theorem~\ref{thrm:sufficient_uihi}. 

\revise{Therefore, when weak inertial force is used, the condition  \eqref{parameter} becomes 
$(\beta_i^{k+1})^2\leq C\nu \rho_i^{k+1} (1-\nu)\rho_i^k$. If we further assume that  $L_i^{(y)}=L_i$, for  $i=1,\ldots,m$ (that is, $L_i^{(y)}$ is independent of $y$, see \cite{ahookhosh2020_inertial}) then the first condition of Theorem \ref{thm:global_convergence} can be verified, that leads to a global convergence without restarting steps.  
}

\revise{In the following, we propose another method to choose $\mathcal G_i^k$ that leads to a new inertial algorithm when Bregman surrogates are used.}
\paragraph{ Heavy ball type acceleration with back-tracking.} Let us choose $$\mathcal G_i^k(x^k_i,x^{k-1}_i) = \kappa_i L_i^k(\nabla \varphi^k_i(\bar x^{k}_i)-\nabla \varphi^k_i(x^k_i)),
$$
where $ \varphi^k_i = \varphi^{(x^{k,i-1})}_i$, $ \bar x^{k}_i=x^k_i+\tau_i^k(x^k_i-x^{k-1}_i)$ with $\tau_i^k$ being extrapolation parameters. 
Recall we assume that $\varphi^k_i(\cdot)$ is strongly convex and differentiable on $\mathbb E_i$, and hence  $\nabla \varphi^k_i(\bar x^{k}_i)$ is well-defined.  The update \eqref{eq:iMM_update} becomes
$$\begin{array}{l}
 \argmin\limits_{x_i} \iprod{\nabla_i f( x^{k,i-1}) }{x_i-x^k_i}+ g_i(x_i)  + \kappa_i L_i^k \big( \varphi^k_i(x_i)  -\langle \nabla \varphi^k_i(\bar x^k_i),x_i-\bar x_i^k\rangle -  \varphi^k_i(\bar x_i^k)\big) \\
\quad =\argmin\limits_{x_i}  \iprod{\nabla_i f( x^{k,i-1}) }{x_i-x^k_i}+ g_i(x_i) + \revise{\kappa_i L_i^k} D_{\varphi^k_i}(x_i,\bar x_i^k),
\end{array}$$
\noindent which has the form of a heavy ball acceleration of \revise{\cite{POLYAK1964}}. Note that we do not assume that  $\nabla\varphi_i^k$ is globally Lipschitz continuous. Therefore, we propose to apply line-search to determine the extrapolation parameter $\tau_i^k$ as follows. 
Starting with $\tau_i^k=1$, we decrease $\tau_i^k$ by multiplying it with a constant $\bar\tau<1$ until the following condition is satisfied
$$ \kappa_i L_i^k\|\nabla \varphi^k_i(\bar x^{k}_i)-\nabla \varphi^k_i(x^k_i) \|^2 \leq C\|x^k_i-x^{k-1}_i \|^2  \rho_i^k \rho_i^{k+1}.
$$
This process terminates after a finite number of steps as we assume $\nabla\varphi_i^k(x_i)$ is Lipschitz continuous on \revise{any given} bounded sets  of $\mathbb E_i$.  Then the condition in~\eqref{parameter} is satisfied with $A^k_i=\frac{\|\nabla \varphi^k_i(\bar x^{k}_i)-\nabla \varphi^k_i(x^k_i) \|}{\|x^k_i-x^{k-1}_i \|}$. Since $\nabla\varphi_i^k(\cdot)$  is Lipschitz continuous on \revise{any given} bounded subsets, we have $A^k_i$ is bounded over the bounded set that contains the generated sequence. 
\subsection{TITAN with quadratic surrogates}
\label{sec:quadratic-surrogate}
\label{ex:quadratic_surrogate} 

The quadratic surrogate, which has been used for example in~\cite{Emilie2016,Ochs2019}, has the following form
\begin{equation}
u_i(x_i,y) =  f(y) + \langle\nabla_i f(y), x_i- y_i\rangle + \frac{\kappa_i}{2}(x_i- y_i)^T H^{(y)}_i(x_i- y_i),
\end{equation}
where $\kappa_i\geq 1$, $f$ is twice differentiable and $H^{(y)}_i$ is a positive definite matrix such that $(H^{(y)}_i - \nabla_i^2 f(x_i,y_{\ne i}))$ is positive definite ($H^{(y)}_i$ may depend on $y$). 

Taking $y=x^{k,i-1}$, we note that the quadratic surrogate   is a special case of the Bregman surrogate (Section~\ref{sec:Bregman_TITAN}) with $\varphi^k_i(x)=x_i^T H^k_i x_i$, $L_i^k=1$ and $\rho_{\varphi_i^k}$ being the smallest eigenvalue of $H_i^k$. However, it is important noting that the kernel function $\varphi_i^k(x_i)=\iprod{x_i}{H_i^k x_i}$ is globally $\|H_i^k\|$-Lipschitz smooth. Therefore, we can recover the heavy ball type acceleration as in Section~\ref{ex:Bregman_surrogate} but without back-tracking for the extrapolation parameters as follows. 
We choose $G_i^k$ as 
$$
\mathcal G_i^k(x^k_i,x^{k-1}_i) = \kappa_i (H^k_i(\bar x^{k}_i)-H^k_i(x^k_i))= \kappa_i \tau_i^k H^k_i(x_i^k-x_i^{k-1}),
$$
where $\bar x^{k}_i = x^{k}_i + \tau_i^k(x_i^k-x_i^{k-1})$. In this case, $A_i^k=\kappa_i \tau_i^k \|H^k_i\|$. The update in \eqref{eq:iMM_update} has the form of a heavy ball acceleration
$$\argmin_{x_i}  \iprod{\nabla_i f_i( x^k_i) }{x_i-x^k_i}+ g_i(x_i) + \frac{\kappa_i}{2}(x_i-\bar x_i^k)^T H_i^k (x_i-\bar x_i^k).
$$
\revise{The condition in \eqref{parameter} for this case is 
$ (\kappa_i \tau_i^{k+1} \| H_i^{k+1}\|)^2\leq C \nu \rho_i^{k+1} (1-\nu) \rho_i^k$, where $\rho_i^k=(\kappa_i-1) \lambda_{\min}(H_i^k)$, $\kappa_i>1$, if no convexity is assumed for $g_i$, and $\rho_i^k=\lambda_{\min}(H_i^k)$, $\kappa_i=1$, if $g_i$ is convex. The upper bound of $ \lambda_{\min}(H_i^k) $ highly depends on specific applications. In case this bound is not easy to estimate, a restarting step can be used to have global convergence. 
}

\subsection{TITAN with composite surrogates}
\label{sec:compositeTITAN}
\label{ex:composite_surrogate} 

In this section, we derive new inertial algorithms when using composite surrogates. Suppose $f$ has the form 
\begin{equation}
\label{eq:f-composite}
f(x)=\psi(x) + \phi(r(x)),
\end{equation}  
where
\begin{itemize}

\item $\psi: \mathcal X \to \mathbb R$ is a nonsmooth nonconvex function, \revise{and let us denote $u_i^{\psi}(x_i,y)$, for $i \in [m]$, to be block surrogate functions of $\psi$},

\item  $r = (r_1,...,r_m)$, where $r_i:\mathcal X_i\to \mathcal Y_i \subset\mathbb F_i$ are Lipschitz continuous (that is, $\| r_i(x_i) - r_i(y_i)\| \leq L_{r_i} \|x_i-y_i\|$ for $x_i, y_i \in \mathcal X_i$) 
and $\mathbb F_i$ ($i=1,\ldots,m$) are finite dimensional real linear spaces, and 

\item $\phi:\mathcal Y := \mathcal Y_1\times ... \times\mathcal Y_m\to \mathbb R_+$ is a continuously differentiable and block-wise concave function with Lipschitz gradient.

\end{itemize} 
There are several practical problems in machine learning that minimize an objective function of the form~\eqref{eq:f-composite}; see for example \cite{brafea,fanvar,phagro}. We will provide an example with the MCP in Section~\ref{sec:experiment}. 

Considering $f$ of the form~\eqref{eq:f-composite}, we propose to use the following composite block surrogate functions: 
\begin{equation*}
    u_i(x_i,y) = u_i^{\psi}(x_i,y) + \phi(r(y)) + \langle\nabla_i\phi(r(y)), r_i(x_i) - r_i(y_i)\rangle.
\end{equation*}
Since the block function of $\phi$ is concave, we have
\begin{equation}
\label{eq-concavity}
 (\phi\circ r)(x_i, y_{\neq i})  \leq \phi(r(y)) + \langle\nabla_i\phi(r(y)), r_i(x_i) - r_i(y_i)\rangle,
\end{equation}
where $\langle\nabla_i\phi(r(y))$ is the gradient of $\phi$ at $r(y)$ with respect to block $i$.

\paragraph{Verifying Assumption~\ref{assump:surrogate_assum}.} Let us assume the block surrogate functions $u_i^\psi(\cdot,\cdot) $ of $\psi(\cdot)$ satisfy Assumption~\ref{assump:surrogate_assum}. We prove that the block surrogate functions $u_i$ of $f$ also satisfy Assumption~\ref{assump:surrogate_assum}. Indeed, we have 
\begin{align*}
& h_i(x_i,y)= u_i(x_i,y)-f_{\ne i}(x_i,y)\\
& = u^\psi_i(x_i,y)-\psi(x_i,y_{\ne i}) + \phi(r(y))+ \langle\nabla_i\phi(r(y)), r_i(x_i) - r_i(y_i)\rangle -  \phi\circ r(x_i, y_{\neq i}).  
\end{align*}
Moreover, as we assume $\nabla_i \phi$ is Lipschitz continuous, we have
$$
  \phi(r(y)) + \langle\nabla_i\phi(r(y)), r_i(x_i) - r_i(y_i)\rangle  -  (\phi\circ r)(x_i,y_{\neq i})\leq  \frac{L_i^\phi}{2}\|r_i(x_i) - r_i(y_i)\|^2,
$$
for some constant $ L_i^\phi$. 
Therefore, we obtain 
\begin{equation}
\label{eq:hi_composite}
\begin{array}{ll}
h_i(x_i,y)&\leq  u^\psi_i(x_i,y)-\psi(x_i,y_{\ne i}) +\frac{L_i^\phi}{2}\|r_i(x_i) - r_i(y_i)\|^2\\
 & \leq  u^\psi_i(x_i,y)-\psi(x_i,y_{\ne i}) + \frac{L_i^\phi (L_{r_i})^2}{2}\|x_i - y_i\|^2,
 \end{array}
\end{equation}
where we use the Lipschitz continuity of $r_i(\cdot)$ in the last inequality. Since $u_i^\psi(\cdot,\cdot) $ satisfies Assumption~\ref{assump:surrogate_assum}, it follows from~\eqref{eq:hi_composite} that $u_i(\cdot,\cdot)$ satisfies Assumption~\ref{assump:surrogate_assum}.

\paragraph{Choosing $\mathcal G_i^k$ and determining $A_i^k$.}
The values of $A^k_i$ of Theorem~\ref{thrm:sufficient_uihi} depends on how we choose 
block surrogate functions for $\psi$, and how we choose $\mathcal G_i^k$. 
Specific examples and their corresponding values of $A^k_i$ that were presented in Section~\ref{sec:Lipschitz-surrogate}, Section~\ref{sec:Bregman_TITAN} and Section~\ref{sec:quadratic-surrogate} can be used for $\psi$. 

\paragraph{Verifying the~\eqref{requirement}.}
Let us determine the values of $\rho_i^k$ of Theorem~\ref{thrm:sufficient_uihi} for the two cases (i) $u_i^{\psi}$ satisfies \revise{Condition}~\ref{condition-hi}, and (ii) $u_i^\psi(\cdot,y)$ satisfies \revise{Condition}~\ref{condition-ui} and $x_i\mapsto  \langle\nabla_i\phi(r(y)), r_i(x_i) \rangle$ is convex. 
For the first case, we see that  $u_i(x_i,y)$ also satisfies \revise{Condition}~\ref{condition-hi}. Indeed,  it follows from Inequality~\eqref{eq-concavity} that 
$$
h_i(x_i,y)\geq u_i^\psi(x_i,y) - \psi(x_i,y_{\ne i}) \geq \frac{\rho_i^{(y)}}{2}\|x_i-y_i\|^2.
$$
For the second case, we see that $u_i(x_i,y) + g_i(x_i)$ is also a $\rho_i^y$-strongly convex function.   
The formulas of  $\eta_i^k$ and $\gamma_i^k$ are then determined as in Theorem~\ref{thrm:sufficient_uihi} \revise{and the condition in \eqref{parameter} tells us how to choose the corresponding extrapolation parameters such that a subsequential convergence is guaranteed}.

\begin{remark}
\label{remark:composite}
Let us consider the case when $g_i(x_i)$ and $x_i\mapsto  \langle\nabla_i\phi(r(y)), r_i(x_i) \rangle$,  for $i\in [m]$, are convex, $\psi(x)$ is a block-wise convex function, 
and its block functions $x_i\mapsto \psi(x_i,y_{\ne i})$ are continuously differentiable with $L_i^{(y)}$-Lipschitz gradient. 
We choose the Lipschitz gradient surrogate for $\psi$, 
and $\mathcal G_i^k$ as in  \eqref{Gik_Lipschitz}. 
Let $y=x^{k,i-1}$ and $L_i^k=L_i^{(x^{k,i-1})}$. Using the same technique as in the proof of \cite[Remark 3]{Hien_ICML2020}, we get the following inequality (note that we can also take $F=\psi(x) + \sum_{i=1}^m \big(\langle\nabla_i\phi(r(y)), r_i(x_i) \rangle + g_i(x_i) \big)$ in~\eqref{eq:strongcase} to obtain the result): 
$$\begin{array}{ll}
&\psi(x^{k,i-1}) +  \langle\nabla_i\phi(r(y)), r_i(x_i^{k})  \rangle + g_i(x^{k}_i) + \frac{L^k_i}{2}\big( (\tau_i^k)^2+ \frac{(\beta_i^k-\tau_i^k)^2}{\nu}\big)\|x_i^{k}-x^{k-1}_i\|^2
\\
 &\geq \psi(x^{k,i})+\langle\nabla_i\phi(r(y)), r_i(x_i^{k+1})  \rangle + g_i(x^{k+1}_i)+ \frac{(1-\nu)L^k_i}{2}\|x_i^{k+1}-x^k_i\|^2. 
\end{array}
$$
Together with \eqref{eq-concavity}, we obtain 
\begin{equation}
\label{eq:composite-strong}
\begin{array}{ll}
&\psi(x^{k,i-1}) + \phi(r(y)) + g_i(x^{k}_i) + \frac{L^k_i}{2}\big( (\tau_i^k)^2+ \frac{(\beta_i^k-\tau_i^k)^2}{\nu}\big)\|x_i^{k}-x^{k-1}_i\|^2 \\
 &\geq \psi(x^{k,i}) + (\phi\circ r)(x^{k+1}_i, y_{\neq i})+ g_i(x^{k+1}_i)+ \frac{(1-\nu)L^k_i}{2}\|x_i^{k+1}-x^k_i\|^2. 
 \end{array}
\end{equation}
Moreover, recall that $F(x)=\psi(x)+ \phi(r(x)) + \sum_{i=1}^m g_i(x_i)$. Therefore, Inequality~\eqref{eq:composite-strong} recovers Inequality~\eqref{eq:strongcase}, and we can take $\eta_i^k$ and $\gamma_i^k$ as in \eqref{eq:gammastrong}.
\end{remark}


\section{Extension to essentially cyclic rule}
\label{sec:essentially_cyclic}

In this section, we extend TITAN to allow the essentially cyclic rule in the block updates \revise{; see e.g., \cite{Xu2017,Tseng2001,Hong2017_BSUMcomplexity,Latafat2022}}. 
Instead of cyclically updating the $m$ blocks as in Algorithm~\ref{algo:iMM_multiblock}, the updated block of variables,   $i_k \in [m]$, is randomly or deterministically chosen. 
The essentially cyclic rule with interval $T\geq m$ imposes that each of the $m$ blocks is at least updated once every $T$ steps.  
Starting with two initial points $x^{-1}$ and $x^0$, at iteration $k$, $k\geq 0$, TITAN with essentially cyclic rule will update $x^k$ as follows: 
\begin{equation}
   \label{eq:Titan_ess_cyc}
       x^{k+1}_{i_k}\in\argmin_{x_{i_k}\in\mathcal X_{i_k}}\biggl\{u_{i_k}(x_{i_k},x^{k}) - \langle \mathcal G^k_{i_k}(x^{k}_{i_k}, x^{prev}_{i_k}),x_{i_k}\rangle + g_{i_k}(x_{i_k}) \biggr\},
   \end{equation}
  and set $x^{k+1}_a = x^{k}_a$ for all $a\neq {i_k}$. Here we use $x^{prev}_{i_k}$ to denote the value of block $i_k$ before it was updated to $x^k_{i_k}$. To simplify the presentation of our upcoming analysis, we  use the following notations: 
  \begin{itemize}
  
\item Starting from $x^0$, we split the generated sequence $\{x^k\}$ into partitions of $T$ consecutive iterates.  
We denote $\mathbf{x}^{k}$ the last iterate in every partition,  
that is, $\mbfx^{k} = x^{kT}$ for $k\geq 0$. 
We denote $\mbfx^{-1} = x^{-1}$. 

\item $\mbfx^{k,j}$ for $j \in [T]$ are the points within the sequence  $\{x^k\}$ lying between $\mbfx^{k}$ and $\mbfx^{k+1}$, that is, $\mbfx^{k,j}= x^{kT+j}$.

\item Since a block may not be updated in some consecutive iterations, we denote $\barmx^{k,l}_i$ the value of block $i$ after it has been updated $l$ times with the $k$-th partition $$[\mbfx^{k}, \mbfx^{k,1},\ldots, \mbfx^{k,T-1}, \mbfx^{k+1}=\mbfx^{k,T}].$$ In other words, $\barmx^{k,l}_i$ records the value of the $i$-th
block when it is actually updated. The
previous value of block $i$ before it is updated to $\barmx^{k,l}_i$  ( which is $x_i^{k,j}$ for some $j$)
is $\barmx^{k,l-1}_i$ (which is $x_i^{k,j-1}$). Correspondingly, we use $d^k_i$ to denote the total number of times the $i$-th block is updated during the $k$-th partition. 

\item $\mbfx^{k}_{prev}$ stores the previous values of the blocks of $\mbfx^{k}$, that is, $ (\mbfx^{k+1}_{prev})_i=\barmx_i^{k,d_i^k-1}$.
   \end{itemize}
Using these notations, we express the generated sequence $\{x^n\}_{n\geq 0}$ as the following sequence $\{\mbfx^{k,j}\}_{k\geq 0, j=0,\ldots,T-1}$:
  \begin{equation}
  \label{xsequence}
  \ldots,\mbfx^{k}=\mbfx^{k,0}, \mbfx^{k,1},\ldots, \mbfx^{k,T-1}, \mbfx^{k+1}=\mbfx^{k,T},\ldots
  \end{equation}
So $x^n=\mbfx^{k,j}$ with $k=\lfloor\frac{n}{T}\rfloor$ being the largest integer number that does not exceed $\frac{n}{T}$. 
Let us now translate the~\eqref{requirement} using this notation. 
The inequality~\eqref{requirement} for updating block $i$ in the $k$-th partition becomes
\begin{equation}
\label{eq:condition1}
F(\mbfx^{k,j-1}) +  \frac{\gamma^{(\mbfx^{k,j-1})}_i}{2}\|\mbfx^{k,j-1}_i-x_i^{prev}\|^2 \geq F(\mbfx^{k,j})+ \frac{\eta^{(\mbfx^{k,j-1})}_i}{2}\|\mbfx^{k,j}_i-\mbfx_i^{k,j-1}\|^2. 
\end{equation}
Note that $ x^{prev}_i$, $\mbfx_i^{k,j-1}$ and $\mbfx_i^{k,j}$ are three consecutive points of  $\{\barmx^{k,l}_i\}_{l=-1,\ldots,d_i^k}$.  We remark that $\barmx^{k,-1}_i=(\mbfx^k_{prev})_i$. So if $\mbfx_i^{k,j-1}$ is $\barmx^{k,l-1}_i$ then $\barmx^{k,l-2}_i=x^{prev}_i$ and \mbox{$\barmx^{k,l}_i=\mbfx_i^{k,j}$}. Inequality~\eqref{eq:condition1} is rewritten as
\begin{equation}
\label{eq:condition_2}
F(\mbfx^{k,j-1})+ \frac{\bar\gamma^{k,l-1}_i}{2}\|\barmx^{k,l-1}_i-\barmx^{k,l-2}_i\|^2 \geq F(\mbfx^{k,j})+ \frac{\bar\eta^{k,l-1}_i}{2}\|\barmx^{k,l}_i-\barmx^{k,l-1}_i\|^2, 
\end{equation}
where $\bar\eta^{k,l-1}_i=\eta^{(\mbfx^{k,j-1})}_i$ and $\bar\gamma^{k,l-1}_i=\gamma^{(\mbfx^{k,j-1})}_i$. 
All the convergence results so far still hold for TITAN with the essentially cyclic update rule. For example, the following proposition has the same essence as Proposition~\ref{prop:sufficient_decrease}.

\begin{proposition}
\label{prop:sufficient_decrease_2}
Considering TITAN with essentially cyclic rule, let $\{\mbfx^{k,l} \}$ be the generated sequence of TITAN, see \eqref{xsequence}. Assume that the parameters are chosen such that the conditions in~\eqref{eq:condition_2}(or its equivalent form in~\eqref{eq:condition1}), and Assumption~\ref{assump:surrogate_assum} is satisfied. Furthermore, suppose for $k=0,1,\ldots$ and $l \in [d^k_i]$,  we have  
\begin{equation}
\label{chooseparameter_2}
\bar \gamma_i^{k,l} \leq C{\bar\eta}_i^{k,l-1} ,
\end{equation} 
for some constant  $0< C <1$. Let $\bar\eta_i^{0,-1}=\bar\gamma_i^{0,0}/C$.  

(A) We have
    \begin{equation}
   \label{eq:finite_length}
  F(\mbfx^{K})+ (1-C)\sum_{k=0}^{K-1} \sum_{i=1}^m  \sum_{l=1}^{d_i^k} \frac{{\bar\eta}_i^{k,l-1} }{2} \|\barmx^{k,l}_i-\barmx^{k,l-1}_i\|^2
\leq F(\mbfx^{0})+  C\sum_{i=1}^m \frac{{\bar\eta}_i^{0,-1} }{2} \|\mbfx^{0}_i-\mbfx^{-1}_i\|^2.
    \end{equation}

(B) If there exists positive number $\underline{l}$ such that $\min_{i,k,l}\big\{ \frac{{\bar\eta}_i^{k,l} }{2}\big\}\geq \underline{l}$,   then  
$$\sum_{k=0}^{+\infty}\sum_{i=1}^m \sum_{l=1}^{d^k_i}\|\barmx^{k,l}_i - \barmx^{k,l-1}_i\|^2 < +\infty.$$ 

\end{proposition}

\begin{proof}
See Appendix~\ref{proof-essential}
\end{proof}

A subsequence $\{x^{k_n}\}$ of $\{x^n\}_{n\geq0}$ when being expressed as $\mbfx^{k,l}$ (see \eqref{xsequence}) is $\{\mbfx^{\bar k_n,l_n}\}$ with $\bar k_n= \lfloor\frac{k_n}{T} \rfloor$ and $l_n=k_n-T \lfloor\frac{k_n}{T} \rfloor$.   We derive from Proposition~\ref{prop:sufficient_decrease_2} that if $\barmx^{k,l_k}_i$ converges to $x_i^*$ as $k$ goes to 0, then $\barmx^{k,l}_i$ also converges to $x_i^*$ for $l=1,\ldots,d_i^k$.  From this fact, we use the same technique in the proof of Theorem~\ref{thm:subsequential_converge} to establish the subsequential convergence. We omit the details here. 

For the global convergence, \revise{we follow the the proof of Theorem~\ref{thm:global_convergence} (see Appendix \ref{proof_global}). To do so, we need to define the following potential function}  
$$
\Phi^\delta(x,y) :=  \Phi(x) + \sum_{i=1}^m\frac{\delta_i}{2}\|x_i - y_i\|^2,
$$
\revise{define the following auxiliary sequence}
$$
\begin{array}{ll}
\varphi_k^2=\sum\limits_{i=1}^m \sum\limits_{l=0}^{d_i^{k}} \frac12\|\barmx_i^{k,l}-\barmx_i^{k,l-1}\|^2=\sum\limits_{i=1}^m \sum\limits_{l=1}^{d_i^{k}}\frac12 \|\barmx_i^{k,l}-\barmx_i^{k,l-1}\|^2 + \frac12\|\mbfx^k-\mbfx^k_{prev}\|^2, 
\end{array} 
$$
and let $\mathbf z^k=(\mbfx^k,\mbfx_{prev}^k)$. Then, we have 
 $$\begin{array}{ll}
\Phi^\delta(\mathbf z^k)-\Phi^\delta(\mathbf z^{k+1})=F(\mbfx^k) - F(\mbfx^{k+1})+\sum\limits_{i=1}^m\frac{\delta_i}{2}\|\mbfx^k_i-(\mbfx_{prev}^k)_i\|^2 - \sum\limits_{i=1}^m\frac{\delta_i}{2}\|\mbfx^{k+1}_i - (\mbfx_{prev}^{k+1})_i\|^2.
\end{array}$$
Similarly to Theorem~\ref{thm:global_convergence} we assume there exists $\overline{l}$ such that $\max_{i,k,l}\big\{ \frac{{\bar\eta}_i^{k,l} }{2}\big\}\leq \overline{l}$.  
Then, as for Theorem~\ref{thm:global_convergence}, 
we can  prove that the whole sequence $\{\mbfx^k\}$ converges to $x^*$ in the following two cases: 
$C<\underline{l}/\overline{l}$, 
or applying restarting steps for~\eqref{eq:Titan_ess_cyc}. Hence each sequence $\{\mbfx^k_i\}_{k\geq 0}$ converges to $x_i^*$ for $i\in [m]$. Finally, note that 
$$
\begin{array}{ll}
\|\mbfx^{k,j-1}-x^*\|^2&\leq (T-j+2)\big(\sum_{a=j-1}^{T-1}\|\mbfx^{k,a}-\mbfx^{k,a+1}\|^2+\|\mbfx^{k+1}-x^*\|^2\big)\\
&\leq  (T-j+2) \big(\sum_{i=1}^m\sum_{l=1}^{d^k_i} \|\barmx_i^{k,l-1}-\barmx_i^{k,l}\|^2 + \|\mbfx^{k+1}-x^*\|^2\big).
\end{array}
$$
Together with Proposition~\ref{prop:sufficient_decrease_2}(B) it implies that the whole sequence $\{x^k\}$ converges.

\section{Numerical results}  \label{sec:experiment}

In this section, we apply TITAN to the sparse NMF~\eqref{sparseNMF} and the MCP~\eqref{MF}. All tests are preformed using Matlab
R2019a on a PC 2.3 GHz Intel Core i5
of 8GB RAM. The code is available from \url{https://github.com/nhatpd/TITAN}.

\subsection{Sparse Non-negative Matrix Factorization}
\label{sec:sparseNMF}

Let us consider the sparse NMF problem~\eqref{sparseNMF}, with two blocks of variables: $x_1 = U$ and $x_2 = V$. 
The functions $\nabla_U f(U,V)=(UV-M) V^T$ and $\nabla_V f(U,V) = U^T (UV-M)$ are   Lipschitz continuous with constants $L_1=\|VV^T\|$ and $L_2=\|U^TU\|$, respectively. Hence we choose the block Lipschitz surrogate for $f$ as in Section~\ref{ex:Lipschitz_surrogate}. 
Let us also choose the Nesterov-type acceleration as in Section~\ref{sec:recover_Nesterov}. The corresponding update in \eqref{eq:iMM_update} for $U$ is 
\begin{align*}
U^{k+1}=\argmin_{U} \iprod{\nabla_U f(\bar U^k,V^k)}{U}  + \frac{\kappa_1 L_1^k}{2} \|U-\bar U^k\|^2 + g_1(U),
\end{align*}
where  $\kappa_1>1$ is a constant, $\bar U^k=U^k + \beta_1^k(U^k-U^{k-1})$, $L_1^k=\|V^k (V^k)^T\|$, and the corresponding update for $V$ is
$$\begin{array}{ll}
V^{k+1}&=\argmin_{V} \iprod{\nabla_V f(U^{k+1},\bar V^k)}{V}  + \frac{ L_2^k}{2} \|V-\bar V^k\|^2 + g_2(V)\\
&=\big[\bar V^k- \frac{1}{ L_2^k}\nabla_V f(U^{k+1},\bar V^k) \big]_+,
\end{array}$$
where $\bar V^k=V^{k} + \beta_2^k(V^k-V^{k-1})$, $L_2^k=\|(U^{k+1})^TU^{k+1}\|$ and $[a]_+$ denotes $\max\{a,0\}$. 
It was shown in~\cite{Bolte2014} that the update of $U$ has the form 
$$U^{k+1}= \mathcal T_s\Big(\big[ \bar U^k- \frac{1}{\kappa_1 L_1^k} \nabla_U f(\bar U^k,V^k)\big]_+\Big),$$
where $\mathcal{T}_s(a)$ keeps the $s$ largest values of $a$ and sets the remaining values of $a$ to zero. 

Let us now determine $\eta_i^k $ and $\gamma_i^k$ for $i=1,2$, of Condition~\eqref{parameter}. Note that $f(\cdot,V)$,  $f(U,\cdot)$ and  $g_2(\cdot)$ are convex functions but $g_1(\cdot)$ is nonconvex. It follows from  Section~\ref{sec:recover_Nesterov} that $\rho^k_1(V)=(\kappa_1-1)L_1^k$ and $ A^k_1=\kappa_1 \beta_1^k L_1^k$ for the block $U$ surrogate functions. Applying Theorem~\ref{thrm:sufficient_uihi}, we get $\eta_i^k$ and $\gamma_i^k$, and the condition~\eqref{parameter} for block $U$ becomes
$
\beta_1^k \leq \frac{\kappa_1-1}{\kappa_1}\sqrt{\frac{C\nu_1 (1-\nu_1)L_1^{k-1}}{L_1^k}},
$
where $0<C,\nu_1<1$. 
 Considering block $V$, as both  $f(U,\cdot)$ and  $g_2(\cdot)$ are convex, it follows from Remark~\ref{remark:Lipschitz} that $\gamma_2^k=L_2^k(\beta_i^k)^2 $ and $\eta_2^k=L_2^k$. Hence, the condition~\eqref{parameter} for block $V$ becomes
$
\beta_2^k\leq \sqrt{\frac{C L_2^{k-1}}{L_2^k}},
$
where $0<C<1$. 
 In our experiments, we choose
 $$ 
\begin{array}{ll}
& C=0.9999^2, \mu_0=1, \mu_k=\frac12(1+\sqrt{1+4\mu_{k-1}^2}),\nu_1=1/2, \\
&\beta_1^k =\min\Big\{\frac{\mu_{k-1}-1}{\mu_k},\frac{\kappa_1-1}{\kappa_1}\sqrt{\frac{C\nu_1 (1-\nu_1)L_1^{k-1}}{ L_1^k}}\Big\}, 
\beta_2^k =\min\Big\{\frac{\mu_{k-1}-1}{\mu_k},\sqrt{\frac{C L_2^{k-1}}{ L_2^k}}\Big\}. 
\end{array}
$$
Since TITAN also works with essentially cyclic rule, in our experiment, we update $U$ several times before updating~$V$ and vice versa. As explained in~\cite{Gillis2012}, repeating update~$U$ or $V$ accelerates the algorithm compared to the cyclic update since the terms $V V^T$ and $MV^T$ in the gradient of~$U$ (resp.\ the terms $U^TU$ and $U^T M$ in the gradient of~$V$) do not need to be re-evaluated hence the next evaluation of the gradient only requires $O(\mathbf m \mathbf r^2)$ (resp.\ $O(\mathbf n \mathbf r^2)$) operations in the update of $U$ (resp.~$V$) compared to $O(\mathbf m \mathbf n \mathbf r)$ of the cyclic update. 
In our experiments, we use $\kappa_1=1.0001$ and use ``TITAN - $\kappa=1.0001$"  to denote the respective TITAN algorithm. As we do not use restarting, the TITAN algorithm guarantees a sub-sequential convergence. To verify the effect of inertial terms, we compare our TITAN algorithms with its non-inertial version, which is the proximal alternating linearized minimization (PALM) proposed in~\cite{Bolte2014}. 

\revise{It is worth mentioning iPALM which is another inertial version of PALM proposed by \cite{Pock2016}. We observe from Section 5.1 of the paper that iPALM with dynamic inertial parameters much outperforms other variants of iPALM that use constant inertial parameters, and iPALM using constant inertial parameters just perform similarly to PALM. However, the convergence analysis of \cite{Pock2016} does not support the setting of iPALM with dynamic inertial parameters. As our main purpose of this section is to verify the effect of inertial terms of our TITAN algorithms (note that the inertial parameters $\beta_1^k$ and $\beta_2^k$ of TITAN are dynamic, and we still have convergence guarantee), we will only report the performance of TITAN algorithms and PALM in the following.}

\paragraph{Dense facial images data sets} In the first experiment, we test the algorithms on four facial image data sets: 
Frey\footnote{\url{https://cs.nyu.edu/~roweis/data.html}} (1965 images of dimension  $28 \times 20$), CBCL\footnote{\url{http://cbcl.mit.edu/software-datasets/heisele/facerecognition-database.html}}  
(2429 images of dimension 19 $\times$ 19), Umist\footnote{\url{https://cs.nyu.edu/~roweis/data.html}} (575 images of dimension  $92 \times 112$),  and ORL\footnote{\url{ https://cam-orl.co.uk/facedatabase.html}} (400 images of dimension  $92 \times 112$). 
We choose $\mbfr=25$ and take a sparsity of $s$ equal to 0.25$\mbfr$, that is, each column of $U$ contains at most 25\% non-zero entries. 
For each data set, we run all the algorithms 20 times, use the same initialization each time for all algorithms which is generated by the Matlab commands $W=rand(\mbfm,\mbfr)$ and $H=rand(\mbfr,\mbfn)$, and run each algorithm for 100 seconds for the Frey and CBCL data sets, and 300 seconds for the Umist and ORL data sets. We define the relative error as $\|M-UV\|_F/\|M\|_F$. Figure~\ref{fig:sparsenmf} reports the evolution with respect to time of the average values of $E(k):=\|M-U^kV^k\|_F/\|M\|_F - e_{\min}$, where $e_{\min}$ is the smallest value of all the relative errors in all runs. Table~\ref{tab:sparsenmf} reports the average and the standard deviation (std) of the relative errors.  
\begin{figure*}[ht]
\begin{center}
\begin{tabular}{cc}
Frey & CBCL\\
\includegraphics[width=0.43\textwidth]{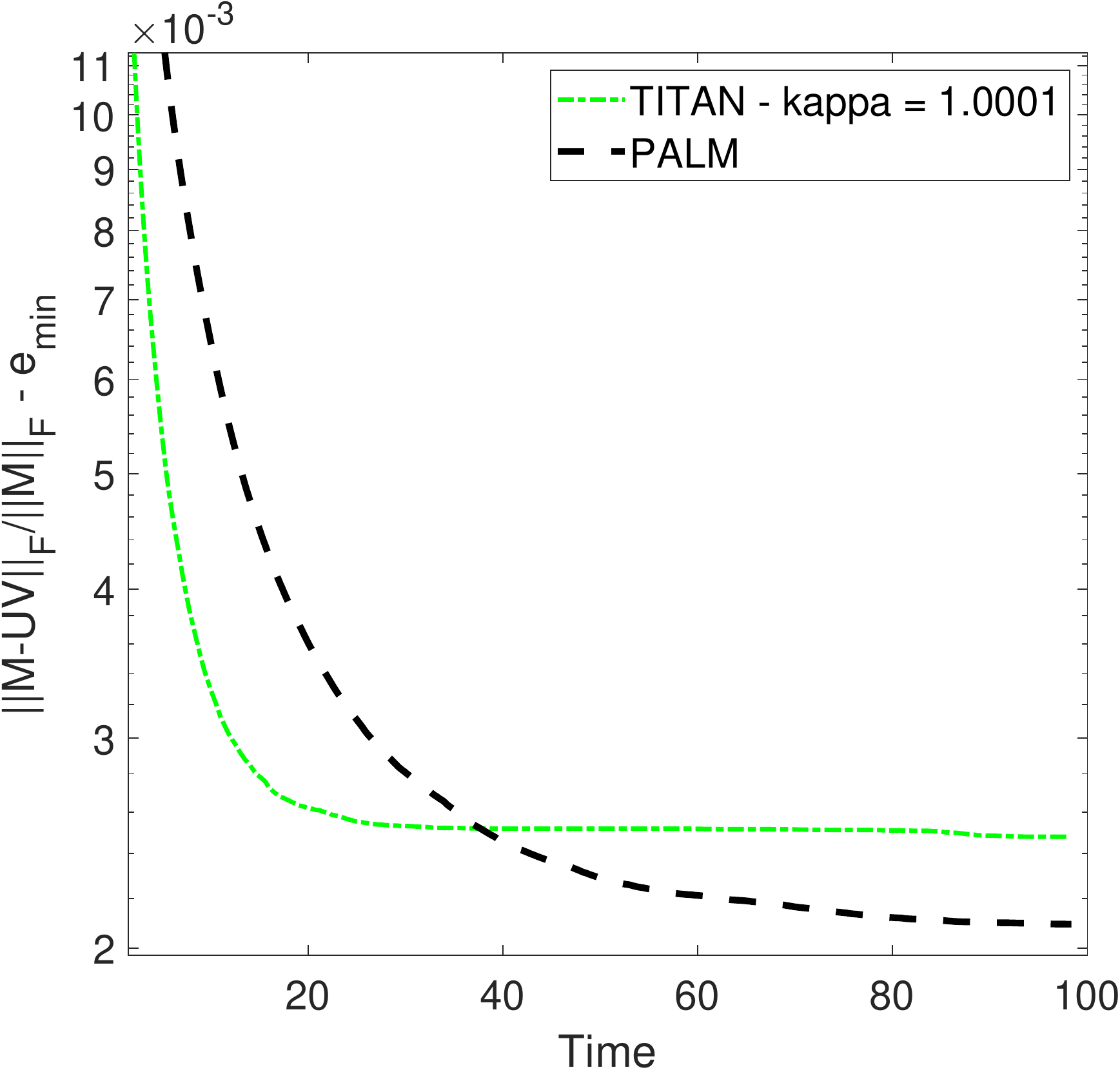}  & 
\includegraphics[width=0.421\textwidth]{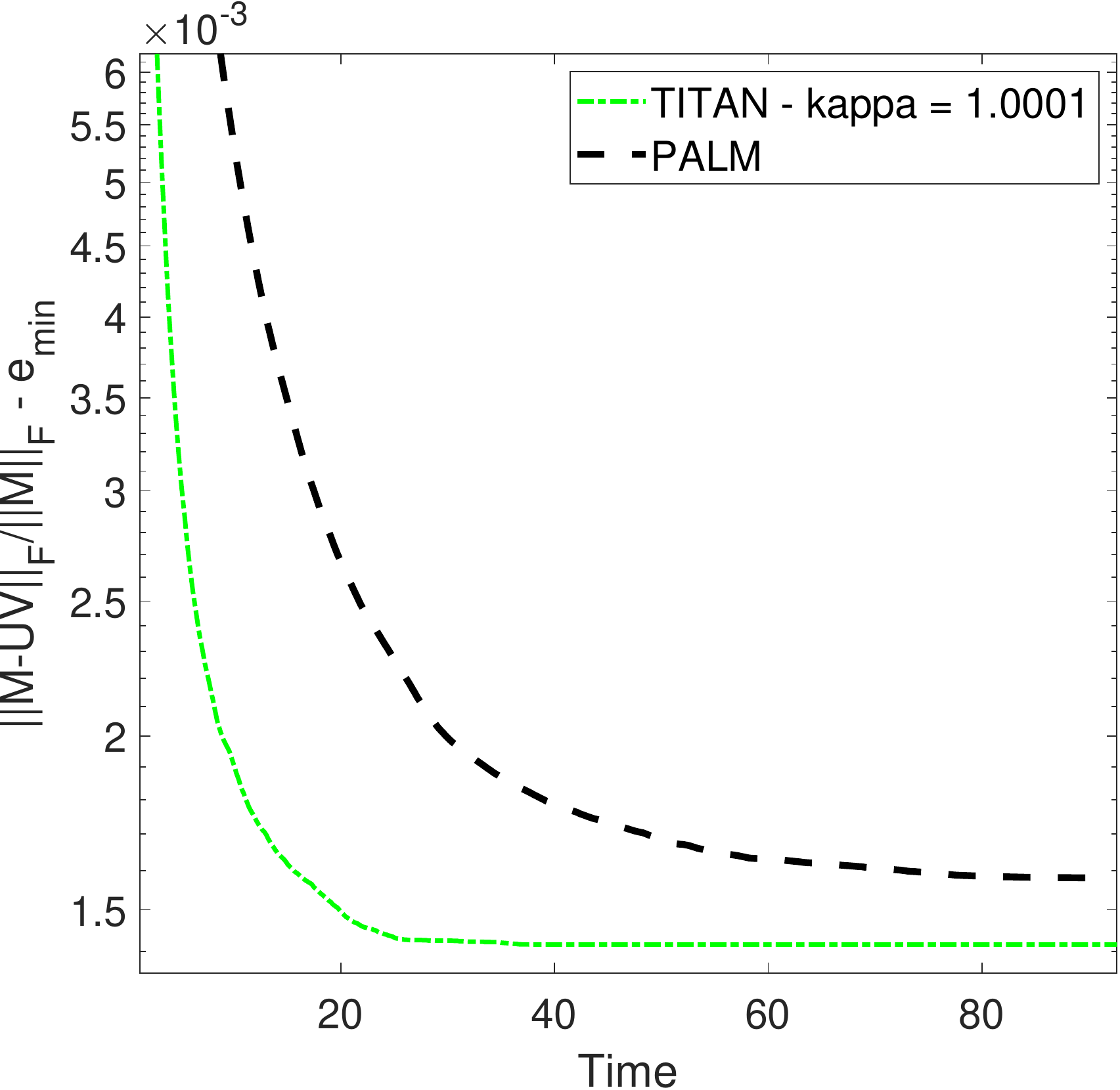}\\
Umist & ORL\\
\includegraphics[width=0.433\textwidth]{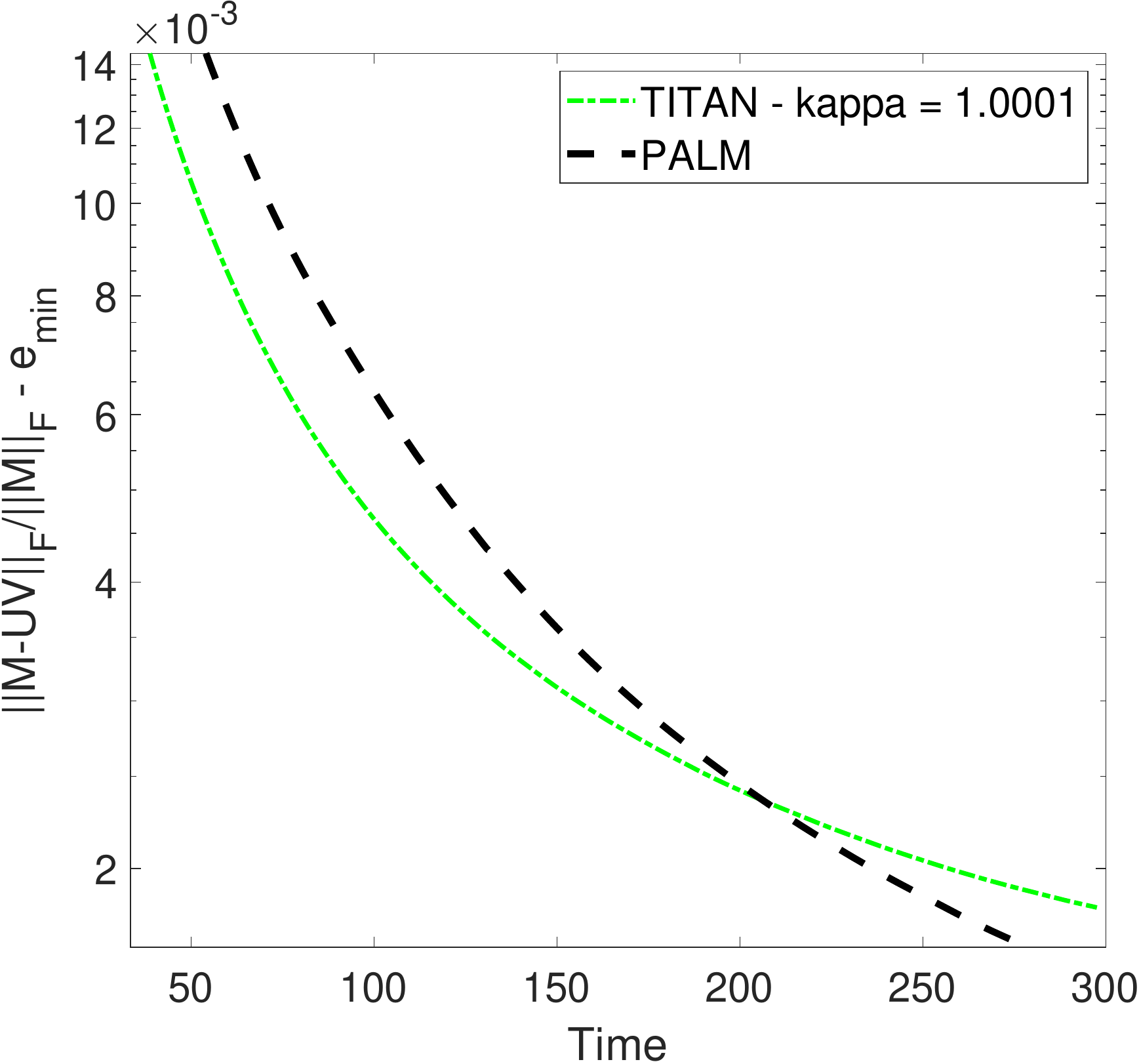}  & 
\includegraphics[width=0.425\textwidth]{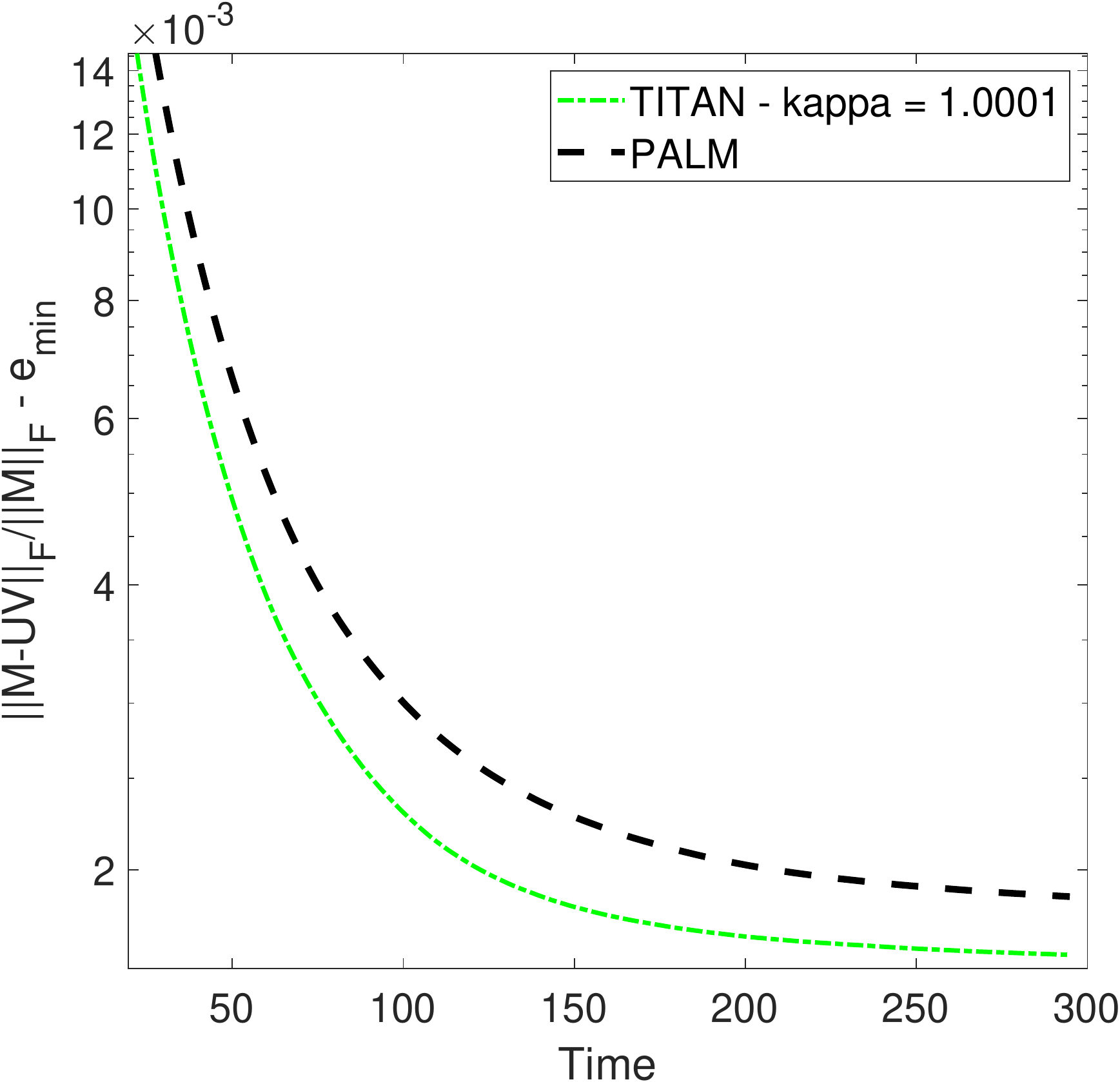}
\end{tabular}
\caption{TITAN and PALM applied on sparse NMF.  
The plots show the evolution of the average value of $E(k)$ with respect to time on the image data sets. 
\label{fig:sparsenmf} } 
\end{center}
\end{figure*}

\begin{table}[tbh!]
\centering
\caption{Average and std of relative errors obtained by TITAN and PALM applied on sparse NMF~\eqref{sparseNMF}. 
Bold values correspond to the best results for each data set. 
\label{tab:sparsenmf}}
\begin{tabular}{@{}cll@{}}
\toprule
Data set      & Method &    mean $\pm$ std \\ 
\midrule
   \multirow{2}{*}{Frey}   & PALM &  $ \mathbf{1.4901 \, 10^{-1} \pm 1.0342\, 10^{-3}}$   \\
  & TITAN - $\kappa=1.0001$ &  $1.4939\, 10^{-1} \pm 1.0448\, 10^{-3}$  \\
  
  \midrule
  \multirow{2}{*}{cbclim} & PALM  & $1.1955\, 10^{-1} \pm 7.4322\, 10^{-4}$ \\
  & TITAN - $\kappa=1.0001$ &  $\mathbf{1.1939\, 10^{-1} \pm 7.1868	\, 10^{-4}}$   \\
  
  \midrule
  \multirow{2}{*}{Umist} & PALM  & $\mathbf{ 1.2002 \ 10^{-1} \pm 8.1340\, 10^{-4}}$ \\
  & TITAN - $\kappa=1.0001$ &  $1.2031 \, 10^{-1} \pm 9.3527	\, 10^{-4}$   \\

  \midrule
  \multirow{2}{*}{ORL} & PALM  & $1.9108\, 10^{-1} \pm 6.5507\, 10^{-4}$ \\
  & TITAN - $\kappa=1.0001$ &  $ \mathbf{1.9084 \, 10^{-1} \pm 8.4325	\, 10^{-4}}$   \\
  
  \hline
\end{tabular}
\end{table}

We observe that TITAN - $\kappa=1.0001$ converges initially faster than PALM for all data sets.  
In term of the accuracy of the final solutions, TITAN - $\kappa=1.0001$ provides better relative errors on average for the CBCL and ORL data sets, while PALM for the Frey and Umist data sets. This is expected since sparse NMF is a hard nonconvex problem, and hence different algorithms converge towards different critical points with different objective function values (even if they are initialized with the same solution).

\paragraph{Sparse document data sets}  In the second experiment, we test the two algorithms on six sparse document data sets: classic, sports, reviews, hitech, k1b and tr11, see~\cite{ZG05}. We choose $\mathbf r=25$, $s= 0.25\mbfr$, 
run all algorithms 20 times, use the same random initialization for all algorithms in each run, and run each algorithm for 100 seconds.  Figure~\ref{fig:document} reports the evolution with respect to time of the average values of $E(k)$. Table~\ref{tab:document}  reports the average and the standard deviation (std) of the relative errors. 

\begin{figure*}[ht]
\begin{center}
\begin{tabular}{cc}
classic & sports\\
\includegraphics[width=0.35\textwidth]{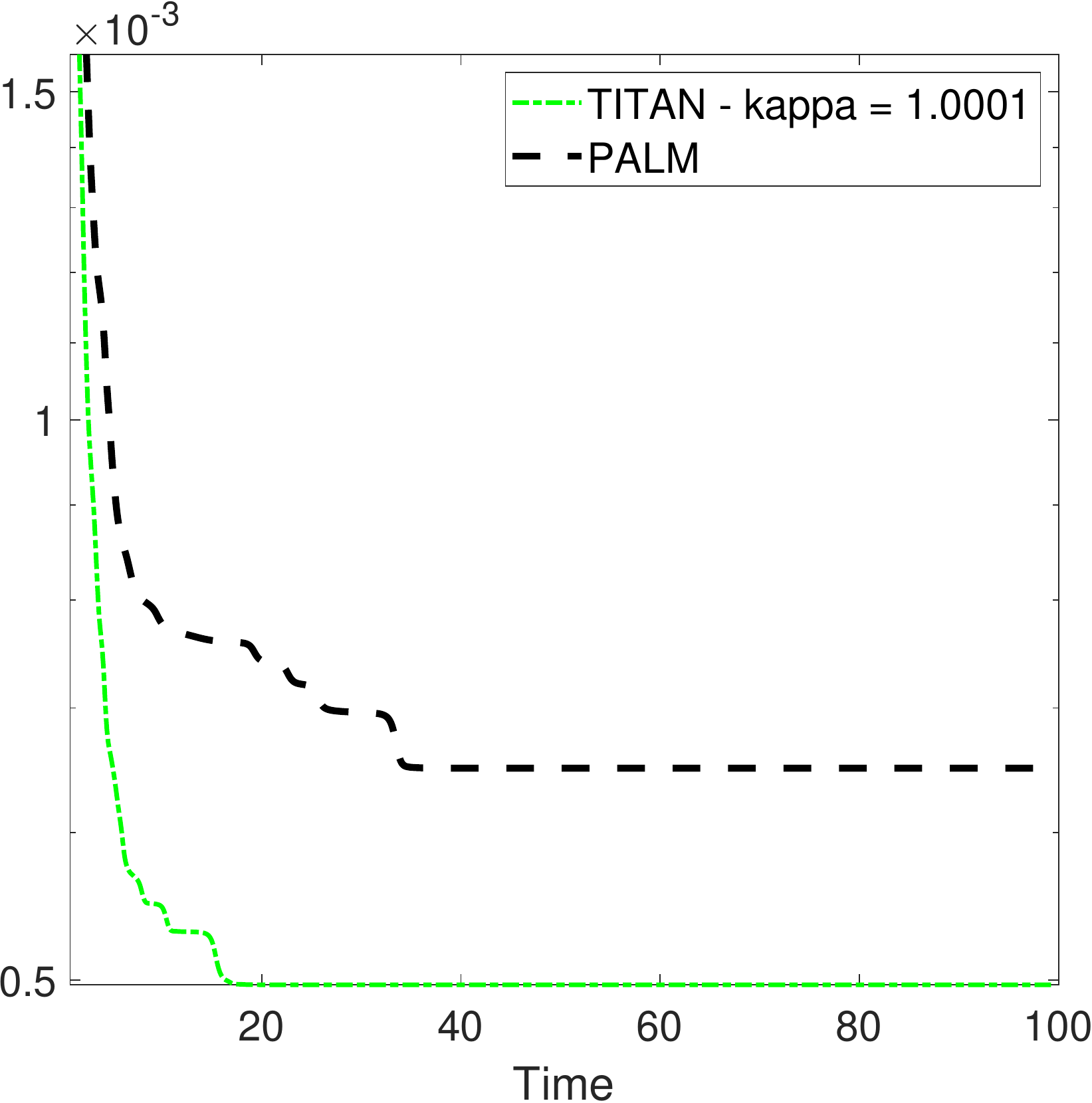}  & 
\includegraphics[width=0.375\textwidth]{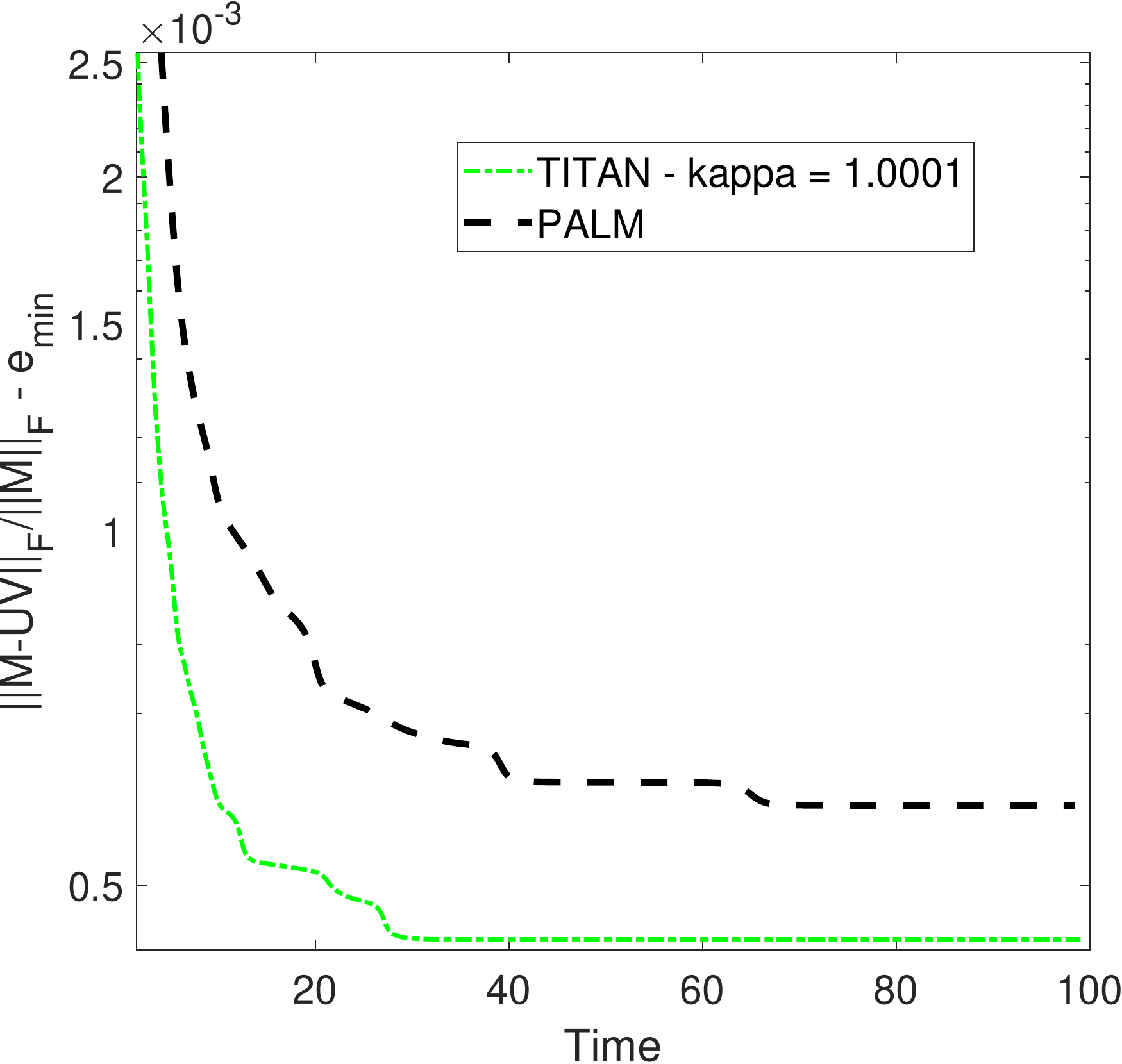}
\\
reviews & hitech\\
\includegraphics[width=0.37\textwidth]{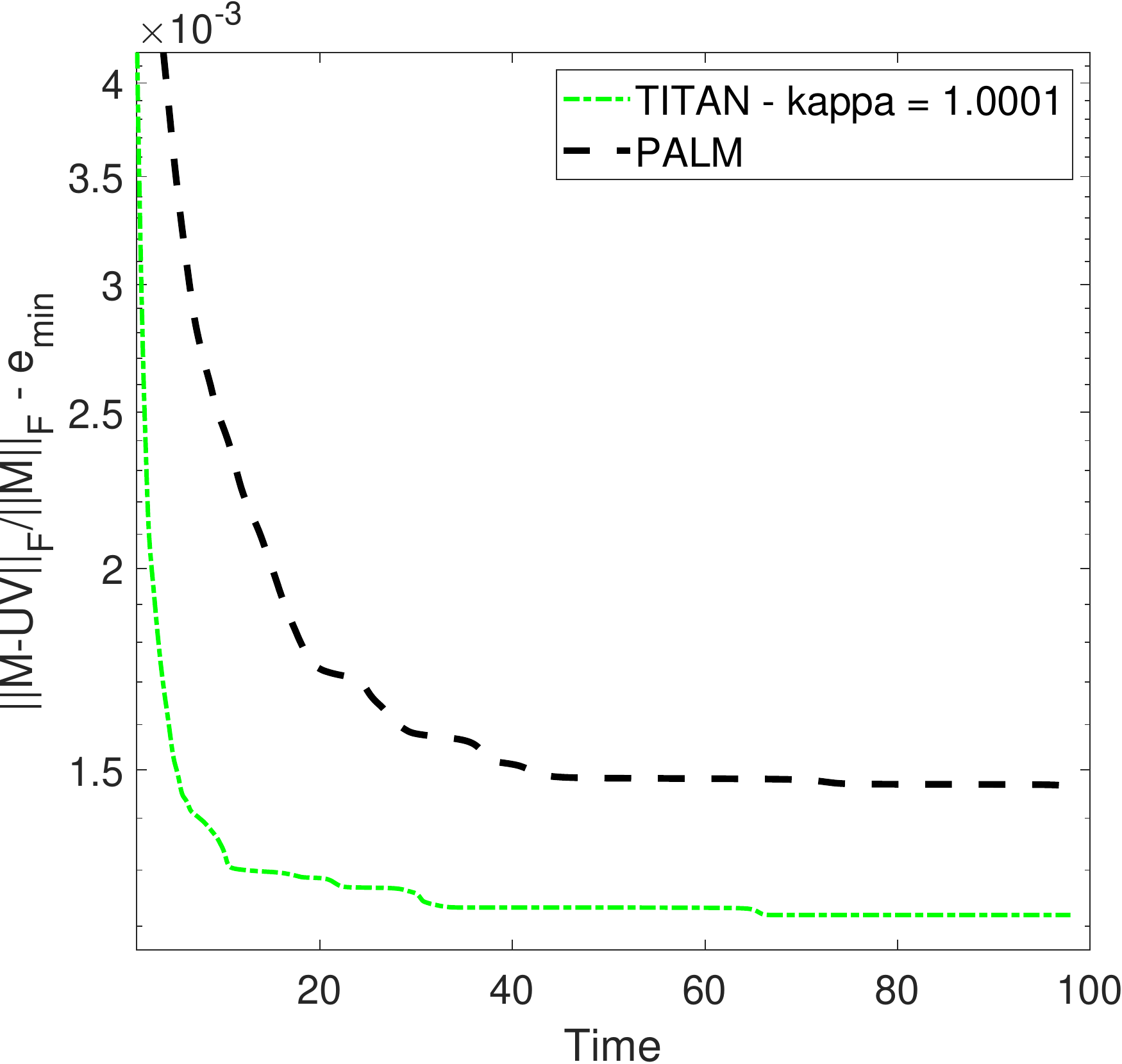}  & 
\includegraphics[width=0.375\textwidth]{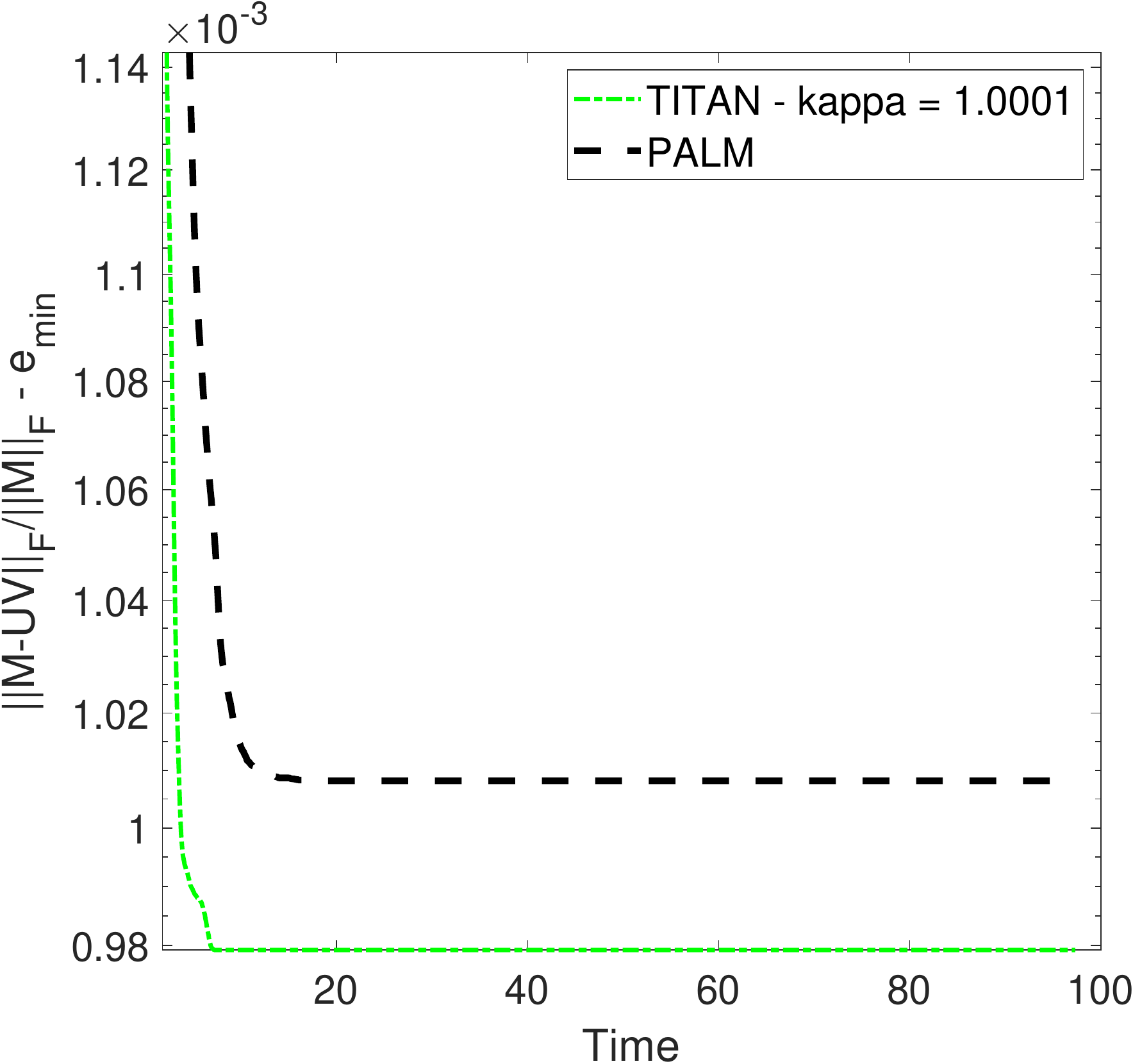}
\\
k1b & tr11\\
\includegraphics[width=0.36\textwidth]{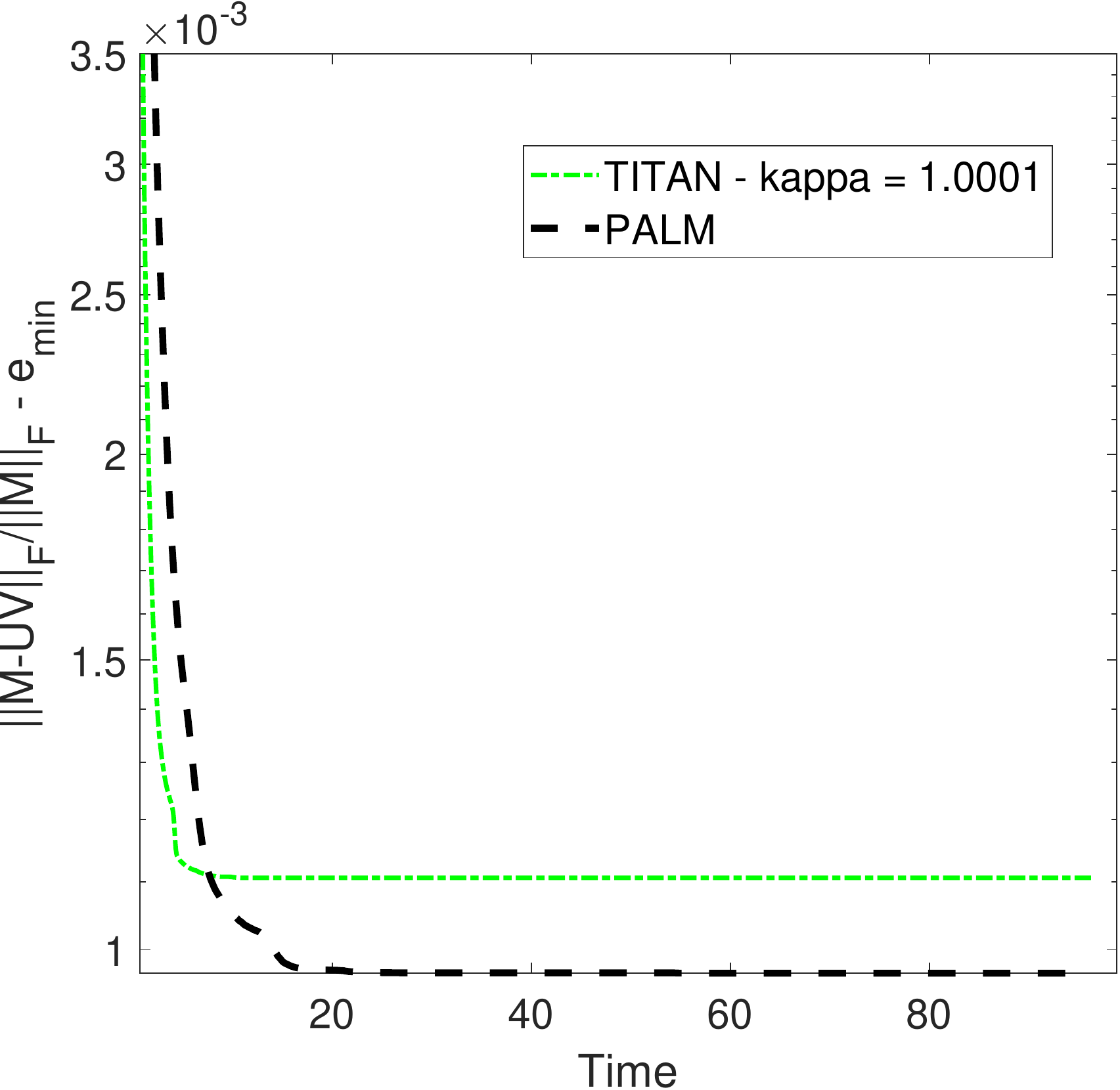}  & 
\includegraphics[width=0.355\textwidth]{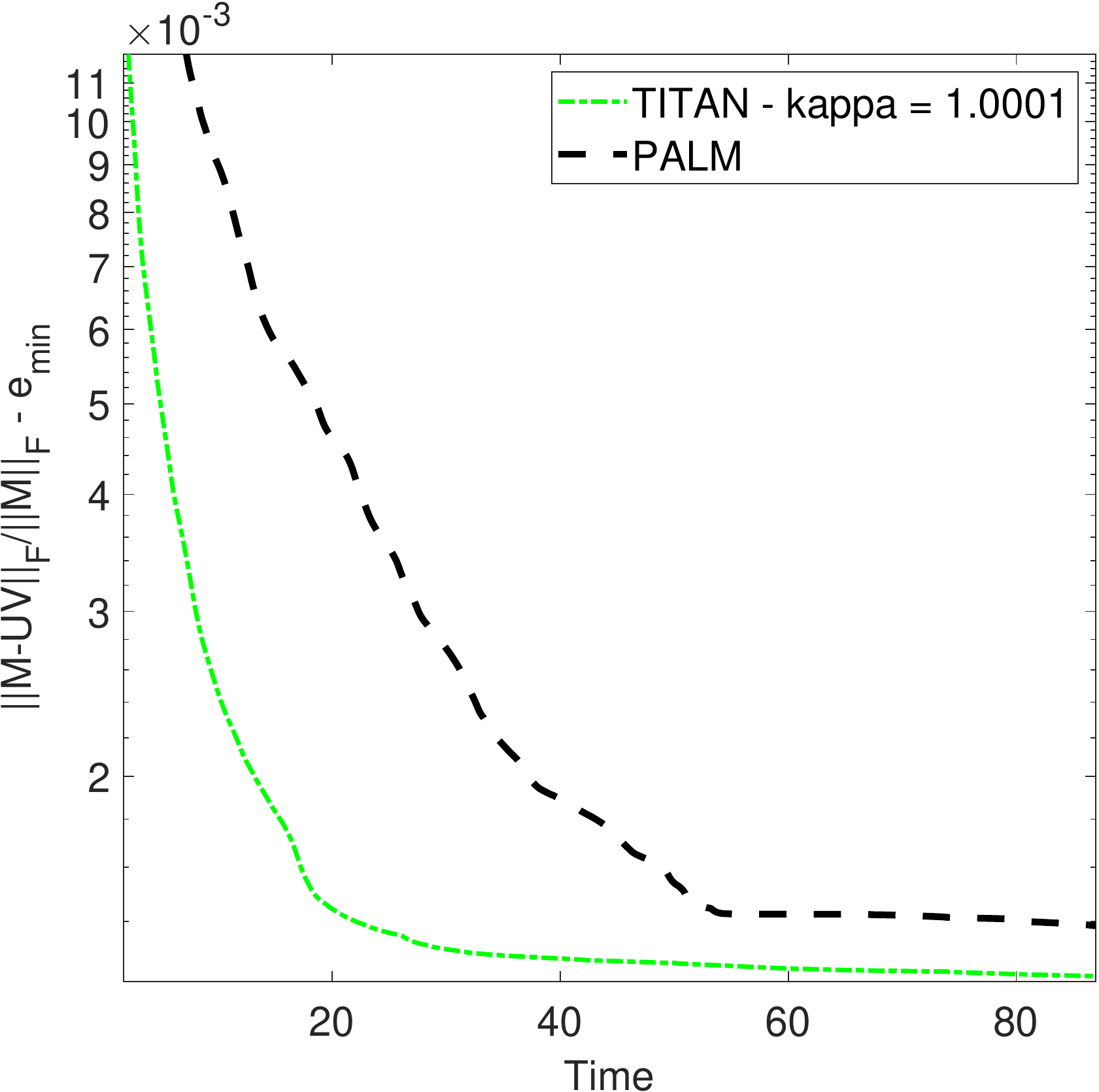}
\end{tabular}
\caption{TITAN and PALM applied on sparse NMF.  
The plots show the evolution of the average value of $E(k)$ with respect to time on the sparse document data sets. 
\label{fig:document}} 
\end{center}
\end{figure*}

\begin{table}[tbh!]
\centering
\caption{Average and std of relative errors obtained by TITAN and PALM applied on sparse NMF~\eqref{sparseNMF}. 
Bold values correspond to the best results for each data set. 
\label{tab:document}}
\begin{tabular}{@{}cll@{}}
\toprule
Data set      & Method &    mean $\pm$ std \\ 
\midrule
   \multirow{2}{*}{classic}   & PALM &  $ 8.9160 \, 10^{-1} \pm 7.4522\, 10^{-4}$   \\
  & TITAN - $\kappa=1.0001$ &  $\mathbf{8.9145\, 10^{-1} \pm 3.1633\, 10^{-4}}$  \\
  
  \midrule
  \multirow{2}{*}{sports} & PALM  & $8.1190\, 10^{-1} \pm 4.3938\, 10^{-4}$ \\
  & TITAN - $\kappa=1.0001$ &  $\mathbf{8.1177\, 10^{-1} \pm 2.9569	\, 10^{-4}}$   \\
  
  \midrule
  \multirow{2}{*}{reviews} & PALM  & $ 8.0803 \ 10^{-1} \pm 5.6695\, 10^{-4}$ \\
  & TITAN - $\kappa=1.0001$ &  $\mathbf{8.0779 \, 10^{-1} \pm 7.0906	\, 10^{-4}}$   \\
  
  \midrule
  \multirow{2}{*}{hitech} & PALM  & $8.6305\, 10^{-1} \pm 5.5024\, 10^{-4}$ \\
  & TITAN - $\kappa=1.0001$ &  $ \mathbf{8.6302 \, 10^{-1} \pm 6.2594	\, 10^{-4}}$   \\
  
   \midrule
  \multirow{2}{*}{k1b} & PALM  & $ \mathbf{8.1829 \ 10^{-1} \pm 6.1890\, 10^{-4}}$ \\
  & TITAN - $\kappa=1.0001$ &  $8.1842 \, 10^{-1} \pm 7.5261	\, 10^{-4}$   \\
  
   \midrule
  \multirow{2}{*}{tr11} & PALM  & $ 1.4768 \ 10^{-1} \pm 7.4810\, 10^{-4}$ \\
  & TITAN - $\kappa=1.0001$ &  $\mathbf{1.4752 \, 10^{-1} \pm 5.3136	\, 10^{-4}}$   \\
  
  \hline
\end{tabular}
\end{table}

We again observe that TITAN - $\kappa=1.0001$ converges on average faster than PALM in all data sets.  In terms of the relative errors of the final solutions computed within the allotted time, 
TITAN - $\kappa=1.0001$ performs on average better than PALM, except for the  k1b data set.

\subsection{Matrix Completion Problem}
\label{sec:MCP}

In this section, we illustrate the advantages of using block surrogate functions by deploying TITAN for the MCP~\eqref{MF}, as explained in Section~\ref{ex:composite_surrogate}. 
As for sparse NMF, we use two blocks of variables, $x_1 = U$ and $x_2 = V$. 
Since $\psi(U,V)$ is continuously differentiable and $\mathcal R(U,V)$ is a block separable function, $F$ (in this case $F=f$) satisfies the condition in~\eqref{partialF}.   Moreover, $\phi$ is block-wise concave and differentiable with Lipschitz gradient on $\mathbb R^{\mbfm\times\mbfn}_+$. Hence, we select the composite surrogate function for $f = \psi + \phi\circ r$ as in Section~\ref{ex:composite_surrogate}, in which we will choose block surrogate functions for $\psi$ as follows. 
Since $\nabla_U \psi(U,V) = -\mathcal P(A -UV)V^T$ and 
$\nabla_V \psi(U,V) = -U^T\mathcal P(A -UV)$ are 
 Lipschitz continuous with constants $L_1=\|VV^T\|$ and $L_2 = \|U^TU\|$, respectively, we choose the block surrogate functions $u_i^\psi$, $i=1,2$, for $\psi$ to be the block Lipschitz gradient surrogate functions as in Section~\ref{ex:Lipschitz_surrogate}. 
  Assumption \ref{assump:surrogate_assum} is then satisfied; see Section~\ref{sec:compositeTITAN}. 
  
 Let us choose the Nesterov-type acceleration.  
 The update in~\eqref{eq:iMM_update} for $U$ is 
\begin{equation}\label{sub:mc}
\begin{split}
   U^{k+1} \in \argmin_U\iprod{\nabla_U \psi(\bar U^k,V^k)}{U}  + \frac{L_1^k}{2} \|U-\bar U^k\|^2 + \langle \nabla_U\phi(r(U^k,V^k)), |U|\rangle,
\end{split}
\end{equation}
where $\nabla_U\phi(r(U^k,V^k)) = \lambda\theta\left(\exp(-\theta\|u_{ij}^k\|)\right)$,  $L_1^k=\|V^k (V^k)^T\|$, $\bar U^k=U^k + \beta_1^k(U^k-U^{k-1})$.
The solution of  \eqref{sub:mc} is given by
\begin{equation}
    U^{k+1} = \mathcal{S}_{1/L_1^k}\left(P^k,\nabla_U\phi\left(r\left(U^k,V^k\right)\right)\right),
\end{equation}
where $P^k = \bar U^k - \frac{1}{L_1^k}\nabla_U \psi(\bar U^k,V^k)$,  and $\mathcal{S}_\tau$ is the soft-thresholding operator with parameter $\tau$, that is, 
\begin{equation}
    \mathcal{S}_\tau(P,W)_{ij} = [|p_{ij}|-\tau w_{ij}]_+\text{sign}(p_{ij}).
\end{equation}
Similarly ,the  update for $V$ is given by 
\begin{equation}
    V^{k+1} = \mathcal{S}_{1/L_2^k}\left(Q^k,\nabla_V\phi\left(r\left(U^{k+1},V^k\right)\right)\right),
\end{equation}
where $L_2^k=\|(U^{k+1})^T U^{k+1} \|$, $Q^k = \bar V^k - \frac{1}{L_2^k}\nabla_V \psi(U^{k+1},\bar V^k)$ and $\bar V^k = V^k +\beta^k_2(V^k - V^{k-1})$.

Let us now determine $\eta_i^k $ and $\gamma_i^k$, for $i=1,2$, of Condition~\eqref{parameter}. 
Note that $x_i\mapsto  \langle\nabla_i\phi(r(y)), r_i(x_i) \rangle$ are convex. 
Furthermore, $\psi(U,V)$ is a block-wise convex function. Therefore, it follows from Remark~\ref{remark:composite} that we can take $\eta_i^k$ and $\gamma_i^k$ as in~\eqref{eq:gammastrong}. Note that $\tau_i^k=\beta_i^k$, since we choose Nesterov-type acceleration.  Condition~\eqref{parameter} becomes
$\beta_i^k \leq \sqrt{{C L_i^{k-1}}/{L_i^k}},
$
where $0<C<1$.  In our experiments, we choose 
\begin{equation}\label{beta-mcp}
\begin{array}{ll}
C=0.9999^2, \mu_0=1, \mu_k=\frac12(1+\sqrt{1+4\mu_{k-1}^2}), \\
\beta_i^k =\min\Big\{\frac{\mu_k-1}{\mu_k},\sqrt{{C L_i^{k-1}}/{ L_i^{k}}}\Big\}.
\end{array}
\end{equation}
{\color{black}We compare three algorithms: 
(1)~TITAN without extrapolation, that is, $\beta_i^k=0$ for all $k$, which is denoted by TITAN-NO, 
(2)~TITAN with extrapolation, that is, $\beta_i^k$ are  chosen as in \eqref{beta-mcp}, which is denoted by TITAN-EXTRA}, and (3)~PALM that alternatively updates $U$ and $V$ 
by solving the following sub-problems 
$$
\begin{array}{ll}
        \min\limits_U\iprod{\nabla_U \psi( U^k,V^k)}{U}  + \frac{L_1(V^k)}{2} \|U- U^k\|^2 + \lambda\sum_{ij}\Big(1-\exp(-\theta |u_{ij}|)\Big)         , \\
     \min\limits_V\iprod{\nabla_V \psi( U^{k+1},V^k)}{V}  + \frac{L_2(U^{k+1})}{2} \|V- V^k\|^2 + \lambda\sum_{ij}\Big(1-\exp(-\theta |v_{ij}|)\Big).
   \end{array}
   $$
These sub-problems can be separated into one-dimensional nonconvex problems
\begin{equation}
    \min_{x\in\mathbb R}\frac{1}{2}\|x-v\|^2 - \gamma\exp(-\theta|x|).
\end{equation}
Although the solutions to these subproblems can be computed via the Lambert $W$ function \citep{Corless_1996}, it does not have a closed-form solution. 
 \revise{To the best of our knowledge, TITAN is the only framework that allows to use  extrapolation while having closed-form updates to solve this particular matrix completion formulation.} 

In our experiments, {\color{black}all the algorithms start from the same initial point $(U^0, V^0)$, where $ U^0$ is an $\mbfm\times \mbfr$ orthogonal matrix whose range approximates the range of $\mathcal P(A)$, which is computed by a power method \cite[Algorithm 4.1]{Halko2011}  with $\mbfr$ iterations and a  tolerance $10^{-6}$. The initial matrix $V^0$ is determined by $V^0 = \mathbf V^T$ with $U\Sigma V^T$ being the singular value decomposition of $(U^0)^T\mathcal P(A)$, i.e.,  $U\Sigma V^T = ( U^0)^T\mathcal P(A)$.} We choose $\lambda = 0.1$ and $\theta = 5$. We note that we do not optimize numerical results by tweaking the parameters as this is beyond the scope of this work. Rather, we simply chose the parameters that are typically used in the literature, see,  e.g., \cite{brafea}. 
It is important noting that we evaluate the algorithms on the same models. 
We carried out the experiments on the two most widely used data sets in the field of recommendation systems, MovieLens and Netflix, which contain ratings of different users. 
The characteristics of the data sets are given in Table~\ref{dataset}.
We respectively choose $\mbfr = 5, 8$, and $13$ for MovieLens 1M, 10M, and Netflix data set. We  randomly picked 70\% of the observed ratings for training and the rest for testing. The process was repeated twenty times. We run each algorithm 20, 200, and 3600 seconds for MovieLens 1M, 10M, and Netflix data sets, 
respectively. 
We are interested in the root mean squared error on the test set: $RMSE = \sqrt{\|\mathcal P_T(A - UV)\|^2/N_T}$, where $\mathcal P_T(Z)_{ij} = Z_{ij}$ if $A_{ij}$ belongs to the test set and $0$ otherwise, $N_T$ is the number of ratings in the test set.
 We plotted the curves of the average value of RMSE and the objective function value 
 versus training time in Figure~\ref{fig:MCP}, and report the average and the standard deviation of the RMSE and the objective function value in Table~\ref{results}. 
\begin{figure*}[ht!]
\begin{center}
\begin{tabular}{cc}
\includegraphics[width=0.418\textwidth]{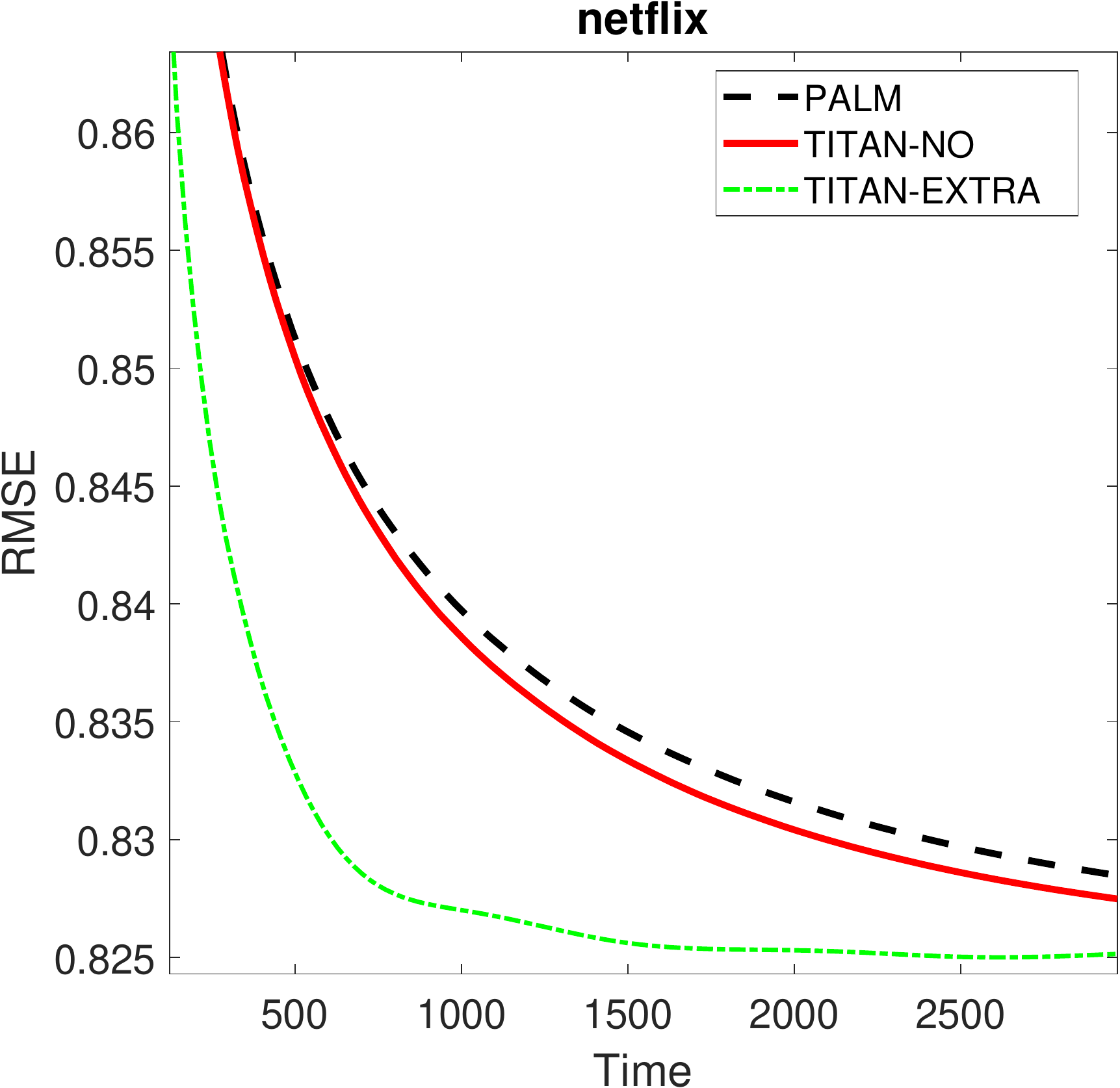}  & 
\includegraphics[width=0.41\textwidth]{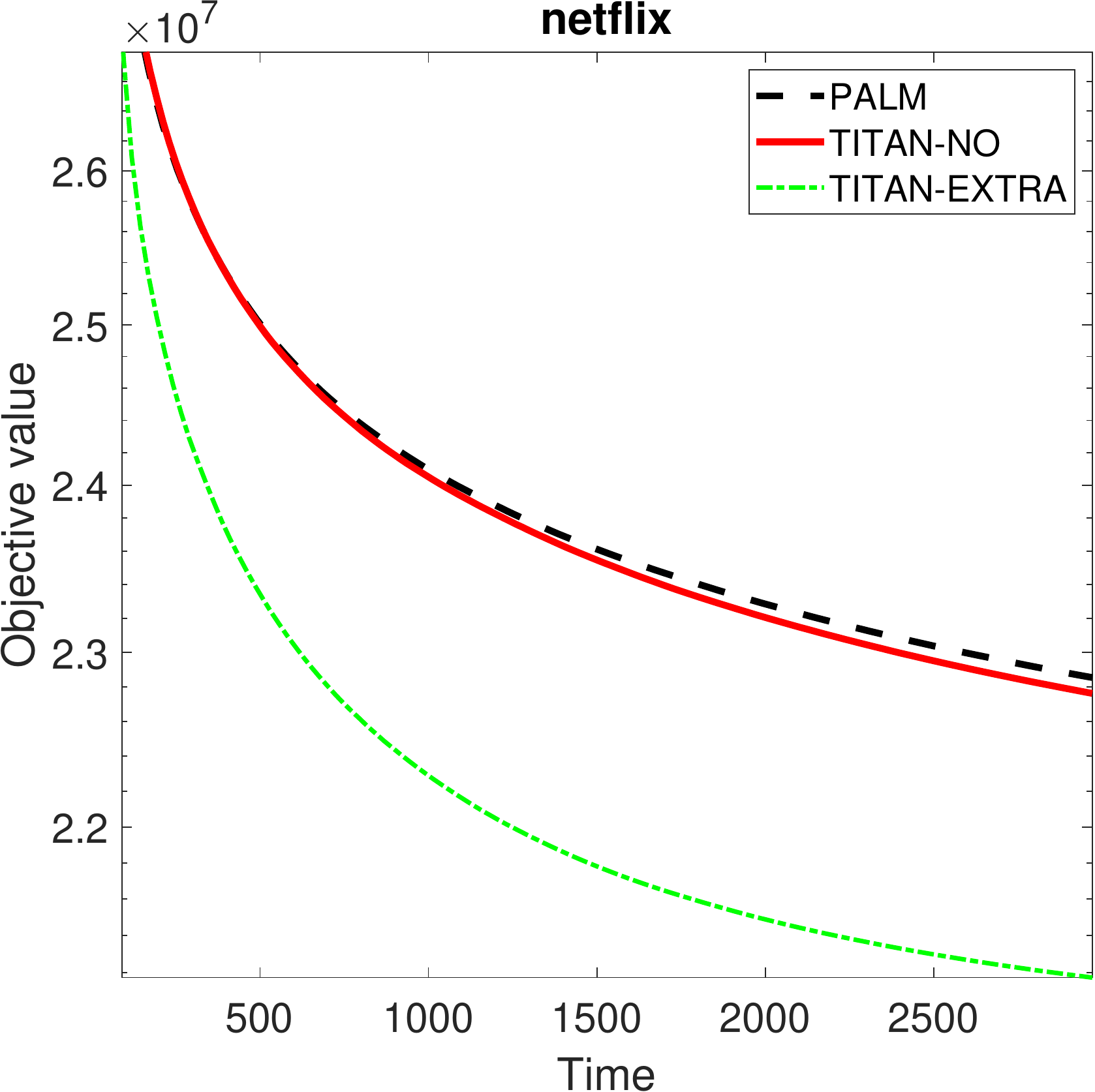} \\
\includegraphics[width=0.413\textwidth]{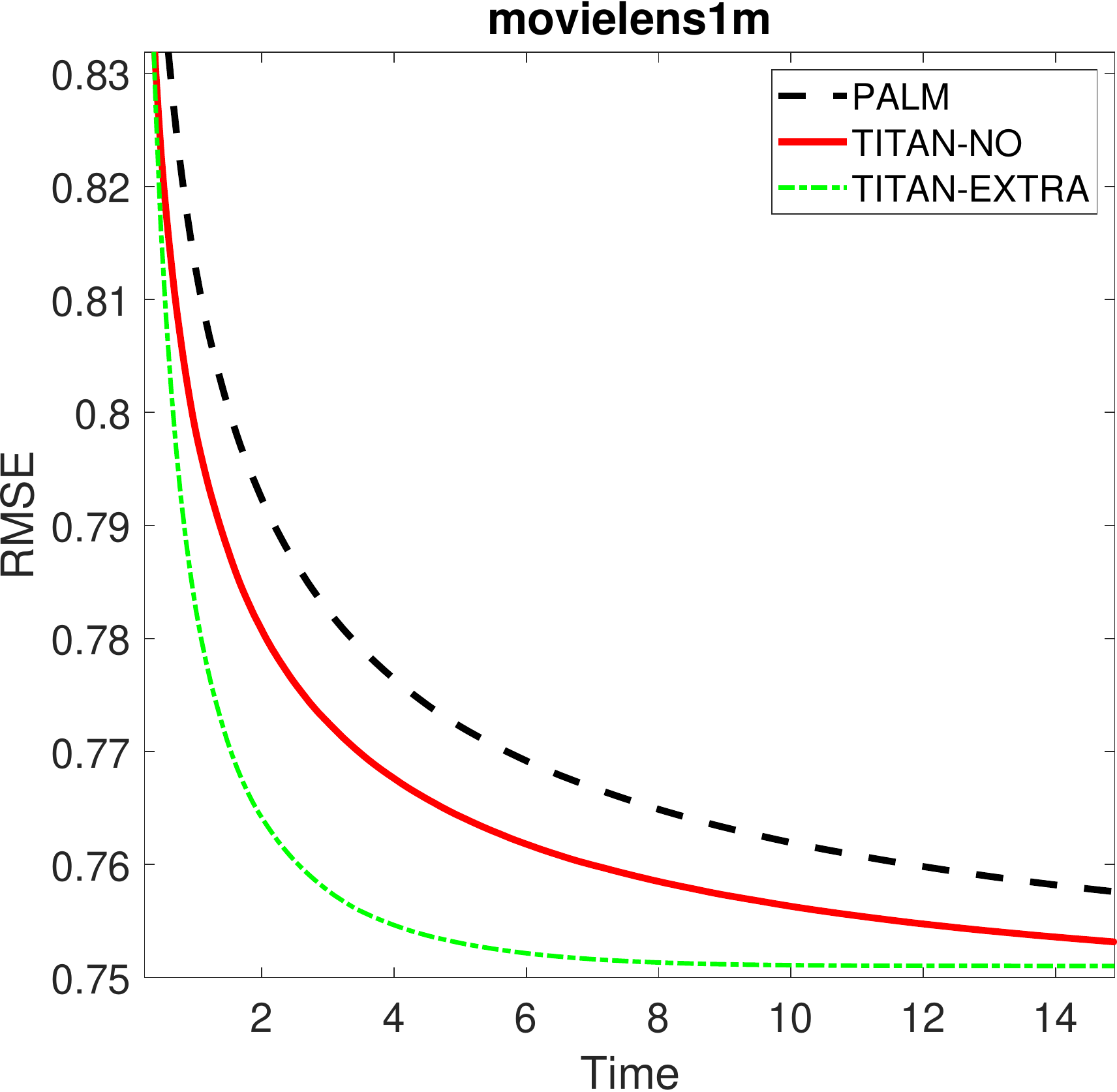}  & 
\includegraphics[width=0.413\textwidth]{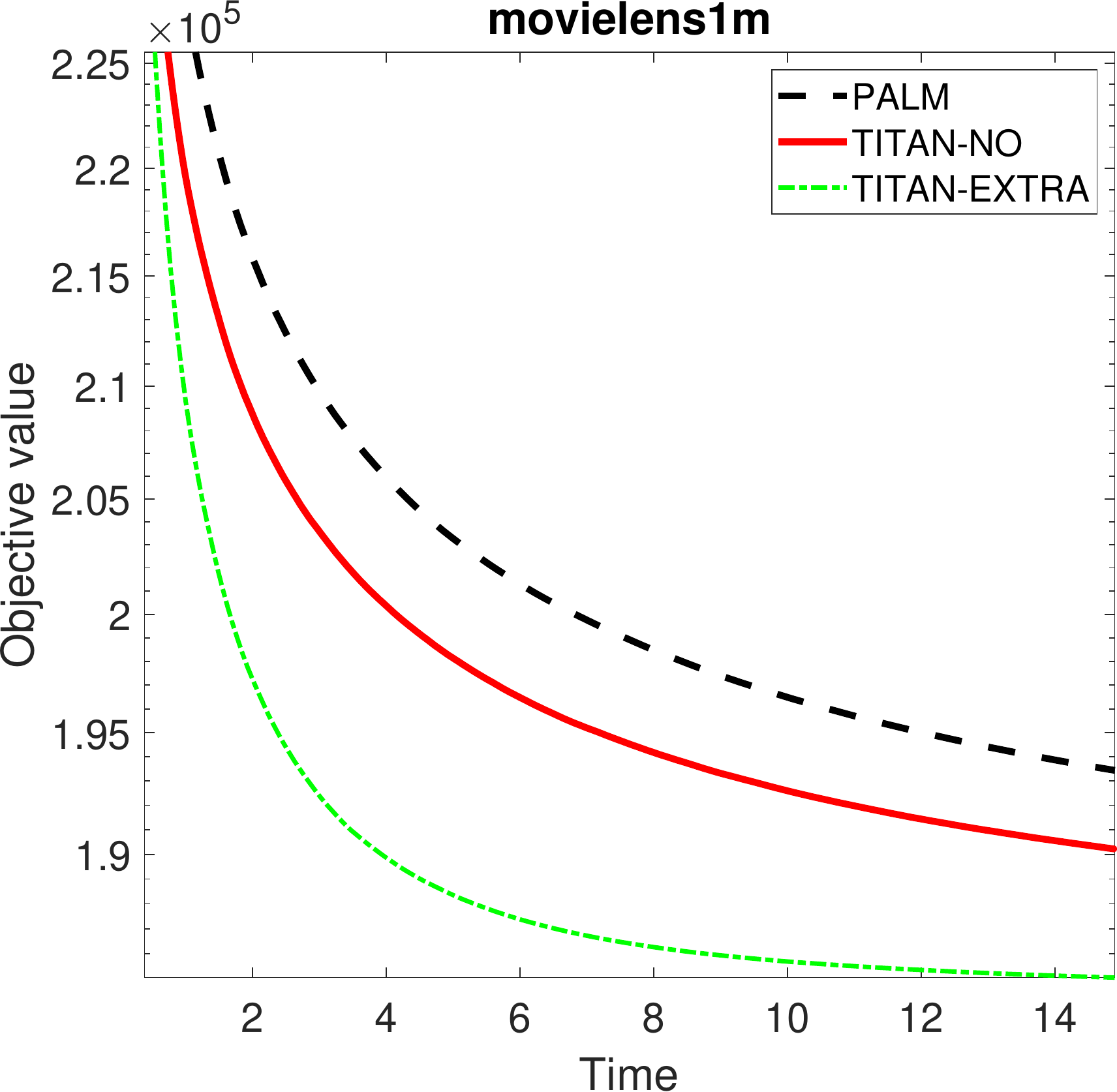} \\
\includegraphics[width=0.415\textwidth]{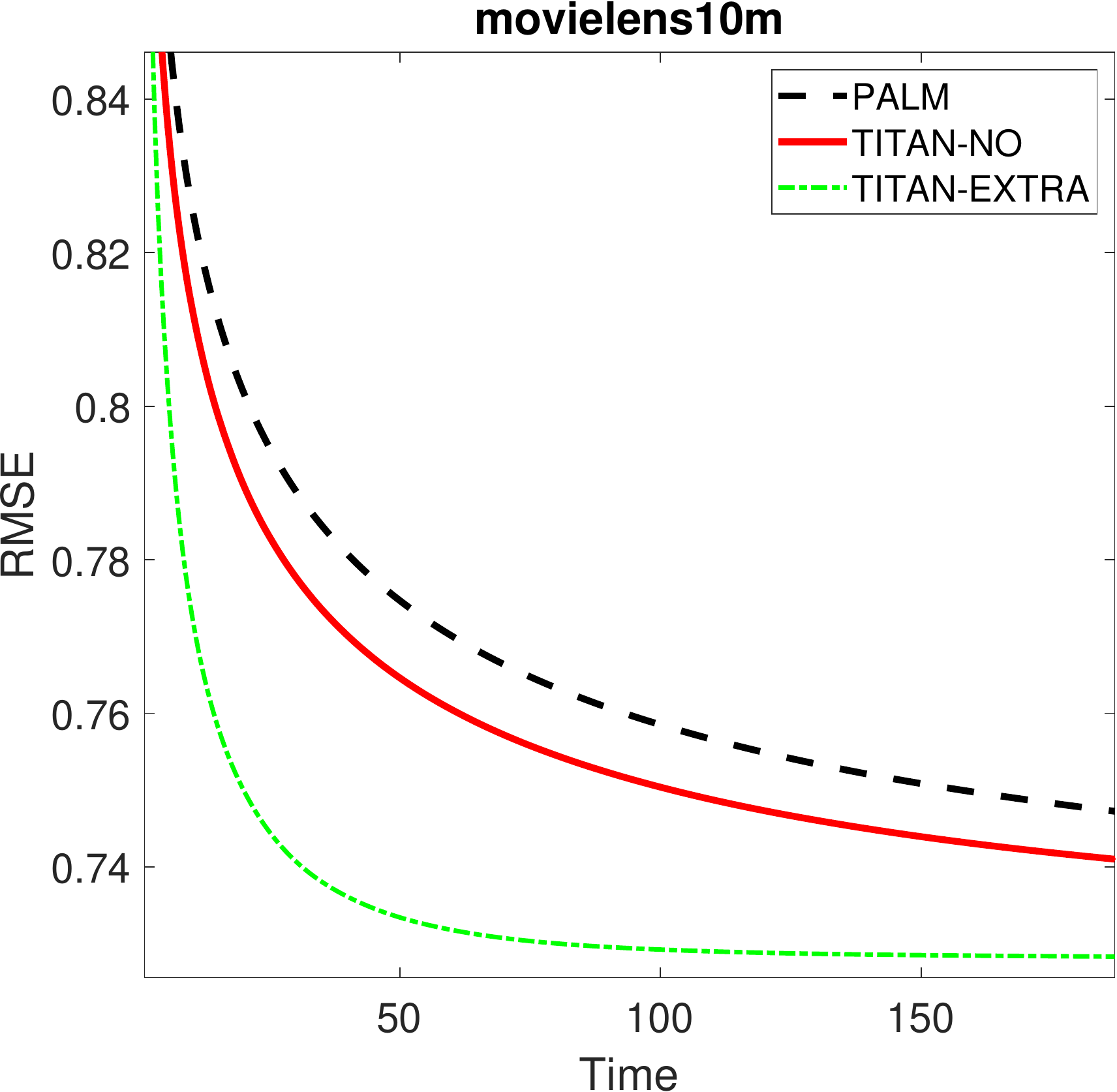}  & 
\includegraphics[width=0.409\textwidth]{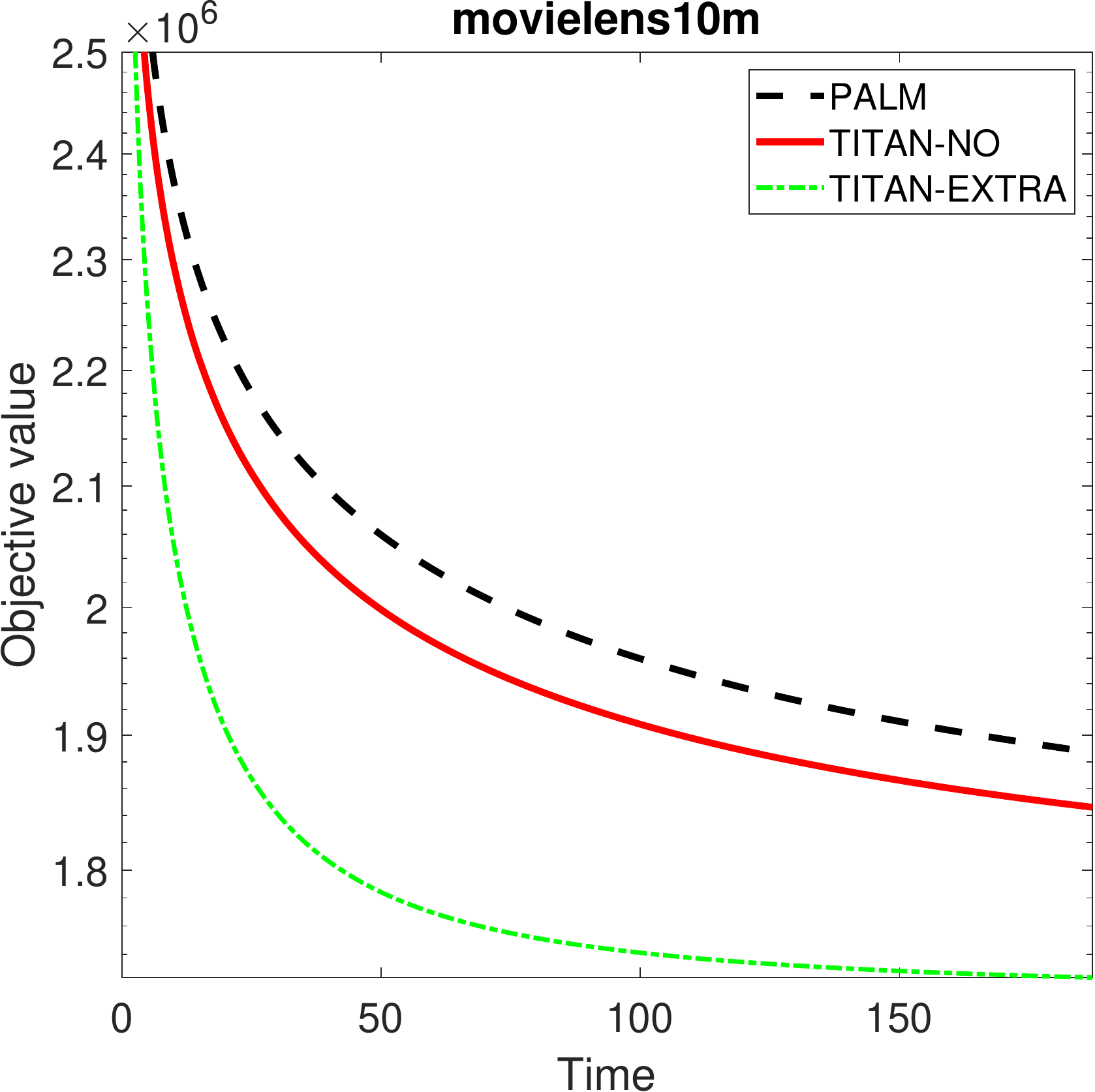}
\end{tabular}
\caption{TITAN and PALM applied on the MCP~\eqref{MF}. Evolution of the average value of the RMSE on the test set and the objective function value with respect to time. 
\label{fig:MCP}} 
\end{center}
\end{figure*}

\begin{table}[tbh!]
\centering
\caption{The number of users, items, and ratings used in each data set.}\label{dataset}
\begin{tabular}{@{}cllll@{}}
\toprule
\multicolumn{2}{l}{Data set}       & \multicolumn{1}{l}{\#users} & \multicolumn{1}{l}{\#items} & \multicolumn{1}{l}{\#ratings} \\ \midrule
\multirow{2}{*}{MovieLens}  & 1M  & 6,040                       & 3,449                       & 999,714                       \\
                            & 10M & 69,878                      & 10,677                      & 10,000,054                    \\
\multicolumn{1}{l}{Netflix} &     & 480,189                     & 17,770                      & 100,480,507                   \\ \bottomrule
\end{tabular}
\end{table}


\begin{table}[tbh!]
\centering
\caption{Comparison of TITAN and PALM applied on the MCP~\eqref{MF}: RMSE and final objective function values obtained within the allotted time. 
Bold values indicate the best results for each data set.}\label{results}
\begin{tabular}{@{}clll@{}}
\toprule
\multirow{2}{*}{Data set}      & \multirow{2}{*}{Method} & RMSE & Objective value \\
& & mean $\pm$ std & (mean $\pm$ std)$\times 10^{-5}$ \\ \midrule
 
  \multirow{3}{*}{MovieLens 1M}   & PALM &  0.7550 $\pm$   0.0016  & 1.9155 $\pm$   0.0088 \\
  & TITAN-NO &  0.7514  $\pm$  0.0013  & 1.8879 $\pm$   0.0066\\
  
 & TITAN-EXTRA  & \textbf{0.7509}  $\pm$  0.0008 & \textbf{1.8483} $\pm$   0.0038\\
  
  \midrule
  \multirow{3}{*}{MovieLens 10M} & PALM  & 0.7462    $\pm$ 0.0006  &  18.8038  $\pm$  0.0348\\
  & TITAN-NO  & 0.7402  $\pm$  0.0006  & 18.4027 $\pm$   0.0375\\
  
  & TITAN-EXTRA & \textbf{0.7283}  $\pm$  0.0005 & \textbf{17.2277}   $\pm$ 0.0236\\
\midrule
  \multirow{3}{*}{Netflix} & PALM & 0.8274 $\pm$   0.0006  &  226.4846  $\pm$  1.1898\\
  & TITAN-NO  & 0.8265  $\pm$  0.0006 & 225.4806 $\pm$   1.1808\\
  
  & TITAN-EXTRA & \textbf{0.8250} $\pm$ 0.0004 & \textbf{210.4999}  $\pm$  0.3569\\
\hline
\end{tabular}
\vspace{-0.1in}
\end{table}

We observe that TITAN-EXTRA converges the fastest on all the data sets, providing a significant acceleration of TITAN-NO:  as shown on Table~\ref{tab:leadtime}, 
TITAN-EXTRA is at least 4 times faster than TITAN-NO on the three data sets.  
TITAN-EXTRA achieves not only the best final objective function values but also the best RMSE on the test set. 
This illustrates the usefulness of the inertial terms.  
Moreover, TITAN-NO performs better PALM on the three data sets which illustrates the usefulness of properly choosing the surrogate function. Recall that TITAN-NO and TITAN-EXTRA are two new algorithms for the MCP~\eqref{MF}, which are specific instances of the TITAN framework. 
\ngi{
\begin{center}
\begin{table}[h!] 
    \begin{center} 
    \normalsize
        \begin{tabular}{l||c|c|c} 
            data set     & TITAN-EXTRA & TITAN-NO  & acceleration  \\
                        & lead time (s) & total time (s)  &  factor \\ 
            \hline 
            netflix       &   674.91  &  3000 & 4.44  \\ 
            movielens1m   &   3.8  & 15  & 3.94  \\
            movielens10m &   28.67   & 200  &  6.97 \\ 
        \end{tabular}
    \end{center} 
\caption{TITAN lead time compared to TITAN-NO to obtain the same objective function value within the allotted time. 
\label{tab:leadtime} 
}
\end{table} 
\end{center}
}

\section{Conclusion}
\label{sec:conclusion}

We have proposed and analysed TITAN, a novel inertial block majorization-minimization algorithmic framework.   
TITAN unifies many inertial block coordinate descent methods, while allowing to derive new highly efficient algorithms, as illustrated in Section~\ref{sec:MCP} on the MCP.  
We proved sub-sequential convergence of TITAN under mild assumptions and global convergence of TITAN under some stronger assumptions. We applied TITAN to sparse NMF and the MCP to illustrate the benefit of using inertial terms in BCD methods, and of using proper surrogate functions. 
Especially, the way we choose the surrogate functions and the corresponding extrapolation operators to derive TITAN-based  algorithms for the MCP illustrated the advantages of using TITAN compared to the typical proximal BCD methods. Our future research direction include the development of TITAN-based algorithms for solving specific practical problems, 
for which using typical proximal BCD methods is not efficient (in particular when a closed-form for the subproblems in each block of variables does not exist).

\begin{appendix}
\label{appendix}
 
\section{Preliminaries of nonconvex nonsmooth optimization}
 \label{sec:prelnnopt}

Let $g: \bbE\to \bbR\cup \{+\infty\} $ be a proper lower semicontinuous function.  
\begin{definition}
\label{def:dd}
\begin{itemize}
\item[(i)] For each $x\in{\rm dom}\,g,$ we denote $\hat{\partial}g(x)$ as
the Frechet subdifferential of $g$ at $x$ which contains vectors
$v\in\mathbb{E}$ satisfying 
\[
\liminf_{y\ne x,y\to x}\frac{1}{\left\Vert y-x\right\Vert }\left(g(y)-g(x)-\left\langle v,y-x\right\rangle \right)\geq 0.
\]
If $x\not\in{\rm dom}\:g,$ then we set $\hat{\partial}g(x)=\emptyset.$  
\item[(ii)] The limiting-subdifferential $\partial g(x)$ of $g$ at $x\in{\rm dom}\:g$
is defined as follows. 
\[
\partial g(x) := \left\{ v\in\mathbb{E}:\exists x^{k}\to x,\,g\big(x^{k}\big)\to g(x),\,v^{k}\in\hat{\partial}g\big(x^{k}\big),\,v^{k}\to v\right\} .
\]
\end{itemize}
\end{definition}
\begin{definition}
\label{def:type2}
We call $x^{*}\in \rm{dom}\,F$ a critical point of $F$ if $0\in\partial F\left(x^{*}\right).$ 
\end{definition}
We note that if $x^{*}$ is a local minimizer of $F$ then $x^{*}$ is a critical point of $F$. 
\begin{definition}
\label{def:KL}
A function $\phi(x)$ is said to have the KL property
at $\bar{x}\in{\rm dom}\,\partial\, \phi$ if there exists $\eta\in(0,+\infty]$,
a neighborhood $U$ of $\bar{x}$ and a concave function $\xi:[0,\eta)\to\mathbb{R}_{+}$
that is continuously differentiable on $(0,\eta)$, continuous at
$0$, $\xi(0)=0$, and $\xi'(s)>0$ for all $s\in(0,\eta),$ such that for all
$x\in U\cap[\phi(\bar{x})<\phi(x)<\phi(\bar{x})+\eta],$ we have
\begin{equation}
\label{ieq:KL}
\xi'\left(\phi(x)-\phi(\bar{x})\right) \, \dist\left(0,\partial\phi(x)\right)\geq1.
\end{equation}
$\dist\left(0,\partial\phi(x)\right)=\min\left\{ \|y\|:y\in\partial\phi(x)\right\}$.
If $\phi(x)$ has the KL property at each point of ${\rm dom}\, \partial\phi$ then $\phi$ is a KL function. 
\end{definition}
Many nonconvex nonsmooth functions in practical applications belong to the class of KL functions, for examples, real analytic functions, semi-algebraic functions, and locally strongly convex functions, \revise{see}~\cite{Bochnak1998,Bolte2014}.

\section{Global convergence recipe}
\label{sec:global_converge}

Let us recall Theorem 2 of~\cite{Hien_ICML2020}. 

\begin{theorem}\cite[Theorem 2]{Hien_ICML2020}
\label{thm:globalrecipe}
Let $\Phi: \mathbb{R}^N\to (-\infty,+\infty]$ be a proper and lower semicontinuous function which is bounded from below. Let $\mathcal{A}$ be a generic algorithm which generates a bounded  sequence $\lrbrace{z^{k}}$ by $
z^{0}\in \bbR^N$, $z^{k+1}\in \calA(z^{k})$, $k=0,1,\ldots$
Assume that there exist positive constants  $\rho_1, \rho_2$  and $\rho_3$ and a non-negative sequence  $\lrbrace{\varphi_k}_{k\in \bbN}$ such that the following conditions are satisfied: 
\begin{itemize}
\item[(B1)] 
\textbf{Sufficient decrease property}: 
\[\rho_1 \|z^{k}-z^{k+1}\|^2 \leq \rho_2 \varphi_k^2 \leq \Phi(z^{k}) - \Phi(z^{k+1}), k=0,1,\ldots  
\] 
\item[(B2)] \textbf{Boundedness of subgradient}: 
\[
\|\omega^{k+1}\|\leq \rho_3 \varphi_k, \omega^{k}\in \partial \Phi (z^{k})  \; \text{ for } \; k=0,1,\ldots
\]
\item[(B3)] \textbf{KL property}: $\Phi$ is a KL function.
\item[(B4)] \textbf{A continuity condition}: If a subsequence 
$\{z^{k_n}\}$   converges to  $\bar{z}$  then $\Phi(z^{k_n})$ converges to $\Phi(\bar{z})$ as $n$ goes to $\infty$.
\end{itemize}
Then we have $
\sum_{k=1}^\infty \varphi_k<\infty
$, 
and $\{z^{k}\}$ converges to a critical point of $\Phi$. 
\end{theorem}

\section{Technical proofs}

In this section, we provide the proofs for Theorem~\ref{thm:global_convergence} and Proposition~\ref{prop:sufficient_decrease_2}.

\subsection{Proof of Theorem~\ref{thm:global_convergence}}
\label{proof_global} 
Let $x^*$ be a limit point of $x^k$. 
 From Theorem~\ref{thm:subsequential_converge} we have $x^*$ is a critical point of $\Phi$. 
 As the generated sequence $\{x^k\}$ is assumed to be bounded, in the following, we only work on the bounded set that contains $\{x^k\}$. 

\textit{Case 1: $C <  \underline{l}/\overline{l}$}. 
Define $\Phi^\delta(x,y) :=  \Phi(x) + \sum_{i=1}^m\frac{\delta_i}{2}\|x_i - y_i\|^2.$ Let $z^k=(x^k,x^{k-1})$ and $\varphi_k^2=\frac{1}{2}\|x^{k+1}-x^k\|^2+\frac12\|x^k-x^{k-1}\|^2 $ .  We verify the conditions of Theorem \ref{thm:globalrecipe} for  $\Phi^\delta(x^k,x^{k-1})$ with $ \delta_i= (\underline{l} + C \overline{l})/2$. 

(B1) \emph{Sufficient decrease property.} From Inequality \eqref{ieq:recursive}, we have 
$$F(x^{k+1}) + \underline{l}\|x^{k+1}-x^k\|^2 \leq F(x^k) + C \overline{l} \|x^k-x^{k-1}\|^2. 
$$
Hence, $ \Phi^\delta(z^k)-\Phi^\delta(z^{k+1}) \geq (\underline{l}-C \overline l)\varphi_k^2$. 

(B2) \emph{Boundedness of subgradient.} We note that 
\begin{equation}
\label{eq:derivative_Phi}
\partial_x \Phi^\delta(x,y)=\partial \Phi(x) + [\delta_i(x_i-y_i)|_{i=1,\ldots,m}], \quad \partial_y \Phi^\delta(x,y) =[\delta_i(y_i-x_i)|_{i=1,\ldots,m}].
\end{equation} 
Writing the optimality condition for \eqref{eq:iMM_update}, we have
$$
\mathcal G^k_i(x^k_i,x^{k-1}_i)\in \partial_{x_i} \big( u_i(x_i^{k+1},x^{k,i-1}) + \mathcal I_{\mathcal X_i}(x_i^{k+1})+ g_i(x_i^{k+1})\big).
$$
Hence, by Assumption~\ref{assump:Lipschitz_ui} (i), there exist $\mathbf s_i^k \in \partial_{x_i} u_i(x_i^{k+1},x^{k,i-1})$ and $ \mathbf v_i^k \in \mathcal \partial (I_{\mathcal X_i}(x_i^{k+1})+ g_i(x_i^{k+1})) $ such that 
$$ \mathcal G^k_i(x^k_i,x^{k-1}_i) = \mathbf s^k_i + \mathbf v_i^k.
$$
As we assume Assumption \ref{assump:Lipschitz_ui} (ii) holds, there exists $\mathbf t_i^k \in \partial_{x_i} f(x^{k+1})$ such that $$\|\mathbf s_i^k- \mathbf t_i^k\| \leq B_i \| x^{k+1} - x^{k,i-1}\|.
$$
We note that $ \mathbf t_i^k + \mathbf v_i^k \in \partial_{x_i} \Phi(x^{k+1})$ by Assumption~\ref{assump:Lipschitz_ui} (i). On the other hand, 
$$\|\mathbf t_i^k + \mathbf v_i^k\| = \|\mathbf t_i^k - \mathbf s_i^k + \mathbf s_i^k +  \mathbf v_i^k \|\leq B_i \| x^{k+1} - x^{k,i-1}\| + A^k_i \| x^k_i-x^{k-1}_i\|,
$$
which implies the boundedness of the subgradient since $A^k_i$ is bounded.

(B3) \emph{KL property.} As $\Phi$ is a KL function, $\Phi^\delta$ is also a KL function. 

(B4) \emph{A continuity condition.} Suppose $z^{k_n} \to z^*$. From Proposition \ref{prop:sufficient_decrease}, we have that if $x^{k_n}$ converges to $x^*$ then  $x^{k_n-1}$ also converges to $x^*$. Hence $z^*=(x^*,x^*)$. 
On the other hand, we can derive from \eqref{eq:foundation} that, for $i\in [m]$, $u_i(x_i^{k_n}, x^{k_n-1,i-1}) + g_i(x_i^{k_n})$ converges to $u_i(x_i^*,x^*) + g_i(x_i^*)$. As we assume $u_i(\cdot,\cdot)$ is continuous we have 
$u_i(x_i^{k_n}, x^{k_n-1,i-1})$ converges to $u_i(x_i^*,x^*)=f(x^*)$. Hence,  $g_i(x_i^{k_n})\to g_i(x_i^*)$.
We then have 
$
F(x^{k_n})=f(x^{k_n}) + \sum g_i(x^{k_n}_i)  
$
converges to $F(x^*)$,  
which leads to $ \Phi^\delta(z^{k_n+1})$ converges to $\Phi^\delta(z^*)$

Applying Theorem~\ref{thm:globalrecipe}, we get $0\in \partial \Phi^\delta(x^*,x^*)$, which leads to $0\in \partial \Phi(x^*)$.

\textit{Case 2: With restart}.  We use the technique in the proof of \cite[Theorem~1]{Bolte2014}  with some modification.
 A restarting step would be taken when $F(x^{k+1}) \geq F(x^k)$. When restarting happens,  Condition~\eqref{requirement} is assumed to be satisfied with $\gamma_i^k=0$,  \revise{we thus have 
 \begin{equation}
     \label{ieq:recursive-restart}
 F(x^{k+1}) + \sum_{i=1}^m \frac{\eta_i^k}{2}\|x_i^{k+1}-x_i^k\|^2 \leq F(x^k).
\end{equation}
Hence, we have 
  \begin{equation}
     \label{ieq:recursive-general}
     F(x^{k+1}) + \sum_{i=1}^m \frac{\eta_i^k}{2}\|x_i^{k+1}-x_i^k\|^2 \leq F(x^k) + \hat C \sum_{i=1}^m \frac{\eta_i^{k-1}}{2}\|x_i^{k}-x_i^{k-1}\|^2,
\end{equation} where $\hat C=C$ in normal situation as in Inequality~\eqref{ieq:recursive} and $\hat C=0$ when restarting happens. Thus the result in Proposition~\ref{prop:sufficient_decrease} does not change. Exactly as for the proof of the continuity condition (B4) above (the first case), 
we can show that $
F(x^{k_n})\to F(x^*)$. } Since $F(x^k)$ is non-increasing we have $F(x^k)\to F(x^*)$. This also means $\Phi(x)$ is constant on the set $\Omega$ of all limit points of $x^k$. From Proposition~\ref{prop:sufficient_decrease}, we have $\|x^k -x^{k-1}\| \to 0$. Hence,   \cite[Lemma 5]{Bolte2014} yields that $\Omega$ is a compact and connected set.    

\revise{Let us recall that restarting happens when $F(x^{k+1}) \geq F(x^k)$ and when it happens Inequality~\eqref{ieq:recursive-restart} holds}. 
Therefore, as long as $x^{k+1} \ne x^k$, $F(x^k)$ is strictly decreasing (that is $F(x^{k+1}) < F(x^k)$).  Hence, if there exists an integer $\bar k$ such that $F(x^{\bar k}) = F(x^*)$ then we have $F(x^k)=F(x^*)$ and $x^k=x^{\bar k}$ for all $k\geq \bar k$. So this case is trivial. 

Let us consider $F(x^{k})> F(x^*)$ for all $k$. Then there exists a positive integer $k_0$ such that $F(x^k) < F(x^*) + \eta$ for all $k>k_0$. On the other hand, there exists a positive integer $k_1$ such that ${\rm dist} (x^k,\Omega) < \varepsilon $ for all $k>k_1$. Applying \cite[Lemma 6]{Bolte2014} we have 
\begin{equation}
\label{eq:KL}
\xi'\big(\Phi \lrpar{x^{k}}-\Phi (x^*)\big) \dist \big(0, \partial \Phi (x^{k})\big)\geq 1, \text{for any}\, k>\mathbf k:=\max\{k_0,k_1\}.
\end{equation}
 On the other hand, \revise{exactly as for Case 1 without restarting step, we can prove} that $\exists\varpi>0$ such that for some $\omega^{k+1} \in \partial \Phi(x^{k+1})$ we have 
$\| \omega^{k+1}\| \leq \varpi \varphi_k \leq \frac{\varpi}{\sqrt{2}} (\|x^{k+1} - x^{k}\| + \|x^{k} - x^{k-1}\| ).
 $ 
 Therefore, it follows from~\eqref{eq:KL} that 
 \begin{equation}
 \label{eq:KL2}
  \begin{array}{ll}
 \xi'\big(\Phi \lrpar{x^{k}}-\Phi (x^*)\big) (\|x^{k+1} - x^{k}\| + \|x^{k} - x^{k-1}\| )\geq \frac{\sqrt 2}{\varpi}
 \end{array}
 \end{equation}
 \revise{From Inequality~\eqref{ieq:recursive-general} and noting that $\bar C\leq C$,} we get 
 \begin{equation}
 \label{eq:KL3}
 \begin{array}{ll}
 \Phi(x^k) - \Phi(x^{k+1}) \geq \sum_{i=1}^m \frac{\eta_i^k}{2}\|x_i^{k+1} - x_i^k\|^2   - C  \sum_{i=1}^m  \frac{\eta_i^{k-1}}{2}\|x_i^{k} - x_i^{k-1}\|^2 
 \end{array}
\end{equation}  
 Denote 
 $ A_{i,j}=\xi (\Phi(x^i)-\Phi(x^*))-\xi(\Phi(x^j)-\Phi(x^*))$. From the concavity of $\xi$
 we get  
 $
 A_{k,k+1} \geq  \xi'\big(\Phi \lrpar{x^{k}}-\Phi (x^*)\big) \big(\Phi (x^{k})-\Phi (x^{k+1})\big)
$. 
 Together with \eqref{eq:KL2} and \eqref{eq:KL3} we get
  \begin{equation}
 \label{eq:KL4}
 \sum_{i=1}^m \frac{\eta_i^k}{2}\|x_i^{k+1} - x_i^k\|^2  \leq C  \sum_{i=1}^m  \frac{\eta_i^{k-1}}{2}\|x_i^{k} - x_i^{k-1}\|^2  + \frac{\varpi}{\sqrt 2} A_{k,k+1}(\|x^{k+1} - x^{k}\| + \|x^{k} - x^{k-1}\| )
 \end{equation}
Denote $\Upsilon^k=\sum_{i=1}^m \frac{\eta_i^k}{2}\|x_i^{k+1} - x_i^k\|^2$.  Using inequality $\sqrt{a + b }\leq \sqrt{a}+\sqrt{b}$ and $\sqrt{ab}\leq t a + b/4t$, for $t>0$, from \eqref{eq:KL4} we get
 $$\begin{array}{ll}
 \sqrt{\Upsilon^k} &\leq \sqrt{C \Upsilon^{k-1}} + \sqrt{\frac{\varpi A_{k,k+1}}{\sqrt 2}(\|x^{k+1} - x^{k}\| + \|x^{k} - x^{k-1}\| )}\\
 &\leq  \sqrt{C \Upsilon^{k-1}}  + \frac{(1-\sqrt{C})\sqrt{\underline l}}{3} (\|x^{k+1} - x^{k}\| + \|x^{k} - x^{k-1}\| )+ \frac{3\varpi A_{k,k+1}}{4\sqrt 2\sqrt{\underline l}(1-\sqrt{C})} 
 \end{array}$$
 Summing up this inequality from $k=\mathbf k +1 $ to $K$ we obtain 
 
 \begin{align*}
   \sqrt{\Upsilon^K}+\sum_{k=\mathbf k +1}^{K-1}(1-\sqrt{C}) \sqrt{\Upsilon^k} &\leq \sqrt{C \Upsilon^{\mathbf k}} + \frac{(1-\sqrt{C})\sqrt{\underline l}}{3} \sum_{k=\mathbf k + 1}^{K} (\|x^{k+1} - x^{k}\| + \|x^{k} - x^{k-1}\| ) \\
 & \qquad\qquad+ \frac{3\varpi}{4\sqrt 2 \sqrt{\underline{l}} (1-\sqrt{C})} A_{\mathbf k+1,K+1}.
  \end{align*}
On the other hand, we note that $\sqrt{ \Upsilon^k} \geq \sqrt{\underline{l}} \|x^{k+1} - x^{k}\|$. Therefore, we get 
   \begin{align*}
  \frac23(1-\sqrt{C})\sqrt{\underline l}\sum_{k=\mathbf k +1}^{K}\|x^{k+1} - x^{k}\|
\leq \frac{(1-\sqrt{C})\sqrt{\underline l}}{3} \sum_{k=\mathbf k +1 }^{K}\|x^{k} - x^{k-1}\| + \frac{3\varpi A_{\mathbf k + 1,K+1}}{4\sqrt 2 \sqrt{\underline{l}}(1-\sqrt{C})} ,
\end{align*} 
which implies that 
$
\sum_{k=\mathbf k +1}^{K}\|x^{k+1} - x^{k}\| \leq \|x^{\mathbf k+1}-x^{\mathbf k}\| + \frac{9\varpi}{4\sqrt 2(1-\sqrt{C})^2\underline l} A_{\mathbf k,K+1}.
$
Hence, $\sum_{k=1}^{ \infty}\|x^{k+1} - x^{k}\| < +\infty$. The result follows.

\subsection{Proof of Proposition~\ref{prop:sufficient_decrease_2}} 

\label{proof-essential}
Let us prove Statement (A). Statement (B) of Proposition~\ref{prop:sufficient_decrease_2} is a  consequence of Statement (A). 
From Inequality~\eqref{eq:condition_2} we get
{\small\begin{equation}
\label{eq:recursion_2}
F(\mbfx^{k,j})+ \frac{\bar\eta^{k,l-1}_i}{2}\|\barmx^{k,l}_i-\barmx^{k,l-1}_i\|^2
\leq F(\mbfx^{k,j-1})+ \frac{C \bar\eta^{k,l-2}_i}{2} \|\barmx^{k,l-1}_i-\barmx^{k,l-2}_i\|^2. 
\end{equation}}
Summing up Inequality~\eqref{eq:recursion_2} from $j=1$ to $T$ we obtain 
{\small
\begin{equation*}
F(\mbfx^{k+1})+ \sum_{i=1}^m\sum_{l=1}^{d_i^k}\frac{\bar\eta^{k,l-1}_i}{2}\|\barmx^{k,l}_i-\barmx^{k,l-1}_i\|^2
\leq
 F(\mbfx^{k})+  C\sum_{i=1}^m\sum_{l=1}^{d_i^k} \frac{\bar\eta^{k,l-2}_i}{2} \|\barmx^{k,l-1}_i-\barmx^{k,l-2}_i\|^2.
\end{equation*}}
Therefore,
{\small
\begin{equation}
\label{eq:recursion_4}
\begin{split}
&F(\mbfx^{k+1})+ C \sum_{i=1}^m \frac{\bar\eta_i^{k,d^k_i-1} }{2} \|\barmx^{k,d^k_i}_i-\barmx^{k,d^k_i-1}_i\|^2 + (1-C)\sum_{i=1}^m  \sum_{l=1}^{d_i^k} \frac{\bar\eta_i^{k,l-1}}{2} \|\barmx^{k,l}_i-\barmx^{k,l-1}_i\|^2
\\
&\qquad\leq 
F(\mbfx^{k})+ C\sum_{i=1}^m \frac{\bar\eta_i^{k,-1}}{2} \|\barmx^{k,0}_i-\barmx^{k,-1}_i\|^2.
\end{split}
\end{equation}}
Note that $ \barmx^{k,0}_i= \barmx^{k-1,d_i^{k-1}}_i$, $\barmx^{k,-1}_i=\barmx^{k-1,d_i^{k-1}-1}_i=(\mbfx^{k-1}_{prev})_i$ and $ \bar\eta_i^{k+1,-1}= \bar\eta_i^{k,d_i^{k}}$. Hence, from \eqref{eq:recursion_4} we obtain
{\small
\begin{equation}
\label{eq:recursion_5}
\begin{split}
&F(\mbfx^{k+1})+ C \sum_{i=1}^m \frac{\bar\eta_i^{k+1,-1}}{2} \|\mbfx^{k+1}_i-(\mbfx^{k+1}_{prev})_i\|^2 +  (1-C)\sum_{i=1}^m  \sum_{l=1}^{d_i^k} \frac{\bar\eta_i^{k,l-1} }{2}\|\barmx^{k,l}_i-\barmx^{k,l-1}_i\|^2
\\
&\qquad\leq F(\mbfx^{k})+ C\sum_{i=1}^m  \bar\eta_i^{k,-1} \|\mbfx^{k}_i-(\mbfx^{k}_{prev})_i\|^2.  
\end{split}
\end{equation}}
Summing up Inequality~\eqref{eq:recursion_5} from $k=0$ to $K-1$ we get
{\small
\begin{equation*}
\begin{array}{ll}
&F(\mbfx^{K})+ C\sum\limits_{i=1}^m \frac{\bar\eta_i^{K,-1}}{2} \|\mbfx^{K}_i-(\mbfx^{K}_{prev})_i\|^2 + (1-C)\sum_{k=0}^{K-1} \sum\limits_{i=1}^m  \sum\limits_{l=1}^{d_i^k} \frac{\bar\eta_i^{k,l-1}}{2} \|\barmx^{k,l}_i-\barmx^{k,l-1}_i\|^2
\\
&\qquad\leq F(\mbfx^{0})+ C\sum\limits_{i=1}^m  \frac{\bar\eta_i^{0,-1}}{2} \|\mbfx^{0}_i-(\mbfx^{0}_{prev})_i\|^2,  
\end{array}
\end{equation*}}
which gives the result.
\end{appendix}



\bibliography{iMM}

\begin{thebibliography}{59}
\providecommand{\natexlab}[1]{#1}
\providecommand{\url}[1]{\texttt{#1}}
\expandafter\ifx\csname urlstyle\endcsname\relax
  \providecommand{\doi}[1]{doi: #1}\else
  \providecommand{\doi}{doi: \begingroup \urlstyle{rm}\Url}\fi

\bibitem[Adly and Attouch(2020)]{Adly2020}
S.~Adly and H.~Attouch.
\newblock Finite convergence of proximal-gradient inertial algorithms combining
  dry friction with {H}essian-driven damping.
\newblock \emph{SIAM Journal on Optimization}, 30\penalty0 (3):\penalty0
  2134--2162, 2020.
\newblock \doi{10.1137/19M1307779}.
\newblock URL \url{https://doi.org/10.1137/19M1307779}.

\bibitem[Aharon et~al.(2006)Aharon, Elad, Bruckstein, et~al.]{aharon2006k}
M.~Aharon, M.~Elad, A.~Bruckstein, et~al.
\newblock K-svd: An algorithm for designing overcomplete dictionaries for
  sparse representation.
\newblock \emph{IEEE Transactions on signal processing}, 54\penalty0
  (11):\penalty0 4311, 2006.

\bibitem[Ahookhosh et~al.(2019)Ahookhosh, Hien, Gillis, and
  Patrinos]{ahookhosh2019multi}
M.~Ahookhosh, L.~T.~K. Hien, N.~Gillis, and P.~Patrinos.
\newblock Multi-block {B}regman proximal alternating linearized minimization
  and its application to sparse orthogonal nonnegative matrix factorization.
\newblock arXiv:1908.01402, 2019.

\bibitem[Ahookhosh et~al.(2020)Ahookhosh, Hien, Gillis, and
  Patrinos]{ahookhosh2020_inertial}
M.~Ahookhosh, L.~T.~K. Hien, N.~Gillis, and P.~Patrinos.
\newblock A block inertial {B}regman proximal algorithm for nonsmooth nonconvex
  problems with application to symmetric nonnegative matrix tri-factorization.
\newblock \emph{arXiv:2003.03963}, 2020.

\bibitem[Attouch and Bolte(2009)]{Attouch2009}
H.~Attouch and J.~Bolte.
\newblock On the convergence of the proximal algorithm for nonsmooth functions
  involving analytic features.
\newblock \emph{Mathematical Programming}, 116\penalty0 (1):\penalty0 5--16,
  Jan 2009.
\newblock ISSN 1436-4646.
\newblock \doi{10.1007/s10107-007-0133-5}.
\newblock URL \url{https://doi.org/10.1007/s10107-007-0133-5}.

\bibitem[Attouch et~al.(2010)Attouch, Bolte, Redont, and
  Soubeyran]{Attouch2010}
H.~Attouch, J.~Bolte, P.~Redont, and A.~Soubeyran.
\newblock Proximal alternating minimization and projection methods for
  nonconvex problems: An approach based on the {K}urdyka-{{\L}}ojasiewicz
  inequality.
\newblock \emph{Mathematics of Operations Research}, 35\penalty0 (2):\penalty0
  438--457, 2010.
\newblock \doi{10.1287/moor.1100.0449}.
\newblock URL \url{https://doi.org/10.1287/moor.1100.0449}.

\bibitem[Attouch et~al.(2013)Attouch, Bolte, and Svaiter]{Attouch2013}
H.~Attouch, J.~Bolte, and B.~F. Svaiter.
\newblock Convergence of descent methods for semi-algebraic and tame problems:
  proximal algorithms, forward--backward splitting, and regularized
  gauss--seidel methods.
\newblock \emph{Mathematical Programming}, 137\penalty0 (1):\penalty0 91--129,
  Feb 2013.

\bibitem[Bauschke et~al.(2017)Bauschke, Bolte, and Teboulle]{Bauschke2017}
H.~H. Bauschke, J.~Bolte, and M.~Teboulle.
\newblock A descent lemma beyond {L}ipschitz gradient continuity: First-order
  methods revisited and applications.
\newblock \emph{Mathematics of Operations Research}, 42\penalty0 (2):\penalty0
  330--348, 2017.
\newblock \doi{10.1287/moor.2016.0817}.
\newblock URL \url{https://doi.org/10.1287/moor.2016.0817}.

\bibitem[Beck and Tetruashvili(2013)]{Beck2013}
A.~Beck and L.~Tetruashvili.
\newblock On the convergence of block coordinate descent type methods.
\newblock \emph{SIAM Journal on Optimization}, 23:\penalty0 2037--2060, 2013.

\bibitem[Biswas et~al.(2006)Biswas, Lian, Wang, and Ye]{Biswas2006}
P.~Biswas, T.-C. Lian, T.-C. Wang, and Y.~Ye.
\newblock Semidefinite programming based algorithms for sensor network
  localization.
\newblock \emph{ACM Trans. Sen. Netw.}, 2\penalty0 (2):\penalty0 188--220,
  2006.

\bibitem[Blumensath and Davies(2009)]{BLUMENSATH2009}
T.~Blumensath and M.~E. Davies.
\newblock Iterative hard thresholding for compressed sensing.
\newblock \emph{Applied and Computational Harmonic Analysis}, 27\penalty0
  (3):\penalty0 265 -- 274, 2009.
\newblock ISSN 1063-5203.
\newblock \doi{https://doi.org/10.1016/j.acha.2009.04.002}.
\newblock URL
  \url{http://www.sciencedirect.com/science/article/pii/S1063520309000384}.

\bibitem[Bochnak et~al.(1998)Bochnak, Coste, and Roy]{Bochnak1998}
J.~Bochnak, M.~Coste, and M-F. Roy.
\newblock \emph{Real Algebraic Geometry}.
\newblock Springer, 1998.

\bibitem[Bolte et~al.(2007)Bolte, Daniilidis, and Lewis]{Bolte2007}
J.~Bolte, A.~Daniilidis, and A.~Lewis.
\newblock The {{\L}}ojasiewicz inequality for nonsmooth subanalytic functions
  with applications to subgradient dynamical systems.
\newblock \emph{SIAM Journal on Optimization}, 17\penalty0 (4):\penalty0
  1205--1223, 2007.
\newblock \doi{10.1137/050644641}.

\bibitem[Bolte et~al.(2014)Bolte, Sabach, and Teboulle]{Bolte2014}
J.~Bolte, S.~Sabach, and M.~Teboulle.
\newblock Proximal alternating linearized minimization for nonconvex and
  nonsmooth problems.
\newblock \emph{Mathematical Programming}, 146\penalty0 (1):\penalty0 459--494,
  Aug 2014.

\bibitem[Bradley and Mangasarian(1998)]{brafea}
P.~S. Bradley and O.~L. Mangasarian.
\newblock Feature selection via concave minimization and support vector
  machines.
\newblock In \emph{Proceeding of international conference on machine learning
  ICML'98}, 1998.

\bibitem[Chouzenoux et~al.(2016)Chouzenoux, Pesquet, and Repetti]{Emilie2016}
E.~Chouzenoux, J.-C. Pesquet, and A.~Repetti.
\newblock A block coordinate variable metric forward–backward algorithm.
\newblock \emph{Journal of Global Optimization}, 66:\penalty0 457–485, 2016.

\bibitem[Corless et~al.(1996)Corless, Gonnet, Hare, Jeffrey, and
  Knuth]{Corless_1996}
R.~M. Corless, G.~H. Gonnet, D.~E.~G. Hare, D.~J. Jeffrey, and D.~E. Knuth.
\newblock On the lambertw function.
\newblock \emph{Advances in Computational Mathematics}, 5, 1996.

\bibitem[Dacrema et~al.(2019)Dacrema, Cremonesi, and Jannach]{dacrema2019we}
M.~F. Dacrema, P.~Cremonesi, and D.~Jannach.
\newblock Are we really making much progress? a worrying analysis of recent
  neural recommendation approaches.
\newblock In \emph{Proceedings of the 13th ACM Conference on Recommender
  Systems}, pages 101--109, 2019.

\bibitem[Fan and Li(2001)]{fanvar}
J.~Fan and R.~Li.
\newblock Variable selection via nonconcave penalized likelihood and its oracle
  properties.
\newblock \emph{J. Amer. Stat. Ass.}, 96\penalty0 (456):\penalty0 1348--1360,
  2001.

\bibitem[Gillis(2020)]{Gillis2020}
N.~Gillis.
\newblock \emph{Nonnegative Matrix Factorization}.
\newblock SIAM, Philadelphia, 2020.
\newblock \doi{10.1137/1.9781611976410}.

\bibitem[Gillis and Glineur(2012)]{Gillis2012}
N.~Gillis and F.~Glineur.
\newblock Accelerated multiplicative updates and hierarchical als algorithms
  for nonnegative matrix factorization.
\newblock \emph{Neural Computation}, 24\penalty0 (4):\penalty0 1085--1105,
  2012.

\bibitem[Grippo and Sciandrone(2000)]{GRIPPO20001}
L.~Grippo and M.~Sciandrone.
\newblock On the convergence of the block nonlinear gauss--seidel method under
  convex constraints.
\newblock \emph{Operations Research Letters}, 26\penalty0 (3):\penalty0 127 --
  136, 2000.
\newblock ISSN 0167-6377.
\newblock \doi{https://doi.org/10.1016/S0167-6377(99)00074-7}.
\newblock URL
  \url{http://www.sciencedirect.com/science/article/pii/S0167637799000747}.

\bibitem[Halko et~al.(2011)Halko, Martinsson, and Tropp]{Halko2011}
N.~Halko, P.~G. Martinsson, and J.~A. Tropp.
\newblock Finding structure with randomness: Probabilistic algorithms for
  constructing approximate matrix decompositions.
\newblock \emph{SIAM Review}, 53\penalty0 (2):\penalty0 217--288, 2011.

\bibitem[Hien and Gillis(2021)]{HienNicolas_KLNMF}
L.~T.~K. Hien and N.~Gillis.
\newblock Algorithms for nonnegative matrix factorization with the
  {Kullback-Leibler} divergence.
\newblock \emph{Journal of Scientific Computing}, \penalty0 (87):\penalty0 93,
  2021.

\bibitem[Hien et~al.(2020)Hien, Gillis, and Patrinos]{Hien_ICML2020}
L.~T.~K. Hien, N.~Gillis, and P.~Patrinos.
\newblock Inertial block proximal method for non-convex non-smooth
  optimization.
\newblock In \emph{Thirty-seventh International Conference on Machine Learning
  (ICML)}, 2020.

\bibitem[Hien et~al.(2022)Hien, Phan, and Gillis]{HPNiADMM2021}
L.~T.~K. Hien, D.~N. Phan, and N.~Gillis.
\newblock Inertial alternating direction method of multipliers for non-convex
  non-smooth optimization.
\newblock \emph{Computational Optimization and Applications}, 2022.

\bibitem[Hildreth(1957)]{Hildreth}
C.~Hildreth.
\newblock A quadratic programming procedure.
\newblock \emph{Naval Research Logistics Quarterly}, 4\penalty0 (1):\penalty0
  79--85, 1957.
\newblock \doi{10.1002/nav.3800040113}.
\newblock URL
  \url{https://onlinelibrary.wiley.com/doi/abs/10.1002/nav.3800040113}.

\bibitem[Hong et~al.(2017)Hong, Wang, Razaviyayn, and
  Luo]{Hong2017_BSUMcomplexity}
M.~Hong, X.~Wang, M.~Razaviyayn, and Z.-Q. Luo.
\newblock Iteration complexity analysis of block coordinate descent methods.
\newblock \emph{Mathematical Programming}, 163:\penalty0 85–114, 2017.

\bibitem[Kim and Leskovec(2011)]{Kim2011}
M.~Kim and J.~Leskovec.
\newblock The network completion problem: Inferring missing nodes and edges in
  networks.
\newblock In \emph{Proceedings of the 11th International Conference on Data
  Mining}, pages 47--58, 2011.

\bibitem[{Koren} et~al.(2009){Koren}, {Bell}, and {Volinsky}]{kormat}
Y.~{Koren}, R.~{Bell}, and C.~{Volinsky}.
\newblock Matrix factorization techniques for recommender systems.
\newblock \emph{Computer}, 42\penalty0 (8):\penalty0 30--37, 2009.

\bibitem[Kurdyka(1998)]{Kurdyka1998}
K.~Kurdyka.
\newblock On gradients of functions definable in o-minimal structures.
\newblock \emph{Annales de l'Institut Fourier}, 48\penalty0 (3):\penalty0
  769--783, 1998.
\newblock \doi{10.5802/aif.1638}.
\newblock URL \url{http://http://www.numdam.org/item/AIF_1998__48_3_769_0}.

\bibitem[Latafat et~al.(2022)Latafat, Themelis, and Patrinos]{Latafat2022}
P.~Latafat, A.~Themelis, and P.~Patrinos.
\newblock Block-coordinate and incremental aggregated proximal gradient methods
  for nonsmooth nonconvex problems.
\newblock \emph{Mathematical Programming}, 193:\penalty0 195–224, 2022.

\bibitem[{Liu} et~al.(2013){Liu}, {Lin}, {Yan}, {Sun}, {Yu}, and {Ma}]{Lui2013}
G.~{Liu}, Z.~{Lin}, S.~{Yan}, J.~{Sun}, Y.~{Yu}, and Y.~{Ma}.
\newblock Robust recovery of subspace structures by low-rank representation.
\newblock \emph{IEEE Transactions on Pattern Analysis and Machine
  Intelligence}, 35\penalty0 (1):\penalty0 171--184, 2013.

\bibitem[Lu et~al.(2018)Lu, Freund, and Nesterov]{lu2018relatively}
H.~Lu, R.~M. Freund, and Y.~Nesterov.
\newblock Relatively smooth convex optimization by first-order methods, and
  applications.
\newblock \emph{SIAM Journal on Optimization}, 28\penalty0 (1):\penalty0
  333--354, 2018.

\bibitem[Mairal(2013)]{Mairal_ICML13}
J.~Mairal.
\newblock Optimization with first-order surrogate functions.
\newblock In \emph{Proceedings of the 30th International Conference on
  International Conference on Machine Learning - Volume 28}, ICML’13, pages
  783--791. JMLR.org, 2013.

\bibitem[Natarajan(1995)]{Natarajan1995}
B.~Natarajan.
\newblock Sparse approximate solutions to linear systems.
\newblock \emph{SIAM Journal on Computing}, 24\penalty0 (2):\penalty0 227--234,
  1995.
\newblock \doi{10.1137/S0097539792240406}.
\newblock URL \url{https://doi.org/10.1137/S0097539792240406}.

\bibitem[Nesterov(1983)]{Nesterov1983}
Y.~Nesterov.
\newblock A method of solving a convex programming problem with convergence
  rate {O}$(1/k^2)$.
\newblock \emph{Soviet Mathematics Doklady}, 27\penalty0 (2), 1983.

\bibitem[Nesterov(1998)]{Nesterov1998}
Y.~Nesterov.
\newblock On an approach to the construction of optimal methods of minimization
  of smooth convex functions.
\newblock \emph{Ekonom. i. Mat. Metody}, 24:\penalty0 509--517, 1998.

\bibitem[Nesterov(2004)]{Nesterov2004}
Y.~Nesterov.
\newblock \emph{Introductory lectures on convex optimization: A basic course}.
\newblock Kluwer Academic Publ., 2004.

\bibitem[Nesterov(2005)]{Nesterov2005}
Yu. Nesterov.
\newblock Smooth minimization of non-smooth functions.
\newblock \emph{Math. Prog.}, 103\penalty0 (1):\penalty0 127--152, 2005.

\bibitem[Ochs(2019)]{Ochs2019}
P.~Ochs.
\newblock Unifying abstract inexact convergence theorems and block coordinate
  variable metric ipiano.
\newblock \emph{SIAM Journal on Optimization}, 29\penalty0 (1):\penalty0
  541--570, 2019.
\newblock \doi{10.1137/17M1124085}.
\newblock URL \url{https://doi.org/10.1137/17M1124085}.

\bibitem[Ochs et~al.(2014)Ochs, Chen, Brox, and Pock]{Ochs2014}
P.~Ochs, Y.~Chen, T.~Brox, and T.~Pock.
\newblock i{P}iano: Inertial proximal algorithm for nonconvex optimization.
\newblock \emph{SIAM Journal on Imaging Sciences}, 7\penalty0 (2):\penalty0
  1388--1419, 2014.
\newblock \doi{10.1137/130942954}.
\newblock URL \url{https://doi.org/10.1137/130942954}.

\bibitem[Peharz and Pernkopf(2012)]{Peharz2012}
R.~Peharz and F.~Pernkopf.
\newblock Sparse nonnegative matrix factorization with ℓ0-constraints.
\newblock \emph{Neurocomputing}, 80:\penalty0 38 -- 46, 2012.
\newblock ISSN 0925-2312.
\newblock \doi{https://doi.org/10.1016/j.neucom.2011.09.024}.
\newblock URL
  \url{http://www.sciencedirect.com/science/article/pii/S0925231211006370}.
\newblock Special Issue on Machine Learning for Signal Processing 2010.

\bibitem[Phan and {Le Thi}(2019)]{phagro}
D.~N. Phan and H.~A. {Le Thi}.
\newblock Group variable selection via $\ell_{p,0}$ regularization and
  application to optimal scoring.
\newblock \emph{Neural Networks}, 118:\penalty0 220 -- 234, 2019.

\bibitem[Pock and Sabach(2016)]{Pock2016}
T.~Pock and S.~Sabach.
\newblock Inertial proximal alternating linearized minimization (i{PALM}) for
  nonconvex and nonsmooth problems.
\newblock \emph{SIAM Journal on Imaging Sciences}, 9\penalty0 (4):\penalty0
  1756--1787, 2016.
\newblock \doi{10.1137/16M1064064}.
\newblock URL \url{https://doi.org/10.1137/16M1064064}.

\bibitem[Polyak(1964)]{POLYAK1964}
B.T. Polyak.
\newblock Some methods of speeding up the convergence of iteration methods.
\newblock \emph{USSR Computational Mathematics and Mathematical Physics},
  4\penalty0 (5):\penalty0 1 -- 17, 1964.
\newblock ISSN 0041-5553.
\newblock \doi{https://doi.org/10.1016/0041-5553(64)90137-5}.
\newblock URL
  \url{http://www.sciencedirect.com/science/article/pii/0041555364901375}.

\bibitem[Powell(1973)]{Powell1973}
M.~J.~D. Powell.
\newblock On search directions for minimization algorithms.
\newblock \emph{Mathematical Programming}, 4\penalty0 (1):\penalty0 193--201,
  Dec 1973.
\newblock ISSN 1436-4646.

\bibitem[Razaviyayn et~al.(2013)Razaviyayn, Hong, and Luo]{Razaviyayn2013}
M.~Razaviyayn, M.~Hong, and Z.~Luo.
\newblock A unified convergence analysis of block successive minimization
  methods for nonsmooth optimization.
\newblock \emph{SIAM Journal on Optimization}, 23\penalty0 (2):\penalty0
  1126--1153, 2013.
\newblock \doi{10.1137/120891009}.

\bibitem[Rendle et~al.(2019)Rendle, Zhang, and Koren]{rendle2019difficulty}
S.~Rendle, L.~Zhang, and Y.~Koren.
\newblock On the difficulty of evaluating baselines: A study on recommender
  systems.
\newblock \emph{arXiv preprint arXiv:1905.01395}, 2019.

\bibitem[Rockafellar and Wets(1998)]{RockWets98}
R.~Tyrrell Rockafellar and Roger J.-B. Wets.
\newblock \emph{Variational Analysis}.
\newblock Springer Verlag, Heidelberg, Berlin, New York, 1998.

\bibitem[Teboulle and Vaisbourd(2020)]{Teboulle2020}
M.~Teboulle and Y.~Vaisbourd.
\newblock Novel proximal gradient methods for nonnegative matrix factorization
  with sparsity constraints.
\newblock \emph{SIAM Journal on Imaging Sciences}, 13\penalty0 (1):\penalty0
  381--421, 2020.
\newblock \doi{10.1137/19M1271750}.
\newblock URL \url{https://doi.org/10.1137/19M1271750}.

\bibitem[Tseng(2001)]{Tseng2001}
P.~Tseng.
\newblock Convergence of a block coordinate descent method for
  nondifferentiable minimization.
\newblock \emph{Journal of Optimization Theory and Applications}, 109\penalty0
  (3):\penalty0 475--494, Jun 2001.

\bibitem[Tseng and Yun(2009)]{Tseng2009}
P.~Tseng and S.~Yun.
\newblock A coordinate gradient descent method for nonsmooth separable
  minimization.
\newblock \emph{Mathematical Programming}, 117\penalty0 (1):\penalty0 387--423,
  Mar 2009.

\bibitem[Xu and Yin(2013)]{Xu2013}
Y.~Xu and W.~Yin.
\newblock A block coordinate descent method for regularized multiconvex
  optimization with applications to nonnegative tensor factorization and
  completion.
\newblock \emph{SIAM Journal on Imaging Sciences}, 6\penalty0 (3):\penalty0
  1758--1789, 2013.
\newblock \doi{10.1137/120887795}.
\newblock URL \url{https://doi.org/10.1137/120887795}.

\bibitem[Xu and Yin(2016)]{XuYin2016}
Y.~Xu and W.~Yin.
\newblock A fast patch-dictionary method for whole image recovery.
\newblock \emph{Inverse Problems \& Imaging}, 10:\penalty0 563, 2016.
\newblock ISSN 1930-8337.
\newblock \doi{10.3934/ipi.2016012}.

\bibitem[Xu and Yin(2017)]{Xu2017}
Y.~Xu and W.~Yin.
\newblock A globally convergent algorithm for nonconvex optimization based on
  block coordinate update.
\newblock \emph{Journal of Scientific Computing}, 72\penalty0 (2):\penalty0
  700--734, Aug 2017.

\bibitem[Zavriev and Kostyuk(1993)]{Zavriev1993}
S.K. Zavriev and F.V. Kostyuk.
\newblock Heavy-ball method in nonconvex optimization problems.
\newblock \emph{Computational Mathematics and Modeling}, 1993.

\bibitem[Zhong and Ghosh(2005)]{ZG05}
S.~Zhong and J.~Ghosh.
\newblock Generative model-based document clustering: a comparative study.
\newblock \emph{Knowledge and Information Systems}, 8\penalty0 (3):\penalty0
  374--384, 2005.

\bibitem[Zhou(2018)]{Zhou2018misc}
Xingyu Zhou.
\newblock On the fenchel duality between strong convexity and lipschitz
  continuous gradient, 2018.
\newblock arXiv:1803.06573.

\end{thebibliography}

\end{document}